\documentclass[11pt]{amsart}
\usepackage[a4paper,margin=3cm]{geometry}
\usepackage[utf8]{inputenc} 
\usepackage[T1]{fontenc} 
\usepackage{lmodern}
\usepackage{todonotes}
 \usepackage{dsfont}

\usepackage[autolanguage]{numprint} 
\usepackage{graphicx} 
\usepackage[all]{xy}
\usepackage{xcolor,cancel}
\usepackage{amsmath,amssymb,amsthm}
\usepackage[colorlinks=true]{hyperref}
\usepackage[normalem]{ ulem }
\newcommand{\bdelta}{\overline{\delta}}

\newcommand\veps{\varepsilon}
\renewcommand\epsilon{\varepsilon}
\newcommand\NN{\mathbb{N}}

\newcommand\RR{\mathbb{R}} 

\newcommand\PP{\mathbb{P}} 
\newcommand\EE{\mathbb{E}}
\newcommand\Car{\mathds{1}}
\newcommand\Var{\mathrm{Var}}

\newcommand{\tr}{\mathrm{tr}}

\newtheorem{theorem}{Theorem}[section]
\newtheorem{corollary}[theorem]{Corollary}
\newtheorem{proposition}[theorem]{Proposition}
\newtheorem{lemma}[theorem]{Lemma}

\newtheorem{fact}[theorem]{Fact}

\newtheorem{definition}[theorem]{Definition}

\newtheorem{remark}[theorem]{Remark}

\newtheorem{assumption}[theorem]{Assumption}

\numberwithin{equation}{section}
\newcommand{\Jn}{J_{n,{\bf a},{\bf b}}}

\title[CLT for characteristic polynomial of random Jacobi matrices]{A CLT for the characteristic polynomial of random Jacobi matrices, and the
G$\beta$E}

\author[F.\ Augeri]{Fanny Augeri$^{\mathsection}$}
\address{$^{\mathsection}$Department of Mathematics, Weizmann Institute of Science, 
POB 26, Rehovot 76100,\newline \indent  Israel. 
Current address: Laboratoire de Probabilit\'{e}s, Statistique et Mod\'{e}lisation (LPSM), \newline \indent
Universit\'{e} de Paris, 75205 Paris Cedex 13, France}
\email{augeri@lpsm.paris}
\author[R.\ Butez]{Raphael Butez$^*$}
\address{$^*$Department of Mathematics, Weizmann Institute of Science 
POB 26, Rehovot 76100,\newline \indent
Israel. Current address: D\'{e}partement Math\'{e}matiques, Universit\'{e} de Lille, 59655 \newline \indent Villeneuve d'Ascq Cedex, France}
\email{butez@ceremade.dauphine.fr}
\author[O.\ Zeitouni]{Ofer Zeitouni$^{\#}$}
\address{$^{\#}$Department of Mathematics, Weizmann Institute of Science, 
 POB 26, Rehovot 76100, \newline\indent Israel, and
  Courant Institute, New York University,
  251 Mercer St, New York, NY 10012, \newline \indent
  USA}
 \email{ofer.zeitouni@weizmann.ac.il}

\date{}
\begin{document}
\begin{abstract}
  We prove a central limit theorem for the real part of the logarithm 
   of the 
  characteristic polynomial
  of  random Jacobi matrices. Our results cover the G$\beta$E models for $\beta>0$.
\end{abstract}
\maketitle
\section{Introduction}
We derive in this paper a central limit theorem for the real part
of the logarithm of the
characteristic polynomial of certain Jacobi matrices; as discussed below,
see Remark \ref{rem-GbetaE},
the ensemble we consider covers the important cases of the G$\beta$E ensembles.
We begin by introducing the model.

The $n$-by-$n$ \textit{Jacobi} matrix associated with 
sequences ${\bf a}=(a_k)_{k=1}^{n-1}$, ${\bf b}=(b_k)_{k=1}^n$ of real numbers
is defined as the matrix
\begin{equation}
  \label{eq-Jacobi}
  J_{n,{\bf a},{\bf b}}=\begin{pmatrix}b_n& a_{n-1}&0&\cdots&0\\
    a_{n-1}&b_{n-1}&a_{n-2}&\cdots&0\\
    \vdots&\ddots& \ddots&\ddots&\vdots\\
    0&\cdots&a_{2}&b_{2}&a_1\\
    0&\cdots&\cdots&a_1& b_1
  \end{pmatrix}.
\end{equation}
We will be interested in the spectrum of 
\textit{random} Jacobi matrices, whose sequences 
of coefficients
${\bf a}, {\bf b}$ satisfy the following.
\begin{assumption}
  \label{ass-ab}
With $v=2/\beta>0$,
$\{b_k\}_{k\geq 1}$, $\{g_k\}_{k\geq 1}$ are
two independent sequences
of independent centered random variables,
with $b_k,g_k$ of variance $v+O(1/k)$,
possessing absolutely continuous laws with respect to the Lebesgue measure on
$\RR$.  Let $a_{k-1}$ be such that, for a deterministic sequence
$c_k$ satisfying $c_k=O(1/k)$ and $c_{k+1}-c_k = O(1/k^2)$,
$$ \frac{a_{k-1}^2}{\sqrt{k(k-1)}} = 1-
 c_k+
\frac{g_k}{\sqrt{k}}.$$
Further,
 there exist
$\lambda_0>0$ and $M>0$ independent of $k$ such that
\begin{equation} \label{boundlaplace} \ \EE \big(e^{\lambda_0 |b_k| }\big)\leq M, \ \EE \big(e^{\lambda_0 |g_k|}\big)\leq M.\end{equation}
\end{assumption}
Note that the assumption on the sequence $a_k^2$ entails an implicit further assumption on the variables $g_k$; further note that changing the
sign of any coefficient $a_k$ does  
not change 
the spectrum of  $J_{n,{\bf a},{\bf b}}$.

For a Jacobi matrix $J_{n,{\bf a},{\bf b}}$ with eigenvalues $\lambda_i$, $i=1,\ldots,n$,
we introduce the empirical measure
of eigenvalues $L_n=n^{-1}\sum_{i=1}^n \delta_{\lambda_i/\sqrt{n}}$. 
The following is classical, and can be read e.g. from \cite{popescu}.
\begin{proposition}
  \label{prop-LLN}
  Assume ${\bf a},{\bf b}$ satisfy Assumption \ref{ass-ab}. Then, the
  empirical measure $L_n$ associated with $J_{n,{\bf a},{\bf b}}/\sqrt{n}$
  converges weakly in probability to the semi-circle law
  with density
  \[ \sigma(x)=\frac{1}{2\pi}\sqrt{4-x^2} {\bf 1}_{|x|\leq 2}.\]
\end{proposition}
Let $D_n(z,\Jn)=\det(zI_n-\Jn/\sqrt{n})$ denote the characteristic 
polynomial of $\Jn/\sqrt{n}$. 
Our main result is the following central limit theorem for $\log D_n(z,\Jn)$.
\begin{theorem}
  \label{theo-main}
  Fix $z\in (-2,2)\setminus \{0\}$.
  Then, the sequence of random variables
  \begin{equation}
    \label{eq-defwn}
    w_n(z)= \frac{\log |D_n(z,\Jn)|- n(z^2/4-1/2)}{\sqrt{v \log n/2}}
  \end{equation}
  converges in distribution to a standard Gaussian law.
\end{theorem}
The centering $n(z^2/4-1/2)$, which is the $n$ multiple of the logarithmic potential of the semi-circle distribution, comes about naturally from a certain scaling captured by the deterministic constant  $C_n(z)$ defined in
\eqref{eq-Cn} below, see Remark \ref{rem-asym} for the explicit computation.
\begin{remark}[The G$\beta$E ensemble]
  \label{rem-GbetaE}
  For $\beta>0$, 
  the classical G$\beta$E ensemble is the law on real 
  $n$-particle configurations
  given by the density
  \begin{equation}
    \label{eq-GBE}
     p(\lambda_1,\ldots,\lambda_n)= C_{n,\beta} |\Delta(\lambda)|^{\beta}
  e^{-\frac{\beta}{4}\lambda_i^2},
  \quad \mbox{\rm where}\; \Delta(\lambda)=\prod_{1\leq 
  i<j\leq n}(\lambda_i-\lambda_j).
\end{equation}
  We recall, see e.g. \cite[Section 4.1]{AGZ}, that
the particular cases of $\beta=1,2,4$ can be realized as the law of the
eigenvalues 
of (Orthogonal, Hermitian, Symplectic) random matrices with Gaussian entries.
  We also recall, see e.g. \cite{DumitriuEdelman}, \cite[Section 4.5]{AGZ},
  that by the Dumitriu-Edelman construction, these laws can be realized
  as the laws of eigenvalues of Jacobi matrices, where ${\bf a},{\bf b}$
  independent sequences of independent variables, such that $\sqrt{\beta}
  a_i$ is
  distributed as a $\chi$ random variable with $i\beta$ degrees of freedom
  and $b_i$ is a centered 
 Gaussian of variance $v=2/\beta$.
 It is immediate to check, using the CLT for $\chi^2$ random variables,
 that this setup is covered by Assumption \ref{ass-ab}.
 \end{remark}

   We note that the law of the determinant of $\Jn$, that is
   the case $z=0$, is not covered 
   by Theorem \ref{theo-main}. 
We chose to state the theorem that way for two reasons. First, the case 
$z=0$ is covered, at least for G$\beta$E with $\beta=1,2$, in 
the work of Tao and Vu \cite{TaoVu}, and it is not hard to see that for $z=0$,
their proof could be extended to the case of general $\beta$ and in fact to cover
our assumption \ref{ass-ab}. Second,
   as we describe later in section \ref{sec-structure}, the case $z=0$
   is in fact simpler than $z\neq 0$, in two ways: first, there is no linearized regime, as in Section \ref{sec-scalar}. Second, the rotation angles $\theta_l$,
   see \eqref{eq-lambdadef},
   are, in case $z=0$, constant and equal to $\pi$, which makes the definition of blocks redundant and greatly simplifies the argument. 
   Our proof could be adapted to that case as well, with minor changes.

   As described in Section \ref{sec-previous}, one expects 
   $\sqrt{\log n}w_n(z)$ to converge (as a process in $z$) 
   to a logarithmically correlated Gaussian field, in the sense that
   $ \log n \mbox{\rm Cov}( w_n(z), w_n(z')) =-(v/2)\log (|z-z'|\vee 1/n)+O(1)$ for 
   $|z|,|z'|<2$. Theorem \ref{theo-main}
   can thus be seen as a first step in proving such convergence, and indeed
   we expect the ``blocks'' construction to play an important role in 
   such a proof. The remaining obstacles to such proof are significant 
   enough that we do not address this goal here, but we do note that
   in the G$\beta$E setup,
    such a statement follows
   from the results of \cite{BourgadeModyPain}.
   \subsection{Related work}
  \label{sec-previous}
The computation of the determinant of a random matrix has a long history.
 Goodman \cite{goodman} noticed that for
 random matrices with iid standard complex Gaussian entries, the 
 square of the determinant can be expressed as a product of $\chi^2$ variables,
 from which a central limit theorem is easy to derive. The latter CLT
 was generalized 
 in \cite{Nguyen-Vu} to non-Gaussian i.i.d. entries. These proofs
 do not generalize well to Hermitian 
 matrices due to the lack of independence of entries.
 This motivated Tao and Vu \cite{TaoVu} to consider the GOE/GUE cases.
 They use the Edelman-Dumitriu Jacobi representation and the recursions
 for the determinant of principal minors to derive a CLT for the 
 logarithm of the determinant, that is, for $w_n(0)$. They also showed 
 that the result is universal, at least for Hermitian matrices whose
 entries on and above the diagonal 
 are independent and match the GUE moments up to fourth moment. 
 (The fourth moment matching condition 
 was later  dispensed with in \cite{BourgadeMody}.) Earlier 
   works that give the distribution of the determinant, that is at $z=0$, are
   \cite{MN98} and \cite{DLC00}, for the cases $\beta=2$ and $\beta=1$, respectively.
   In the proof of the GOE/GUE
 cases,
 the recursions \eqref{basic-recursion} below play  a central role,
 together with the observation that 
 $\psi_{k+2}^2+\psi_{k+1}^2= \psi_k^2 +\psi_{k-1}^2+\epsilon_k$,
  with $\epsilon_k$ small (we note that this relation is a manifestation
  of the fact that $\psi_{k+1}=|\psi_k|+\epsilon_k$ with $\epsilon_k$ small).
 That is, the case $z=0$ is a case where the
 angles $\theta_l$
 defined in \eqref{eq-lambdadef} below are roughly equal to $\pi$, and
 in particular do not change significantly with $l$. This fact simplifies significantly the analysis compared to the current paper.

 The interest in the understanding of $w_n(z)$ for $z\neq 0$ stems in part
 from analogous results for the characteristic polynomial of C$\beta$E matrices, which in the case $\beta=2$ corresponds to
 random unitary matrices. 
 In that case, it is natural to consider 
 $z\in S^1$, and the rotation invariance of the ensemble
 implies that the law of of the charateristic polynomial
 in independent of $z\in S^1$. 
Using representation-theoretic tools, 
it was proved by Diaconis and Shashahani \cite{DS}
that for $\beta=2$,
the moments $m_{k,n}:=n\int x^k L_n(dx)$  of the empirical 
 measure of eigenvalues are asymptotically Gaussian; this 
 was amplified in \cite{DiaconisEvans}, where the statistics of linear functionals of $nL_n$ was shown to be Gaussian, 
 with a covariance corresponding to a 
 logarithmically correlated field. 
 This was  further developed  in \cite{Wieand}.
 Motivated by conjectured connections with the zeros 
 of the Riemann zeta function \cite{KeatingSnaith}, 
 the CLT for the logarithmic 
 characteristic polynomial of a CUE matrix was first derived 
 in \cite{HughesKeatingOconnell}. These authors also considered the
 process $\{m_n(z)\}_{z\in S^1}$, where $m_n(z)=\log \det(zI-U_n)$, and 
 proved the logarithmic-correlated structure of the limiting Gaussian field. 

 The recent explosion of interest in logarithmically correlated fields 
 has led to a resurgence of activity concerning $m_n(z)$ for CUE matrices.
 Motivated again by connections with Riemann's zeta function, 
 Fyodorov, Hiary and Keating \cite{FHK} conjectured  the (and provided
 convincing evidence for) limiting distribution of the maximum of $\Re m_n(z)$.
 Recent work \cite{ABB,PZ,CMN,CN} has made significant progress toward that conjecture. In particular, \cite{CN} have shown, in the general 
 C$\beta$E case, that with $\Phi_n^*$ the orthogonal polynomial 
 associated with the CMV representation of $U_n$,
 $|\Phi_n^*|^{-2}$, 
 properly normalized, converges
 to the standard Gaussian Multiplicative chaos on the circle. Interestingly,
 the proofs in \cite{CMN,CN} are based on recursions for Verblunsky coefficients, which play a role similar to the role of the coefficients $(a_k,b_k)$.
 For $\beta=2$ and using Riemann-Hilbert methods,
 convergence of the characteristic polynomial to the GMC in the full subcritical phase was proved in \cite{NSW}, after initial results in the so-called $L^2$ phase 
 \cite{Webb}. See also \cite{L21} for a smoothed analogue, valid for all $\beta$.

 One expects that results true for the logarithm of the characteristic polynomial of C$\beta$E should carry over to the G$\beta$E setup.
Indeed, for $\beta=2$, Riemann-Hilbert methods can be applied to yield 
convergence of exponential moments, which in turn  
imply the (multidimensional, i.e. for a finite collection of $z_i$s) 
CLT, see 
\cite{Krasovsky}.
Generalizations to more general 
potentials beyond the quadratic case (that is, replace $\lambda_i^2$ in
\eqref{eq-GBE} by $V(\lambda_i)$) appear
in \cite{BWW} and \cite{Charlier}.
We note that 
\cite{BWW} prove convergence 
to the GMC of the characteristic polynomial
(in the $L^2$ phase).  We also mention \cite{FKS} for an explicit construction
of the limiting logarithmically correlated field for GUE, and a proof of convergence in mesoscopical scales.

 Our work can be seen as a further 
 step in the direction of obtaining CLTs for random determinants,
 with generalization in the direction 
 where Riemann-Hilbert methods are not applicable. 
 Recent related results 
 are \cite{LP}, where a (much more precise) study of the recursion 
 for $\psi_k$ is undertaken, albeit only for $k<k_0(z)-O(n^{1/3} \log n)$,
 i.e. in what we call the \textit{scalar} regime (and they refer to as the hyperbolic regime), and the follow-up article \cite{LP2} showing that the scaled characteristic polynomial at the edge of the spectrum converges in law to 
 the solution of the stochastic Airy equation. Another closely related study is \cite{CFLW}, where 
 the exponential of the eigenvalue counting measure of the GUE
 (and other
 ensembles related to $\beta=2$) 
 is considered. 

 We have emphasized in this brief summary the fluctuations of the logarithm of the characteristic polynomial, which can be seen as a linear statistics (albeit, a non-smooth one) 
 of the empirical measure $L_N$. For smooth statistics, the fluctuations are Gaussian, and well understood. We refer in particular to 
 \cite{DE2} (for the G$\beta$E case; their proofs cover also our setup) and \cite{popescu}, and to \cite{DiaconisEvans} for the CUE case.

 After the first version of this
 work was completed, we learnt of a very interesting work of 
 Bourgade, Mody and Pain \cite{BourgadeModyPain},
 where they obtain local laws for the distribution
of eigenvalues for the G$\beta$E, and for more general $\beta$-ensembles. As
a consequence, they obtain a CLT for the logarithm of the characteristic polynomial. Compared with our results in this paper, they are able to treat
more general potentials, 
they obtain a multi-dimensional CLT that reveals the log-correlated structure of the limiting field, and they treat also the imaginary part of the logarithm. On the other hand, 
our work allows us to handle models outside of the $\beta$ class of models.
The methods of proof in the two papers are completely different.

Another interesting paper that was completed after the first version of this paper was posted is \cite{JKOP}. The authors derive there the CLT at $z$ in a small, $N$-dependent neighborhood of the edge of the support of the limiting ESD, for Gaussian ensembles ($\beta=1,2$), including spiked models, by exploiting \eqref{eq-recpsik}, in the scalar regime.

\subsection{Notation}
The collection of probability measures on $\RR$ is denoted $\mathcal{P}(\RR)$.
We equip $\mathcal{P}(\RR)$ with the weak topology and write $\mu_n{\leadsto}\mu$ for weak
convergence of $\mu_n$ to $\mu\in \mathcal{P}(\RR)$; we let $d:\mathcal{P}(\RR)\times \mathcal{P}(\RR)\to \RR$ denote a metric compatible
with weak convergence (e.g., the L\'{e}vy metric).
We let $\gamma\in \mathcal{P}(\RR)$ denote the standard Gaussian law.
For a sigma-algebra $\mathcal{F}$, we write $\PP_{\mathcal{F}},
\EE_{\mathcal F}$ for the
conditional probability
and expectation conditioned on $\mathcal{F}$.
We write non-commutative products from right to left, that is, if $C_i, i\leq j$ is a sequence of square matrices of a given size,
$$ \prod_{i=1}^j C_i = C_j C_{j-1} \ldots C_1.$$
For any matrix $A$ we denote by $A^T$ its transpose and by $\|A\|$ its Hilbert-Schmidt norm, that is,
$ \| A\| = \big( \tr (A A^*)\big)^{1/2},$
and define $\|A\|_{\infty} = \max_{i,j} |A_{i,j}|$.
We write $I_n$ for the identity matrix of dimension $n$.
We denote by $\mathcal{M}_2$ the set of matrices of size $2\times 2$ with complex entries. We use the notation $\lesssim$ for inequalities which are valid up to an
absolute  multiplicative constant, independent of all parameters $n,k$, etc.
For a parameter $p$, we write $a=O_p(b)$ to mean
  that $a/b$ is bounded by a universal constant, and  $a=o_p(b)$ if $a/b \to 0$ as $p\to \infty$.
If $p=n$ then we omit $p$ from the notation.
Finally, we write throughout $C,c$ for constants that are absolute but
may change from line to line.

\subsection*{Acknowledgements} 
We thank the anonymous referees for a careful
reading of the manuscript and for their corrections and suggestions.
This work was supported by
funding from the European Research Council (ERC) under the European Unions Horizon 2020 research and innovation program (grant agreement number 692452). 

\subsection*{Data availability} Data sharing not applicable - no new data generated.

\section{The recursion, definitions of regimes, and a high level description
of the proof}

Fix $z \in (-2,2)\setminus\{0\}$. Let $n \in \NN$, $n\geq 1$ and 
let $(\phi_k)_{k \in \{-1,\ldots, n\}}$ be defined by the recursion
\begin{equation}
  \label{basic-recursion}
 \phi_{-1} = 0, \ \phi_0 = 1,\;  \phi_{k} = \big( z\sqrt{n}-b_{k}\big) \phi_{k-1} - a_{k-1}^2\phi_{k-2},
\; k\geq 1.
 \end{equation}
 Using the tri-diagonal structure of $\Jn$, we note that 
 $\phi_n=\det(z\sqrt{n}I_n-\Jn)$.
Fix $\kappa>0$ large and set
\begin{equation}
  \label{eq-zkk0}
  z_k=z\sqrt{n/k}, \quad 
  k_0 = \lfloor {z^2n}/{4} \rfloor, \quad l_0=\lfloor \kappa k_0^{1/3}\rfloor
\end{equation}
(so that $z_{k_0} = 2+O(1/n)$),  and 
\begin{equation}
  \label{eq-alpha}
\alpha_k=\alpha_k(z)=\left\{\begin{array}{cl}
    z_k/2+\sqrt{z_k^2/4-1+c_{k}},& k<k_0-l_0\\
    1,&  k\geq k_0- l_0.
  \end{array}
  \right.
\end{equation}
Let $\psi_k$ denote the scaled variables
\begin{equation}
  \label{eq-psik}
  \psi_k = \frac{\phi_k}{\sqrt{k!}}\prod_{i=1}^k \frac{1}{\alpha_i} ,\end{equation}
  which satisfy the recursion
  \begin{equation}
    \label{eq-recpsik}
    \psi_{k}=\Big(z_k-\frac{b_k}{\sqrt{k}}\Big)\frac{\psi_{k-1}}{\alpha_{k}}-
    \frac{a_{k-1}^2}{\sqrt{k(k-1)}} \frac{\psi_{k-2}}{\alpha_k \alpha_{k-1}}, \quad k\geq 2 ,
  \end{equation}
  with $\psi_0=1, \psi_1=(z\sqrt{n}-b_1)/\alpha_1=1+O(1/\sqrt{n})$ with
  probability approaching $1$ as $n\to\infty$.
  Note that
  \begin{equation}
    \label{eq-Dn}
    D_n(z,\Jn)=\det(zI_n-\Jn/\sqrt{n})=n^{-n/2} \det(z\sqrt{n}I_n-\Jn)= 
    C_n(z) \psi_n, 
  \end{equation}
  where
  \begin{equation}
    \label{eq-Cn}
   C_n(z)= \sqrt{{n!}/{n^n}} \prod_{i=1}^n \alpha_i.
 \end{equation}
  (See Remark \ref{rem-asym} below for an explicit evaluation of $\log C_n(z)$ up to order $O(1)$.)
  Therefore, the study of $w_n(z)$, see \eqref{eq-defwn}, boils down to the study of the recursion \eqref{eq-recpsik}.

   Set $X_k^T = (\psi_{k},\psi_{k-1})$.
  Since the laws of the
coefficients $a_k$ and $b_k$ are absolutely continuous with respect to
the Lebesgue measure, for fixed $z$,   $\psi_k \neq 0$ for all $k\in \{1,\ldots,n\}$,
almost surely.
In particular,  we can define $\delta_k$ such that
\begin{equation}
\label{eq-Xdelta} X_{k} = \psi_{k} \left(\begin{array}{c}
  1\\
  (1+\delta_{k})^{-1}
\end{array}\right),\; k\geq 1.
\end{equation}
From the recursion \eqref{eq-recpsik} we obtain that
\begin{equation}
  \label{eq-Xk}
  \ X_{k} = (A_k +W_k) X_{k-1}, \quad k\geq 1,
\end{equation}
where
\begin{align}
  \label{def-Ak}
  A_k =  \left( \begin{array}{cc}
\frac{z_k}{{\alpha_k}} & -\frac{(1-c_k)}{{\alpha_k \alpha_{k-1}}} \\
1 & 0
\end{array}\right), \ W_k =\left( \begin{array}{cc}
-\frac{b_k}{{\alpha_k}\sqrt{k}} & -\frac{g_k}{{\alpha_k \alpha_{k-1}}\sqrt{k}}\\
0 & 0
\end{array}\right).
\end{align}
  \subsection{The scalar and oscilatory regimes}
  \label{sec-regimes}
  We consider (for most $k$'s except for $k=O(1)$, which eventually 
  do not matter), the matrix $W_k$ and the constant 
 $c_k$ as perturbations. Then, the eigenvalues of $A_k$ (with $c_k=0$)
 are real
 for $k<k_0$, with modulus larger than $1$, and pure imaginary for 
 $k>k_0$.
  We thus distinguish between three regimes.
  \begin{itemize}
    \item \textit{The scalar regime.} This is the regime where $k\leq k_0-l_0$. As we explain momentarily, in this regime the recursions
      \eqref{eq-Xk} can be essentially replaced by recursions for the scalar
      quantity $\psi_k/\psi_{k-1}=1+\delta_k$, with $\delta_k$ small
      and (essentially) satisfying a scalar equation, hence the name. The analysis of the scalar regime is carried out in Section
\ref{sec-scalar}.
    \item \textit{The transition regime.} This is the regime where 
      $k_0-l_0<k\leq k_0+l_0$.
      In this regime, the eigenvalues of $A_k$ are both close to $1$. The analysis of the transition regime is carried out in Section
\ref{sec-trans}.
    \item \textit{The oscilatory regime.} This is the regime where $k>
      k_0+l_0$. In this regime, 
      the eigenvalues of $A_k$ are complex, of modulus very close 
      to $1$, hence the name. The analysis of the oscilatory regime is carried out in Sections
\ref{subsec-change} and \ref{sec-timechange}, 
and constitutes the hardest and most innovative part of the paper.
  \end{itemize}
  \subsection{High level description of the proof}
 \label{sec-structure}
 We begin the description with the scalar regime, even if
 its analysis is brought last in Section \ref{sec-scalar}. In that case, writing 
 $\psi_k=(1+\delta_k) \psi_{k-1}$, see \eqref{eq-Xdelta},
 we obtain, after some manipulations, that
 \begin{equation}
   \label{eq-deltak}
   \delta_{k}=u_k +v_k \delta_{k-1}/(1+\delta_{k-1}), \quad k\geq 2
 \end{equation}
 where $v_k,u_k$ are as in \eqref{eq-ukvk} below and $\delta_1=O(1/\sqrt{n})$ with high probability; 
 in particular, $v_k$ is essentially deterministic
 and the $u_k$'s are independent random variables. Whenever $\delta_k$ is small,
 the equation \eqref{eq-deltak} can be linearized (around $\delta_k\sim 0$), 
 in the form
 \begin{equation}
   \label{eq-bardeltak0}
  \bdelta_1=\delta_1, \quad
 \bdelta_k=u_k+v_k \bdelta_{k-1}, \quad k\geq 2.
 \end{equation}
 The recursion \eqref{eq-bardeltak} can be solved, and it is straightforward
 to obtain a-priori bounds on $\bar \delta_k$, which then can be bootstraped to 
 a-priori bounds on $\delta_k$, see section 
 \ref{sec-det}. 
 In this analysis, it is important that 
 $v_k\sim 1- \sqrt{l/k_0}$, for $k>(1-\veps) k_0$ with $l=k_0-k$, 
 and there is a natural decay of the
 influence of $u_j$ on $\delta_k$ for $j<k-\sqrt{k_0/l}$. 
 In particular, in the full scalar regime, 
 $\delta_k$ is small with high probability. Therefore, 
 \[\log \prod_{i=1}^k (1+\delta_i)
 \sim \sum_{i=1}^k \delta_i-\frac12 \sum_{i=1}^k \delta_i^2 \,.\]
 Further, the correlation structure of the $\delta_k$ is not stationary, and
 the random variables
 $\{\delta_j\}$, $j=k,\ldots,k+\sqrt{k_0/k}$ are strongly correlated. 
 The upshot of this analysis is  a CLT for the random variable
 $\log(|\psi_{k_0-l_0}|/|\psi_1|)$.

 One of the consequences of the analysis of the scalar regime is that
 $\delta_{k_0-l_0}\sim O(k_0^{-1/3})$. In the transition regime, 
 $\delta_k$ grows, and the linearization 
 approach used in the scalar regime fails.
On the other hand, the vector $(1,1)$ is an eigenvector of $A_k$, with 
eigenvalue approximately $1$. Using that $\delta_{k_0-l_0}$ is small and a change of basis, one 
shows in Section \ref{sec-trans},
by a somewhat tedious combinatorial analysis, 
that $\log(|\psi_{k_0+l_0}/\psi_{k_0-l_0}|)=O(1)$. Having established that, 
it is natural to change basis to the basis where $A_k$ is the conjugate of a rotation matrix. In this, a difficulty arises: contrary to
the case $z=0$, the change of basis is not constant, but rather $k$ dependent,
and it has an amplification effect on the noise,
see Fact \ref{newbasis} below.  Because of these facts, a global linearization
is difficult. Instead, we construct random blocks of indices $l\in
k_0+ [l_i,l_{i+1})$, see
Section \ref{sec-timechange}, with the following properties: 
\begin{itemize}
\item The $l_i$'s are stopping times with respect to the filtration $(\mathcal{F}_l)_l$ where $\mathcal{F}_l$ is generated by the variables $g_k,b_k$ for $k\leq k_0+l$.
  \item Along a block, linearization is possible and gives a good control,
conditional 
    in $\mathcal{F}_{l_i}$ and on the event where $\delta_{k_0+l_i}$ is small enough,
    on the change of norm $\|X_{k_0+l_{i+1}}\|/\|X_{k_0+l_i}\|$. 
  \item The rotation achieved by $\prod_{k \in [l_i, l_{i+1})} A_{k_0+l}$ 
      is (almost) a multiple
    of $2\pi$, and so the associated $\delta_{k_0+l_{i+1}}$ is small.
\end{itemize}
It turns out that, essentially,
a good choice is $l_i\sim (k_0/i)^{1/3}$ for $i\lesssim  k_0^\tau$
and $l_i\sim k_0/(k_0-i)^{2}$ for $i\gtrsim k_0^\tau$, with $\tau \in (1/4,2/5)$,
see Section 
\ref{sec-defblocks}.  We remark that for $i<k_0^\tau$, the variance accumulated in each block
is $\sim 1/i$. We also remark that the change of basis induces a  drift of order $1/i$, that eventually will be cancelled by the drift in the scalar part; see Lemma \ref{decompbaseQ}.

Having performed the linearization, the variable $\log \|X_{k_0(1+\veps)}\|/\|X_{k_0-l_0}\|$
is well approximated by a drifted martingale, and a CLT ensues, see  Proposition
\ref{convlawblock}. Finally, the control of $\log(\|X_n\|/\|X_{k_0(1+\veps)}\|)$ is 
straightforward, and it remains to relate $\|X_n\|$ and $\psi_n$. This is done by means of an
anti-concentration argument, see Lemma \ref{lastblock}. 

\subsection{Auxiliary propositions and proof of Theorem \ref{theo-main}}
\label{proof-main}
Recall the notation $k_0,l_0$, see \eqref{eq-zkk0}. 
 Recall also that  $\gamma$
denotes the standard Gaussian distribution on $\RR$, and that
$d$ denotes a metric on $\mathcal{P}(\RR)$ which is
compatible with weak convergence.
The following propositions correspond to 
the different regimes. Since we consider $z$ fixed, we omit it from
the notation.
\begin{proposition}
  \label{prop-scalar}
   Let $\mu_{n,\kappa}$ denote the law of
	\[ \frac{ \log|\psi_{k_0-l_0}|  + \frac{v-1}{6}\log n }{\sqrt{v\log n/3}} . \]
	Then,
	\begin{equation}
	  \label{eq-scalarrec}
	  \lim_{\kappa \to \infty} \limsup_{n \to \infty} d(\mu_{n,\kappa}, \gamma) = 0. 
	\end{equation}
	Further, there exists a constant $C_\kappa$, possibly depending on $\kappa$, so that
	\begin{equation}
	  \label{eq-extra}
	  \lim_{\kappa \to \infty} \limsup_{n \to \infty} 
	  \PP( |\delta_{k_0-l_0}| \leq C_\kappa k_0^{-1/3})=1.
	\end{equation}
\end{proposition}
\begin{proposition}\label{osci}
  Let $\mu_{n,\kappa}^{(0)}$ denote the conditional law of
$$ \frac{\log \Big| \frac{\psi_n}{\psi_{k_0- l_0}}\Big| - \frac{v-1}{6}\log n}{\sqrt{v\log n /6}},$$
given $\mathcal{F}_{k_0-l_0}$. Then there exist  constants
$C_\kappa$, possibly 
depending on $\kappa$, so that
in probability, 
$$\lim_{\kappa \to \infty}\limsup_{n\to \infty} \Car_{\{| \delta_{k_0-l_0}| \leq C_\kappa k_0^{-1/3}\}} d(\mu_{n,\kappa}^{(0)}, \gamma) =0.$$
\end{proposition}
The proofs of Propositions \ref{prop-scalar} and \ref{osci} are
presented in Sections \ref{sec-scalar} and 
\ref{sec-osci}. Based on them, we can bring the proof of Theorem \ref{theo-main}.
\begin{proof}[Proof of Theorem \ref{theo-main}]
  Combining the propositions gives that 
  $ (\log |\psi_n|)/{\sqrt{v\log n/2}}$
  converges to a standard Gaussian random variable. Combined
  with \eqref{eq-Dn}, this yields the theorem.
 \end{proof}
 \begin{remark}
   \label{rem-asym}
   We provide here an explicit expansion for $C_n(z)$.
 Stirling's formula 
  gives that $\sqrt{n!/n^n}=(2\pi n)^{1/4} e^{-n/2}(1+o(1))$.  
Further,
\begin{eqnarray*}
  &&\sum_{k=1}^{k_0-\ell_0} \log \alpha_k=
  \sum_{k=1}^{k_0-\ell_0} \log\Big(\sqrt{\frac{k_0}{k}}+\sqrt{\frac{k_0}{k}-1+O(k^{-1})}\Big)\\
  &=& \frac12 \sum_{k=1}^{k_0-\ell_0} \log(k_0/k)+
  \sum_{k=1}^{k_0-\ell_0} \log(1+\sqrt{1-k/k_0})+O(1)=:I_1+I_2+O(1).
\end{eqnarray*}
Rearanging, we obtain that
\[I_1=\frac{k_0-\ell_0}{2}\log k_0-\frac12(k_0-\ell_0)\log(k_0-\ell_0)+
\frac12(k_0-\ell_0)-\frac{\log n}{4}+O(1)=\frac{k_0}{2}-\frac{\log n}{4}+O(1).\]
On the other hand, we have that
\[ I_2=k_0 \int_{1/k_0}^{1-\ell_0/k_0}\log(1+\sqrt{1-x}) dx+O(1).\]
For $x<1$, with $u(x)=1+\sqrt{1-x}$, we have that 
\[\int^x \log(1+\sqrt{1-y}) dy=-u(x)^2 \log u(x)+u(x)^2/2+2u(x)\log u(x)-2\sqrt{1-x},\]
and therefore we obtain that  $I_2=k_0/2+O(1)$.
Altogether,
we obtain that $\log C_n(z)=n(z^2/4-1/2)+O(1)$,
which coincides as expected to leading order with the
logarithmic potential of the semi-circle.
\end{remark}

  \section{The transition and oscillatory regimes - proof of Proposition \ref{osci}} 
  \label{sec-osci}
This section is devoted to the proof of
Proposition \ref{osci}. 
\subsection{Crossing the $k_0$ barrier - the transition regime}
\label{sec-trans} In this section, we will show that with probability going to $1$ as $\kappa$ goes to $+\infty$, the contribution of the transition regime to
the norm of $X_k$ is of order $1$, that is $\|X_{k_0+l_0}\| \asymp \|X_{k_0-l_0}\|$ up to multiplicative constants which depend on $\kappa$. As a consequence, we will have proven that $\log \|X_{k_0+l_0}\| = \log \|X_{k_0-l_0}\| +O_\kappa(1)$, with probability going to $1$ when $\kappa \to+\infty$. 
 
In a first step, we will prove that the noise accumulated in the block 
$[k_0-l_0,k_0+l_0]$ is at most of order $1$ in the direction $(1,1)$,
whereas in the orthogonal
direction the contribution is
much smaller, as described in the following proposition.
Throughout this section, we set
  $u=(1,1)$ and $v=(0,1)$, without
further mentioning it. 
Recall \eqref{eq-Xdelta} and the notation \eqref{eq-zkk0}.
 \begin{proposition}\label{firstblock}  Fix
 $\kappa \geq 1$. The
     solution $X_{k_0+l_0}$ of \eqref{eq-Xk}
     satisfies, for all $n$ large enough,
$$ \frac{X_{k_0+l_0} }{\psi_{k_0-l_0}}=x_{l_0}u + y_{l_0}v,$$
where, for an absolute constant $C$,
$$  \EE_{\mathcal{F}_{k_0-l_0} }  (x_{l_0}^2) \lesssim  e^{C{\kappa^3}}(1+ | l_0\delta_{k_0-l_0}|)^2, \  \EE_{\mathcal{F}_{k_0-l_0} } (y_{l_0}^2) \lesssim  e^{C{\kappa^3}} k_0^{-\frac{2}{3}}(1+|l_0\delta_{k_0-l_0}|)^2.$$
\end{proposition}
In a second step we will show that the norm  $\| X_{k_0+l_0}\|$ cannot decrease too much from $\|X_{k_0-l_0}\|$ as described in the following proposition.
\begin{proposition}\label{tightlower} For $\kappa\geq 1$ there exists $C_\kappa>0$ so that
$$ \PP_{\mathcal{F}_{k_0-l_0}}\Big( \|X_{k_0+l_0}\|\leq C_\kappa^{-1}\|X_{k_0-l_0}\|
\Big) \leq \kappa^{-1}.$$
\end{proposition}
The proofs of Propositions \ref{firstblock} and \ref{tightlower} are elementary but rather long and technical, and we advise the reader to skip them at first reading, and proceed directly to Section \ref{subsec-change}.
\begin{proof}[Proof of Proposition \ref{firstblock}]
We start by showing a more general result, providing estimates of the variances and expectations accumulated on a sub-block of length $m$ nested
in a interval of length $l_0$ around $k_0$. Recall
  the definition of the matrices $A_k$ and $W_k$ \eqref{def-Ak}.
\begin{lemma}\label{block} Let $\kappa \geq 2$. Let $k_0/2\geq l_0\geq m\geq 1$.
  For any $k_0- l_0\leq k \leq k+m \leq k_0+l_0$ and $x,y\in \RR$,
  \begin{equation}
    \label{eq-xmym}
    \prod_{j=k+1}^{k+m} \big( A_j + W_j\big) (x u+y v)=:x_mu+y_m v
  \end{equation}
where, for an absolute constant $C$,
\begin{equation}
  \label{eq-xm}
  |\EE x_m| \leq Ce^{\frac{Cm^2 l_0}{k_0}}\Big( 1+ \frac{ml_0}{k_0}\Big)(|x|+m|y|), \quad \Var (x_m) \leq C\Big(e^{\frac{m^3}{k_0}}-1 + \frac{l_0}{k_0}\Big)e^{\frac{Cm^2 l_0}{k_0}}(|x|+m|y|)^2,
\end{equation}
and
\begin{equation}
  \label{eq-ym}
  |\EE y_m - y| \leq \frac{Cml_0}{k_0}e^{\frac{Cm^2l_0}{k_0}}(|x|+m|y|), \   \Var(y_m ) \leq \frac{Cl_0}{k_0} e^{\frac{Cm^2l_0}{k_0}}(|x|+ m|y|)^2.\end{equation}
\end{lemma}
\begin{proof}
Recall that  $\alpha_j=1$ for $j\geq k_0-l_0$ (see \eqref{eq-alpha}). Therefore for any $j\geq k_0-l_0+1$, we have that $ A_ j = A + B_j$,
where
$$ A = \left(\begin{array}{cc}
2 & -1 \\
1 & 0
\end{array}\right), \quad B_j = \left(\begin{array}{cc}
z_j-2 & c_j \\
0 & 0
\end{array}\right) \text{ and } \ W_j = \begin{pmatrix} -\frac{b_j}{\sqrt{j}} & -\frac{g_j}{\sqrt{j}} \\ 0 & 0 \end{pmatrix}.$$
Expanding the product $\prod_{j=k+1}^{k+m} (A+B_j+W_j)$, we get
$$ \prod_{j=k+1}^{k+m} (A+B_k+W_k) = A^{m} + \sum_{p=1}^{m} \sum_{ k+1  \leq k_1<\ldots < k_p \leq k+ m} \Big(\prod_{i=1}^{p} A^{k_{i+1}-k_{i}-1}( B_{k_i}+W_{k_i})\Big) A^{k_1-k-1},$$
where in the inner sum, $k_{p+1}= k+m$. From the observation that for any $l\geq 1$,
$$ A^l =   \left( \begin{array}{cc}
l+1 & -l\\
l & -(l-1)
\end{array}\right),$$
we obtain that for any $j<j'$,
$$ A^{j'-j-1}(B_{j} +W_{j}) =  (j'-j)S_{j} + R_{j},$$
where
$$ S_{j} = \left(\begin{array}{cc}
z_{j}-2  - b_{j}/\sqrt{j} & c_j - g_{j}/\sqrt{j} \\
z_{j}-2- b_{j}/\sqrt{j}& c_j - g_{j}/\sqrt{j}
\end{array}\right), \quad {R_j = \left(\begin{array}{cc}
0 & 0 \\
-(z_{j}-2)+ b_{j}/\sqrt{j}& -c_j + g_{j}/\sqrt{j}
\end{array}\right)}.$$
Recall that $z_j= z\sqrt{n/j}$, see \eqref{eq-zkk0}. 
Thus, for any $|k_0-j|\leq k_0/2$, we have that
$|z_j -2| = O(|k_0-j|/k_0)$. Thus,  uniformly in $j\in\{k_0-l_0,\ldots, k_0+l_0\}$, we obtain 
that
$\| \EE R_{j}\| = O(l_0/k_0)$ and $\|\mathrm{Var}(R_{j})\| = O(1/k_0)$, 
where $\mathrm{Var}(R_j)$ denotes the  matrix whose entries are the variances of the corresponding entries of $R_j$.
Thus we can write
\begin{align} \prod_{j=k+1}^{k+m} (A+B_j+W_j) &= A^{m} + \sum_{p=1}^{m}   \sum_{ k+1 \leq k_1<\ldots < k_p \leq k+ m}\Big(\prod_{i=1}^{p} \big((k_{i+1}-k_i) S_{k_i} + R_{k_i}\big) \Big) A^{k_1-k-1} \nonumber\\
& =  A^{m} + C_m+T_m,\label{decompoprod}
\end{align}
where
\begin{equation}
  \label{def-Cm}
 C_m = \sum_{p=1}^{m} \sum_{ k+1 \leq k_1<\ldots < k_p \leq k+ m}\prod_{i=1}^p (k_{i+1}-k_i)S_{k_i} A^{k_1-k-1},
 \end{equation}
with $k_{p+1} = k+m$ in the inner sum, and
\begin{align} T_m & = \sum_{ k+1 \leq k_1<\ldots < k_p \leq k+ m \atop 1 \leq p\leq m}\sum_{1\leq i_1<\ldots <i_q\leq p \atop 1\leq q \leq p} \Big(\prod_{r=1}^q\Big(\prod_{i=i_{r}+1}^{i_{r+1}-1}(k_{i+1}- k_{i}) S_{k_i}\Big) R_{k_{i_r} }\Big) A^{k_1-k-1}, \label{defDeltabis}
\end{align}
with $i_{q+1}=k+m$ in the inner sum.  Note that 
\[ {x_m = \langle e_1,  \prod_{j=k+1}^{k+m} \big( A_j + W_j\big) (xu+yv)\rangle, \ y_m = \langle e_2-e_1,  \prod_{j=k+1}^{k+m} \big( A_j + W_j\big) (xu+yv)\rangle}.\]
Since $A^m u=u$ and $A^mv=-mu +v$, we have
\begin{align}
  x_m &= x-my +\langle e_1, C_m (x u+y v)\rangle + \langle e_1,T_m(x u+y v)\rangle,\nonumber \\
y_m &=  y +  \langle e_2- e_1,T_m(x u+y v)\rangle,\label{decompoxy}
\end{align}
where we used the fact that $\mathrm{Ran}(C_m) \subset \{e_2-e_1\}^\perp$.
We start by  computing the expectations of $T_m u$ and $T_m v$.
Note that for any $l\geq 0$, $A^lu=u$, for any $i$, $\|\EE S_{k_i}\| = O(l_0/k_0)$ and $\|\EE R_{k_i}\| = O(l_0/k_0)$. 
Using the fact that the Hilbert-Schmidt norm is sub-multiplicative, we obtain
that
\[
  \|\EE T_m u \| \leq \!\!\!\!\!\sum_{ I\subset \{k+1,\ldots,k+m\}\atop |I|\geq 1} \sum_{J\subset I \atop |J|\geq 1}c^{|I|} \Big( \frac{ml_0}{k_0}\Big)^{|I\setminus J|} \Big( \frac{l_0}{k_0}\Big)^{|J|}
= {\!\sum_{j=1}^{m} \binom{m}{j} c^j \sum_{j'=1}^j\binom{j}{j'}m^{-j'} \Big(\frac{ml_0}{k_0}\Big)^{j}}\!, \]
where $c$ is an absolute constant. 
As $\sum_{j'=1}^j \binom{j}{j'} m^{-j'}=O(j/m)$ for $j\leq m$, and $\binom{m}{j} \leq m^j/j!$ we deduce that
$$\|\EE T_m u \| \lesssim \frac{1}{m}\sum_{j=1}^{m} \frac{1}{(j-1)!}\Big( \frac{c m^2l_0}{k_0}\Big)^{j}\lesssim \frac{ml_0}{k_0}e^{\frac{cm^2l_0}{k_0}}.$$
Using that $\| A^l\|\lesssim l$ for any $l\geq 1$, we deduce similarly that
$$\|\EE T_m v \| \lesssim \frac{m^2l_0}{k_0}e^{\frac{cm^2l_0}{k_0}},$$
and we  conclude that
\begin{equation} \label{Tmexpct}\| \EE T_m(xu+yv) \| \lesssim \frac{ml_0}{k_0}e^{\frac{cm^2l_0}{k_0}}\big( |x|+|my|\big).\end{equation}
To compute the variance of $T_m u$ and $T_m v$, we define for any $l$, $\hat S_l = S_l -\EE S_l$ and $\hat R_l = R_l - \EE R_l$. Expanding the products, we obtain
$$T_m u = \sum_{ I\subset \{k+1,\ldots,k+ m\}\atop |I| \geq 1}\sum_{ J,K\subset I  \atop |J| \geq 1} M_{I,J,K}u,$$
where
$$
M_{I,J,K} = \prod_{i\in I\setminus J} (k_{i+1}-k_i) \prod_{l\in I} M_l,\;
\mbox{\rm with}\;
M_l = \begin{cases}
 \hat S_l & \text{ if } l\in K \setminus  J, \\
\EE S_l & \text{ if } l\in I\setminus  (K \cup  J),\\
\hat R_l& \text{ if } l\in K\cap J,\\
\EE R_l & \text{ if } l\in J\setminus  K
\end{cases}.$$
Since $(S_l, R_l)_{l\geq 1}$ are independent, we get
$$\EE \| (T_m - \EE T_m)u \|^2 \leq \EE 
\sum_{ J\subset I ,J' \subset I' \atop |J|,|J'| \geq 1}\sum_{ K\subset I\cap I'  \atop |K| \geq 1} \tr \big(M_{I,J,K}^TM_{I',J',K}\big).$$
As  $\EE \| \hat S_{l} \|^2  =O(1/k_0)$, $\EE \| \hat R_l \|^2 = O(1/k_0)$,
$\| \EE R_l \| = O(l_0/k_0)$ and $\| \EE S_{l} \| = O(l_0/k_0)$,  uniformly in $l\in I$,  we deduce that
\begin{align*}
\EE \tr \big(M_{I,J,K}^TM_{I',J',K}\big) \leq & C^{|I|+|I'|} m^{|I\setminus J|+|I'\setminus J'|}  \Big( \frac{1}{k_0}\Big)^{| K| }  \Big( \frac{l_0}{k_0}\Big)^{|I\setminus K|  +|I'\setminus K| }.
\end{align*}
Using the fact that $| J\cap K| +|J'\cap K| =
| (K\cap J)\Delta (K\cap J')| +2|K\cap J \cap J'|$, we can write
\begin{align*}&\EE \tr \big(M_{I,J,K}^TM_{I',J',K}\big)\\&\leq  \Big( \frac{Cml_0}{k_0}\Big)^{|I| +|I'| } \Big( \frac{l_0^2}{k_0}\Big)^{- |K|}  \Big( \frac{1}{m^2}\Big)^{| K\cap J\cap  J'|} \Big( \frac{1}{m}\Big)^{|
    (K\cap J) \Delta (K\cap J')| }\Big( \frac{1}{m}\Big)^{| J\setminus  K  | + | J'\setminus  K | }.
\end{align*}
Denoting
$$a =| I \setminus I'|, \ b=|I'\setminus I|, \ c=|I\cap I'|, \ d =|K|, \
e = |( K\cap J)\setminus (K\cap J')|,$$
$$  f = | (K\cap J')\setminus (K\cap J)|, \ g =|K\cap J\cap J'|, \ h=|J\setminus K| + |J' \setminus K|,$$
we deduce that since $|J|, |J'| , |K|\geq 1$, we have that
\begin{equation}\label{const} h\geq 2 \text{ or } \Big(h=1 \text{ and } e+f+g \geq 1\Big) \text{ or } \Big(h=0 \text{ and } e+f+2g \geq 2\Big).\end{equation}
Observe that for choosing $J\setminus K$ and $J'\setminus K'$ in  respectively $I\setminus K$ and $I' \setminus K$, we have 
\[ \sum_{h_1+h_2 =h }\binom{a+c-d}{h_1}\binom{b+c-d}{h_2} 
= \binom{a+b+2(c-d)}{h} \quad \mbox{\rm choices}.\]
Therefore,
\begin{equation}
  \label{eq-sat1}
  \EE \| (T_m - \EE T_m)u \|^2\leq \sum_{a,b,c,d,e,f,g,h }  \binom{m}{a,b,c} \binom{c}{d}\binom{d}{e,f,g} \binom{a+b+2(c-d)}{h}  Y_{a,b,c,d,e,f,g,h},
\end{equation}
where $a,b,c,d,e,f,g,h \leq m$ satisfy \eqref{const}, and
\begin{equation} \label{defY} Y_{a,b,c,d,e,f,g,h}=
\Big( \frac{Cm l_0}{k_0}\Big)^{a+b+2c}\Big( \frac{l_0^2}{k_0}\Big)^{-d}
\Big( \frac{1}{m^2}\Big)^{g}  \Big( \frac{1}{m}\Big)^{e+f} \Big( \frac{1}{m}\Big)^{h}.\end{equation}
Using that $d\leq c$, we obtain that 
\begin{equation}
  \label{eq-sat2}
   \EE \| (T_m - \EE T_m)u \|^2\leq \sum_{a,b,c,d,e,f,g,h }  \binom{m}{a,b,c} \binom{c}{d}\binom{c}{e,f,g} \binom{a+b+2c}{h}  Y_{a,b,c,d,e,f,g,h}.
\end{equation}
In the expression \eqref{defY}, we split $(Cml_0/k_0)^{a+b+2c}$ into $(Cml_0/k_0)^{a+b+c}(Cml_0/k_0)^c$. Summing over $d$, and using that $c\leq m \leq l_0$, we get
$$ \sum_{ d\leq c} \binom{c}{d}\Big( \frac{Cml_0}{k_0}\Big)^c
\cdot \Big( \frac{l_0^2}{k_0}\Big)^{-d}\leq\sum_{ d\leq c} \binom{c}{d}
\Big( \frac{Cml_0}{k_0}\Big)^{c-d} \leq e^{\frac{Cm^2l_0}{k_0}},$$
which gives
\begin{align*}
&  \EE \| (T_m - \EE T_m)u \|^2 \leq  e^{\frac{Cm^2l_0}{k_0}}\\
& \times \sum_{a,b,c,e,f,g,h }  \binom{m}{a,b,c} \binom{c}{e,f,g} \binom{a+b+2c}{h} \Big( \frac{Cml_0}{k_0}\Big)^{a+b+c}  \Big( \frac{1}{m}\Big)^{e+f} \Big( \frac{1}{m^2}\Big)^{g} \Big( \frac{1}{m}\Big)^{h}.
\end{align*}
We split the sum  on the right-hand side above into the sums $\Sigma_i$,
$i=0,1,2$, over indices such that $h=i$ if $i=0,1$ and $h\geq 2$ if $i=2$.
 For any $i=0,1,2$, we have
$$
\sum_{ h\geq i}\binom{a+b+2c}{h}\Big( \frac{1}{m}\Big)^{h}  \leq \sum_{h\geq i} \frac{1}{h!} \Big( \frac{a+b+2c}{m} \Big)^h \leq \Big(\frac{a+b+2c}{m}\Big)^i e^{(a+b+2c)/m}\lesssim \Big( \frac{a+b+c}{m}\Big)^i,$$
where we used that $a+b+c\leq m$. Similarly, for any $i=0,1,2$,
$$\sum_{ e+f+2g\geq 2-i} \binom{c}{e,f,g} \Big( \frac{1}{m}\Big)^{e+f+2g}\lesssim \Big( \frac{c}{m}\Big)^{2-i}.$$
 Thus, taking advantage of \eqref{const}, we obtain that for any $i=0,1,2$,
 $$ \Sigma_i \lesssim  \frac{1}{m^2}\sum_{a+b+c\geq 1 }  \binom{m}{a,b,c} (a+b+c)^2\Big( \frac{Cml_0}{k_0}\Big)^{a+b+c}.$$
Using that $\sum_{a+b+c=k} \binom{m}{a,b,c} = 3^k \binom{m}{k}$ and $a+b+c\leq m$, we get 
 $$\Sigma_i \lesssim \frac{1}{m^2} \sum_{k=1}^m \binom{m}{k} k^2 \Big( \frac{3Cml_0}{k_0}\Big)^k\leq \frac{1}{m^2} \sum_{k=1}^m \frac{k^2}{k!} \Big( \frac{3C m^2l_0}{k_0}\Big)^k.$$
Dividing the last sum into the case where $k=1$ and $k\geq 2$, and using the fact that $m l_0 \leq k_0$, we get
for any $i =0,1,2$, $\Sigma_i \lesssim (l_0/k_0)\exp(\frac{3Cm^2 l_0}{k_0})$, which leads to
$$\label{varu:1}\EE \| (T_m - \EE T_m)u \|^2\lesssim  \frac{l_0}{k_0} e^{ \frac{3Cm^2l_0}{k_0}},$$
and similarly
$$ \label{varDelta}\EE \|( T_m - \EE T_m)v \|^2 \lesssim  \frac{l_0m^2}{k_0} e^{\frac{3Cm^2l_0}{k_0}}.$$
We  conclude that
\begin{equation}\label{varTm} \EE \| (T_m - \EE T_m)(xu+yv) \|^2 \lesssim \frac{l_0}{k_0}e^{\frac{3Cm^2 l_0}{k_0}}\big( |x|^2+ |my|^2\big).\end{equation}
Using \eqref{Tmexpct} and \eqref{varTm}, we get the claimed estimates \eqref{eq-ym}.

In view of \eqref{decompoprod} and
\eqref{decompoxy}, it remains to estimate the term
$C_m$, see \eqref{def-Cm}.
As $A^l u = u$ for any $l\geq 0$, we have that
$$ C_m u = \sum_{p=1}^{m} \sum_{ k+1 \leq k_1<\ldots < k_p \leq k+ m}\prod_{i=1}^p (k_{i+1}-k_{i})\Big( z_{k_i}-2 - \frac{b_{k_i}}{\sqrt{k_i}} +c_j - \frac{g_{k_i}}{\sqrt{k_i}}\Big) u$$
with $k_{p+1} = k+m$ in the inner sum. 
 Since $|z_j-2| = O(l_0/k_0)$, see \eqref{eq-zkk0},
 and $c_j = O(1/k_0)$, we get that
\begin{align}
|\EE\langle e_1,C_m u\rangle| & \leq \sum_{p=1}^{m} \sum_{ k+1 \leq k_1<\ldots < k_p \leq k+m}\Big( \frac{cm l_0}{k_0}\Big)^p \leq \Big(e^{\frac{cm^2l_0}{k_0}}-1\Big), \label{expectu2}
\end{align}
where $c$ is an absolute constant, and similarly,
$$|\EE \langle e_1,C_m v\rangle|  \lesssim \Big(e^{cm^2l_0/k_0}-1\Big)m,$$
which yields that
\begin{equation} \label{expectCm}| \EE \langle e_1, C_m(x u+yv)\rangle | \lesssim \Big(e^{cm^2l_0/k_0}-1\Big)\big(|x|+|my|\big).\end{equation}
For the purpose of proving later Proposition \ref{tightlower}, we will make a finer variance computation, and estimate the variance of the component of $\langle e_1, C_mu \rangle$ on chaoses of degree greater or equal than a given
$s\in\NN$.  For any such $s\geq 1$,
 let
 \begin{equation}
   \label{chaos}
   u_m^{(s)} = \sum_{ J \subset \{ k+1,\ldots,k+ m\}\atop |J| \geq s}(-1)^{|J|} \prod_{j\in J} \frac{b_j+g_j}{\sqrt{j}} \sum_{K \subset J^c} \prod_{j\in K} \Big( c_j+z_j-2 \Big) \prod_{i=1}^p (k_{i+1}-k_{i}),
 \end{equation}
 where $\{ k_1 < \ldots <k_p\} = K \cup J$ and $k_{p+1}=k+m$. In particular,
 \begin{equation}\label{chaos1}\langle e_1, C_m u\rangle -\EE  \langle e_1, C_m u\rangle= u_m^{(1)}.
\end{equation}
Since $c_j +z_j -2 = O(l_0/k_0)$, we have that
\begin{align*}
\Var  (u_m^{(s)}) &  \leq  \sum_{ J \subset \{ k+1,\ldots,k+ m\}\atop |J| \geq s} \Big(\frac{c}{k_0}\Big)^{|J|}\Big( \sum_{K \subset J^c}  \Big( \frac{cl_0}{k_0}\Big)^{|K|} m^{|K|+|J|}\Big)^2,
\end{align*}
where $c$ is an absolute  constant.
But, as $|J^c| \leq m$, we have that
$\sum_{K \subset J^c}  ( cm l_0/k_0)^{|K|} \leq e^{cm^2 l_0/k_0}$.
Therefore,
 \begin{equation}\label{varu:2}
\Var  (u_m^{(s)})  \leq  e^{2cm^2l_0/k_0} \sum_{ J \subset \{ k+1,\ldots,k+ m\}\atop |J|\geq s} (cm^2/k_0)^{|J|}\leq e^{2cm^2 l_0/k_0} \sum_{j=s}^{m} \frac{1}{j!} (cm^3/k_0)^j .\end{equation}
Similarly, let $v_m^{(s)}$ denote the component of $\langle u, C_m v\rangle$ on chaoses of degree greater or equal to $s$.  As $A^l v = -lu+v$ for any $l\geq 1$, we get
\begin{equation} \label{varv} \Var( v_m^{(s)}) \lesssim
  m^2 e^{2cm^2 l_0/k_0} \sum_{j=s}^{m} \frac{1}{j!} \Big( \frac{m^3}{k_0}\Big)^j , \end{equation}
where $c$ is an absolute constant, which gives in particular
\begin{equation} \label{varCm} \Var  \langle e_1, C_m(xu+yv)\rangle  \lesssim \Big ( e^{m^3/k_0}-1\Big)e^{2cm^2 l_0/k_0}\big( x^2 +(my)^2\big).\end{equation}
Putting together \eqref{Tmexpct}, \eqref{varTm}, \eqref{expectCm} and \eqref{varCm}, we get the claim \eqref{eq-xm}.
\end{proof}
We now come back to the proof of Proposition \ref{firstblock}.
Using Lemma \ref{block} with $m=l_0 = \lfloor \kappa k_0^{1/3}\rfloor$, we deduce that
$$ \frac{X_{k_0+l_0}}{\psi_{k_0 -l_0}} = x_m u+y_mv,$$
where
$$\big| \EE_{\mathcal{F}_{k_0-l_0} } (x_{l_0})\big| \leq e^{C{\kappa^3}} (1+ l_0|\delta_{k_0-l_0}|), \ \Var_{\mathcal{F}_{k_0-l_0} }(x_{l_0})
 \leq e^{C{\kappa^3}}(1 +l_0|\delta_{k_0-l_0}|)^2,$$
$$ |\EE_{\mathcal{F}_{k_0-l_0} }  (y_{l_0})| \leq k_0^{-\frac{1}{3}}e^{C{\kappa^3} }(1+l_0|\delta_{k_0-l_0}|), \    \Var_{\mathcal{F}_{k_0-l_0} } (y_{l_0}) \leq k_0^{-\frac{2}{3}} e^{C{\kappa^3} }(1+ l_0|\delta_{k_0-l_0}|)^2,$$
and $C$ is an absolute constant.
\end{proof}
\begin{proof}[Proof of Proposition \ref{tightlower}]
We will linearize the product
$\prod_{j=k_0- l_0+1}^{k_0+l_0} (A_j+W_j)$ on sub-blocks of order $\veps k_0^{1/3}$ and use the anti-concentration property of sums of independent random variables to show that the norm cannot decrease too much on such a block. To this end, we start by showing the following anti-concentration  lemma.
\begin{lemma}\label{smallblock}
  Let $\veps>0$ such that $\veps \kappa < 1$. Let $m = \lfloor \veps k_0^{1/3}\rfloor$ and $k\in\NN$ such that $k_0 - l_0 \leq k\leq k+m \leq k_0 +l_0$. Fix $x,y \in \RR$ and let $x_m$ be as in \eqref{eq-xmym}.
Then, for  any $\delta <1/2$,
$$\PP\big( |x_m| \leq \veps^{\frac{5}{2}+\delta} \max(|x|,|my|)\big) \lesssim \veps^{1+\delta}.$$
\end{lemma}
\begin{proof}[Proof of Lemma \ref{smallblock}]
  Recall the notation $Q_X,Q_\mu$ for L\'{e}vy's concentration function, see
\eqref{eq-anticonc}.
We will first prove that
\begin{equation} \label{expansionx} x_m = a_m + x_m^{(1)} +x_m^{(2)} + z_m,\end{equation}
where:\\
$\bullet$
 $a_m$ is deterministic\\
 $\bullet$  If $t>0$ then for $n$ large enough,
\begin{equation} \label{claim1} Q_{x_m^{(1)}}\big(t \max(|x|, |m y|)\big)\leq \frac{C t}{\veps^{3/2}},\end{equation}
\begin{equation} \label{claim2} \PP\big( |x_m^{(2)}|\geq t\max(|x|, m|y|) \big) \leq Ce^{-c\sqrt{t/\veps^3}},\end{equation}
and \begin{equation} \label{claim3}\EE (z_m^2) \lesssim\kappa^2 \veps^{10} \max\big(x^2,(m y)^2\big),\end{equation}
where $c,C$ are absolute constants. To see this, recall \eqref{eq-xmym} and the decomposition \eqref{decompoxy}. It is straightforward to check, using
\eqref{Tmexpct} and \eqref{varTm}, that the contribution of $T_m$ can be absorbed in $z_m$ with (better than) the stated bounds.
The contribution of $A^m$ is deterministic, and thus can be absorbed in $a_m$. Thus, it only remains to evaluate the contribution of $C_m$.
From \eqref{chaos}, we see  that the first degree chaos expansion $ \zeta_m^{(1)} $ of $ \langle e_1, C_m u\rangle$ is
$$ \zeta_m^{(1)} =  \sum_{j=k+1}^{k+m} s_j  \frac{b_j+g_j}{\sqrt{j}},\quad
 s_j = -\sum_{K \subset \{k+1,\ldots,k+m\}\setminus \{j\}}\prod_{l\in K} \Big( c_j+z_l-2 \Big) \prod_{i=1}^p (j_{i+1}-j_{i}),$$
 where in the sum defining $s_j$,
 $\{ j_1 < \ldots <j_p\} = K \cup \{j\}$, and $j_{p+1}= k+m$. Similarly, one can check that the  first degree chaos expansion of $ \langle e_1, C_m v\rangle$ can be written as
$$ \xi_m^{(1)} = \tilde{ \xi}_m^{(1)}  + r_m^{(1)},$$
where $\EE (r_m^{(1)})^2 \lesssim \veps^2 \kappa$ and $ \tilde{ \xi}_m^{(1)}   = \sum_{j=k+1}^{k+m} d_j \frac{b_j+g_j}{\sqrt{j}}$ with
$$ d_j =- \sum_{K \subset \{k+1,\ldots,k+m\}\setminus \{j\}}\prod_{l\in K} \Big( c_j+z_l-2 \Big) \prod_{i=0}^p (j_{i+1}-j_{i}), $$
with $j_{p+1}=k+m$, $j_0=k+1$ and  $\{ j_1 < \ldots <j_p\} = K \cup \{j\}$. 
Absorbing $r_m^{(1)}$ into the term $z_m$ (which is possible due to the 
presence of $(my)^2$ in the latter, see \eqref{claim3}, and the fact
that $r_m^{(1)}$ is coming from the contribution of $v$, i.e. is proportional
to $y$),
we will prove that $x_m^{(1) } = \zeta_m^{(1)} x + \tilde{\xi}_m^{(1)}y$ satisfies \eqref{claim1}. We compute first a lower bound on the variance of $\zeta_m^{(1)}$ and $\tilde{\xi}_m^{(1)}$.  On the  one hand,
$$ s_j' = \sum_{ K\subset  \{k+1,\ldots, k+m\} \setminus \{j\}\atop |K|\geq 1} \prod_{j\in K} \Big| c_j+z_j-2 \Big| \prod_{i=1}^p (j_{i+1}-j_{i}) \leq  m \sum_{p\geq 1} \frac{1}{ p !} \Big( \frac{C l_0 m^2}{k_0}\Big)^{p},$$
 where $C$ is an absolute constant and we used that $|z_j-2|=O(m_0/k_0)$ and that $c_j=O(1/k)$. Thus,
 for $\veps$ small enough and using the fact that $ \veps^2\kappa \leq 1$, we obtain that
$ s_j'  =O( \veps^2\kappa m).$
Therefore,
$ s_j = (k+m-j) + O( \veps^2\kappa m).$
Using the definition of $m$, one concludes that
$  \EE (\zeta_m^{(1)})^2 = (2v \veps^3/3) + O(\veps^5\kappa).$
 Now, to compute the variance of $\tilde{\xi}_m^{(1)}$, we observe similarly as above that
\begin{align*}
 d_j' &= \sum_{K \subset \{k+1,\ldots,k+m\}\setminus \{j\} \atop |K|\geq 1}\prod_{l\in K} \Big| c_l+z_l-2 \Big| \prod_{i=0}^p (j_{i+1}-j_{i})  \lesssim \veps^2 \kappa m^2.
\end{align*}
Thus,
$ d_j = (k+m - j)(j-k-1) + O(\veps^2 \kappa m^2),$
which yields
$$ \EE (\tilde{\xi}_m^{(1)})^2 =  2v \sum_{j= k+1 }^{k+m}\Big(\frac{d_j^2}{j} \Big) = (v \veps^3 m^2/15)+O( \veps^5 \kappa m^2),$$
where we used the fact that $\EE (r_m^{(1)})^2 \lesssim \veps^2 \kappa$.
Finally, we compute the covariance between $\zeta_m^{(1)}$ and $\tilde{\xi}_m^{(1)}$:
\begin{align*}  \EE \zeta_m^{(1)} \tilde{\xi}_m^{(1)} & =2v \sum_{j=k+1}^{k+m} \frac{d_j s_j}{j}\\
 &= \frac{2v}{k_0} \sum_{j=1}^{m}(m-j)^2j+O( \veps^5 \kappa m)
= \frac{v}{6} \veps^3 m +O( \veps^5 \kappa m).
\end{align*}
We deduce from these variance computations that, with  $x_m^{(1)}= x\zeta_m^{(1)} + y\tilde{\xi}_m^{(1)}$,
\begin{align*}
 \EE ( x_m^{(1)})^2
&= (2v/3) \veps^3 x^2 +(v/15) \veps^3 m^2y^2 + (v/3) \veps^3 m xy +  O( \veps^5 \kappa \max ( x^2, y^2 m^2))\\
&\gtrsim \veps^3 \max(x^2, (m y)^2),\end{align*}
where in the last inequality we used that $\kappa\veps<1$.
On the other hand, $x_m^{(1)}$
 is the sum of $m$ independent centered random variables, which we denote by  $Z_j$. The uniform bound \eqref{boundlaplace} on the Laplace transform of $b_j$ and $g_j$ implies that the assumption of Lemma \ref{anticonc} is satisfied. Note that $\EE Z_j^2 = O(m^2/k_0)\max( |x|^2,|my|^2)$ for any $j$. Thus, Lemma \ref{anticonc} yields that for any $t>0$ and $n$ large enough,
$$ Q_{x_m^{(1)}}\big( t\max(|x|,|my|)\big) \leq \frac{C t}{\veps^{3/2}},$$
which gives  \eqref{claim1}.
Let now $\zeta_m^{(2)}$, $\xi_m^{(2)}$ denote respectively the second order chaoses of $\langle e_1, C_m v\rangle$ and $\langle e_1, C_m v\rangle$. From \eqref{chaos}, using that
$$\sum_{K \subset \{k+1,\ldots, k+m\}\setminus \{j,l\}\atop |K|\geq1} \prod_{t\in K} \Big| c_t+z_t-2 \Big| \prod_{i=1}^p (j_{i+1}-j_{i})\leq \sum_{q \geq 1} \frac{1}{q!}\Big( \frac{Cl_0}{k_0}\Big)^q m^{2q+2}\lesssim \veps^2 \kappa m^2,$$
since $\kappa \veps^2 <1$, we deduce that
$$ \zeta_m^{(2)} = \sum_{j< j'} \frac{(b_j+g_j)(b_{j'}+g_{j'})}{\sqrt{jj'}} \Big( (k+m-j')(j'-j)  +O(\veps^2 \kappa m^2)\Big).$$
(Here, the $O(\cdot)$ term is deterministic.)
Let
$$ \tilde{\zeta}_m^{(2)} = \sum_{j,j'} \frac{(b_j+g_j)(b_{j'}+g_{j'})}{\sqrt{j j'}}  (k+m-j')(j'-j) , \ \eta_m^{(2)} = \sum_{j< j'} \frac{(b_j+g_j)(b_{j'}+g_{j'})}{\sqrt{j j'}}O(\veps^2 \kappa m^2).$$
 Let $a_j>0$ such that $a_j\leq 2m/\sqrt{k_0}$. Note that
 \eqref{boundlaplace} implies in particular that the third derivatives of the log-Laplace transforms of $b_j$ and $g_j$ are bounded uniformly in a neighborhood of $0$. By Taylor's expansion, we deduce that there exists $\lambda_0'>0$ such that for any $|\lambda| \leq \lambda_0'$,
$$ \EE e^{\lambda b_j} \leq e^{C\lambda^2},  \ \EE e^{\lambda g_j} \leq e^{C\lambda^2},$$
where $C$ is a constant independent of $j$.
Using Chebyshev's inequality, we get that for any $0<\lambda \leq \lambda_0' \sqrt{k_0}/2m$,
$$ \PP\big(  \sum_{j=k+1}^{k+m} a_j (b_j+g_j)  \geq t\big) \leq e^{-\lambda t } e^{2C\sum_{j=k+1}^{k+m}} \lambda^2 a_j^2.$$
Taking $\lambda =\lambda_0' \sqrt{k_0/m^3}$ we get 
 \begin{equation} \label{inechernoff} \PP\big( | \sum_{j=k+1}^{k+m} a_j (b_j+g_j) | \geq t\big) \leq C' e^{-t/C'\veps^{3/2}},\end{equation}
 where $C'$ is an absolute constant.
Note that $(k+m-j')(j'-j) = (k+m-j)(k+m-j') - (k+m-j)^2$. Therefore, applying the above inequality \eqref{inechernoff} alternatively with $a_j = (k+m-j)/\sqrt{j}$, $(k+m-j)^2/m\sqrt{j}$, and $m/\sqrt{j}$, we deduce using a union bound that for $n$ large enough
\begin{equation} \label{tailzeta2}\PP\big(|\tilde{\zeta}_m^{(2)}| \geq t\big) \leq C' e^{-\sqrt{t/C'\veps^3}},\end{equation}
where $C'$ is an  absolute constant.
Similarly as for $\zeta_m^{(2)}$, we can  write
$$ \xi_m^{(2)} = \sum_{j<j'} c_{j,j'} \frac{(b_j+g_j)( b_{j'}+g_{j'})}{\sqrt{jj'}} + r_m^{(2)},$$
where $\EE (r_m^{(2)})^2 \lesssim \veps^2 \kappa$ and  with $K = \{ j_1<\ldots<j_p\}$ and $j_0 = k+1$,
\[
 c_{j,j'} =  \sum_{K \subset  \{k+1,\ldots,k+m\}\setminus\{j,j'\}} \prod_{l\in K} \Big( c_l+z_l-2 \Big) \prod_{i=0}^p (j_{i+1}-j_{i}) = (k+m-j')(j'-j)(j-k-1) + O( \veps^2  \kappa m^3).
\]
Let
$$ \tilde{\xi}_m^{(2)} = \sum_{j,j'} (k+m-j')(j'-j)(j-k-1) \frac{(b_j+g_j)( b_{j'}+g_{j'})}{\sqrt{jj'}},$$
and
$$ \delta_m^{(2)} = \sum_{j<j'} \frac{(b_j+g_j)( b_{j'}+g_{j'})}{\sqrt{jj'}} O( \veps^2 \kappa m^3).$$
Similar computations as above show that for any $t>0$, and $n$ large enough
\begin{equation} \label{tailxi2} \PP( |\tilde{\xi}_m^{(2)}|\geq tm) \leq ce^{-\sqrt{t/c\veps^3}},\end{equation}
where $c$ is an absolute constant. Putting together \eqref{tailzeta2} and \eqref{tailxi2}, we get the claim \eqref{claim2} with $x_m^{(2)}= \tilde{\zeta}_m^{(2)} + \tilde{\xi}_m^{(2)}$.
Finally, we find that
$$ \EE (\delta_m^{(2)})^2 \lesssim \frac{m^2}{k_0^2} (\kappa \veps^5 k_0)^2=\kappa^2 \veps^{10}m^2, \ \EE (\eta_m^{(2)})^2 \lesssim \veps^{10} \kappa^2,$$
which yields the last claim \eqref{claim3} with 
$z_m = \delta_m^{(2)} + \eta_m^{(2)} +r_m^{(2)}+r_m^{(1)}$.

\noindent
 Fix now $t>0$. Using a union bound, we get
\begin{align*}
 \PP\big( |x_m|\leq t\max(|x|,|my|)\big)&\leq \PP\big( |a_m+ x_m^{(1)}|\leq 3 t\max(|x|,|my|)\big)\\
& + \PP\big(|x_m^{(2)}|\geq  t\max(|x|,|my|)\big) +\PP\big(|z_m|\geq  t\max(|x|,|my|)\big).\end{align*}
By \eqref{claim1}, \eqref{claim2} and \eqref{claim3}, we get
$$ \PP\big( |x_m|\leq t\max(|x|,|my|)\big)\leq Ct\veps^{-\frac{3}{2}} + Ce^{-c\sqrt{t/\veps^3}} +  \frac{C\kappa^2 \veps^{10}}{t^2}.$$
Choosing $t = \veps^{\frac{5}{2}+\delta}$ and using that $\veps \kappa<1$ and $\delta<1/2$ yields the lemma.
\end{proof}
We return to the proof  of Proposition \ref{tightlower}.
Let now $ k_0 - \lfloor\kappa k_0^{1/3}\rfloor= j_0 < j_1 <\ldots <j_N =  k_0 + \lfloor\kappa k_0^{1/3}\rfloor$ such that $\lfloor \veps k_0^{1/3}\rfloor \leq j_{i+1}-j_i \leq 2\lfloor \veps k_0^{1/3}\rfloor$. One can find such a sequence so that $c\kappa /\veps \leq N\leq c'\kappa/\veps$, where $c,c'$ are absolute constants. Let $(x_i,y_i)$ be such that
 $$ \prod_{j=j_0+1}^{j_i} \big( A_j + W_j\big) (u+\delta_{j_0} v) = x_i u +y_i v.$$
Denote by $r_i = \max(|x_i|,m|y_i|)$. By Lemma \ref{smallblock} we have for any $i$, 
$$ \PP_{\mathcal{F}_{j_i}}    ( r_{i+1} \leq \veps^{\frac{5}{2}+\delta} r_i ) \leq \PP_{\mathcal{F}_{j_i}}    ( |x_{i+1}| \leq \veps^{\frac{5}{2}+\delta} r_i )\lesssim \veps^{1+\delta}.$$
Thus, by induction we find
$$ \PP_{\mathcal{F}_{j_0}}( r_{N} \leq \veps^{N(\frac{5}{2}+\delta)} r_0)\lesssim N\veps^{1+\delta}.$$
 One can check that if $w= xu+yv$ for $x,y\in\RR$ then $\max(|x|,|y|)/\sqrt{2} \leq \|w\| \leq 3\max(|x|,|y|)$. Using that  $N\geq c \kappa /\veps$, we  conclude that
$$ \PP_{\mathcal{F}_{j_0}}\big( \| X_{k_0+\lfloor\kappa k_0^{1/3}\rfloor}\|  \leq (\veps^{c(\frac{5}{2}+\delta)\frac{\kappa}{\veps} }/3\sqrt{2}) \|X_{k_0-\lfloor\kappa k_0^{1/3}\rfloor}\|
\big) \leq C \kappa\veps^{\delta},$$
where $C$ is an absolute constant.
Taking $\veps$ such that $\veps^{\delta} = \kappa^{-2}/C$, we get the claim.
\end{proof}

\subsection{A change of basis}
\label{subsec-change}
\textit{In the rest of this section, we will denote by $\PP$ and $\EE$ the conditional probability and expectation $\PP_{\mathcal{G}_{k_0-l_0}}$ and $\EE_{\mathcal{G}_{k_0-l_0}}$, where $\mathcal{G}_{k_0-l_0}$ is generated by the variables $g_k, b_k$ for $k\leq k_0-l_0$, and assume that we are on the event where
\begin{equation} \label{cond-init-delta}|\delta_{k_0-l_0}|\leq C_\kappa k_0^{-1/3},\end{equation}
for some constant $C_\kappa>0$ depending on $\kappa$.}
Recall the transfer matrix $A_k$ from \eqref{def-Ak}. In the oscillatory regime we now consider, the deterministic part of the transition matrix has complex conjugated eigenvalues which are almost on the unit circle. To better understand the dynamics of the recursion, we will switch to the basis in which this deterministic transition matrix is a rotation. As we will see, this change of variable will have two consequences: it will increase the variance of the noise, and due to the fact that the change of basis is not constant over time, it will induce a certain drift.

We note that $z_k = z\sqrt{n/k}= 2-l/k_0+o(l/k_0)$.
 In the regime where $l\geq 1$, the eigenvalues of $A_{k_0+l}$ are complex 
 conjugates, and denoted
$\lambda_l$ and $\overline{\lambda_l}$, with
\begin{equation}
\label{eq-lambdadef}
 \lambda_l = \rho_l\frac{w_l + i \sqrt{ 4 -w_l^2}}{2}=:\rho_l e^{i\theta_l}, \quad \mbox{\rm where\, $\rho_l  = \sqrt{1-c_{k_0+l}}$\; and\; $w_l = z_{k_0+l}/\rho_l$.}
 \end{equation}
Therefore, up to a multiplicative constant, $A_{k_0+l}$ is the conjugate of  a multiple of a rotation matrix. Explicitly,  if we let for any $l\geq 1$,
\begin{equation} \label{defQ} Q_l =  \left( \begin{array}{cc} \frac{w_{l} }{2} & - \frac{\sqrt{4-w_{l}^2}}{2} \\
1 & 0
\end{array} \right), \ Q_l^{-1} = \left( \begin{array}{cc} 0 & 1 \\ \frac{-2}{\sqrt{4-w_{l}^2}}  & \frac{w_{l}}{\sqrt{4-w_{l}^2}}
\end{array} \right),\end{equation}
then we get the following fact.
\begin{fact}\label{newbasis} For any $l\geq 1$,
$$ A_{k_0+l} = \rho_{l} Q_l R_l Q_l^{-1}, \ W_{k_0+l} =   \rho_{l} Q_l \big(\hat W_l+V_l\big)  Q_l^{-1}$$
where
\begin{equation} R_l =\frac{1}{2}\left( \begin{array}{cc}
w_{l} & -\sqrt{4 -w_{l}^2}\\
\sqrt{4 -w_{l}^2} & w_{l}
\end{array}\right) ,\ \label{defW}\hat W_l =  \left( \begin{array}{cc}
0 & 0 \\
\sqrt{\frac{2}{l}}\zeta_l & 0
\end{array}\right),\end{equation}
 $\zeta_l, V_l$ are centered independent random variables (matrices)  such that $\Var (\zeta_l )=v$ and $\EE\|V_l\|_2^2= O(1/k_0)$. Further,  there exist $C,\lambda_0$ independent of $l$ such that
\begin{equation} \label{boundlaplacezeta} \EE (e^{\lambda_0 |\zeta_l| }) \leq C.\end{equation}
\end{fact}
In the description above, $\hat{W_l}+V_l$ is an explicit linear combinaison of the variables $b_{k_0+l}$ and $g_{k_0+l}$,
\[ \hat{W_l} +V_l= \frac{1}{\sqrt{k_0+l}} \begin{pmatrix} 0 & 0 \\
 \frac{2}{\sqrt{4-w_l^2}}\big( \frac{w_l}{2} b_{k_0+l} + g_{k_0+l}\big)& - b_{k_0+l} \end{pmatrix}.\]
The variable $\zeta_l$ is defined such that $\sqrt{2/l} \zeta_l e_2 e_1^T$ corresponds to the leading term of the above sum. In the sequel, we 
will not 
need the exploit the explicit expression of $\zeta_l$, but only the fact that it constitutes a family of independent random variables with variance $v$ and whose exponential moments are uniformly bounded as stated in \eqref{boundlaplacezeta}.

For any $l\geq l_0$, set
\begin{equation} \label{changebase} Y_l = \frac{r_{l}}{\psi_{k_0-l_0}}Q_{l}^{-1}X_{k_0+l},\end{equation}
where $r_l = \prod_{j=l_0+1}^{l}\rho_j^{-1}$, see \eqref{eq-lambdadef}, and let
$\mathcal{F}_l$ denote
the $\sigma$-algebra generated by the variables $g_k,b_k$ for $k\leq k_0+l$.
An important consequence of Propositions \ref{firstblock} and \ref{tightlower} is the fact that $\log \| Y_{l_0}\|$ is controlled when $\kappa \to \infty$. Recall our convention in this section concerning $\EE$ and $\PP$.
\begin{lemma}\label{controlYl0}There exists a deterministic constant $c_\kappa>0$ such that
$$\PP\big(c_\kappa^{-1}\leq  \| Y_{l_0}\| \leq c_\kappa \big) \leq \kappa^{-1}$$
\end{lemma}
\begin{proof}
We start with the upper bound. Recall that we assumed that $|\delta_{k_0-l_0}|\leq C_\kappa k_0^{-1/3}$. Thus, from Lemma \ref{firstblock} and \eqref{cond-init-delta}, we know that
$$ \frac{X_{k_0+l_0}}{\psi_{k_0-l_0}} =x_{l_0}u+y_{l_0}v,$$
where $\EE ( x_{l_0}^2 ) \leq c_\kappa$ and $\EE (y_{l_0}^2)\leq c_\kappa k_0^{-\frac{2}{3}}$, with $c_\kappa$ a $\kappa$-dependent constant.  Since $z_k = 2- O( \frac{k-k_0}{k_0})$, we find
$$\| Q_{l_0}^{-1} u\| = O(1), \ \| Q_{l_0}^{-1}\| = O( k_0^{\frac{1}{3}}).$$
Since $c_k =O(1/k)$, there exists  an absolute constant $c$
  such that $c^{-1} \leq |r_{l_0}| \leq c$. We deduce that $\EE (\| Y_{l_0}\|^2)\leq c_{\kappa}'$, with $c_\kappa'$  a $\kappa$-dependent
constant.  For the lower bound, note that as $\| Q_{l_0}\| \leq c$, we have that for any $t>0$,
$$ \PP\big( \|Y_{l_0}\| \leq t\big)  \leq \PP\big( \|X_{k_0+l_0}\| \leq c^2t |\psi_{k_0-l_0}|\big).$$
Using Proposition \ref{tightlower}  and \eqref{cond-init-delta}, we get the claim.
\end{proof}
For any $l \geq l_0+1$, we have the recursion:
$$Y_{l} =\big(R_l+\hat W_l+V_l\big)  Q_{l}^{-1}Q_{l-1} Y_{l-1}.$$
It turns out that
the matrices of change of basis matrices $Q_l$ are slowly varying 
and induce a drift in the recursion, which we describe in the following lemma.
\begin{lemma}\label{decompbaseQ}For any $l\geq l_0+1$,
$$ \big(R_l+\hat W_l+V_l\big) Q_{l}^{-1}Q_{l-1}=  R_l +\hat W_l + \Delta_l + B_l,$$
where $B_l$ are independent  random matrices such that $\|\EE B_l\| = O(1/\sqrt{k_0l})$, $\EE\|B_l - \EE B_l\|^2 = O(1/l^2)$, and
\begin{equation}\label{defDelta}  \Delta_l = \left( \begin{array}{cc}
0 & 0 \\
0 & -\frac{1}{2l}\end{array}\right).\end{equation}
\end{lemma}
\begin{proof}
We have
$  \big(R_l+\hat W_l\big) Q_{l}^{-1}Q_{l-1} =  R_l +N_l + P_l,$
where
$ N_l =  \hat W_lQ_{l}^{-1} Q_{l-1}, \ P_l = (Q_{l}^{-1} Q_{l-1} -I_2)R_l.$
We compute for $l_0 \leq l \leq k_0/\kappa^{-1}$:
\begin{equation} \label{qinvq}Q_{l}^{-1}Q_{l-1}= \left( \begin{array}{cc}
1 & 0 \\
\frac{w_{l}-w_{l-1}}{\sqrt{4-w_{l}^2}} & \sqrt{\frac{4-w_{l-1}^2}{4-w_{l}^2}}
\end{array}\right)= \left( \begin{array}{cc}
1 & 0 \\
0 & 1-\frac{1}{2l}
\end{array}\right) +O\Big( \frac{1}{\sqrt{k_0l}}\Big),\end{equation}
where we used  that $ c_{k_0+l}-c_{k_0+l-1} =O(1/k_0^2)$ and $l\geq l_0$.
Since $R_l = I_2+ O( l/k_0)$, we deduce that
$$ P_l =\Delta_l + O\Big( \frac{1}{\sqrt{k_0 l}}\Big), \ N_l = \hat W_l +H_l,$$
where $H_l$ is a centered random matrix with $\EE \|H_l\|^2 = O(1/l^2)$.  
Finally, as $V_l$ is centered and $\EE \|V_l\|^2 = O(k_0^{-1})$, we observe that $V_l Q_l^{-1} Q_{l-1}$  can be absorbed in the error term $B_l$. This   
completes the proof of the claim.
\end{proof}

\subsection{A time change}
\label{sec-timechange}
 The idea is to study the recursion $Y_k$ at stopping times $l_i$ which will correspond to the times the recursion returns to the direction $(1,0)$. The choice of this specific direction is only made for computational convenience. As we will show, the advantage of making this time change is that we will be able to write down a \textit{closed} recursion for the norm of $Y_{l_i}$ (see Proposition \ref{devepsilon}), and prove that the variance accumulated in the $i^{\text{th}}$-block $[l_i,l_{i+1}]$ is of order $1/j_i$, with $j_i$ as in
\eqref{eq-ip} below. We continue to work under the convention introduced in the beginning of Section \ref{subsec-change}.
\subsubsection{A linearization lemma} Our main tool to control the mean and variance collected in a given block is a  general linearization lemma, which gives an estimate on the error induced by the first order expansion of a product of random matrices of size $2\times 2$. To this end, recall that, given a sequence of $2\times 2$ matrices $C_1,\ldots, C_m$, for any $i<j$, we denote
$$ C_{i,j} = C_j C_{j-1}\ldots C_{i+1},$$
and $C_{i,j} =I_2$ for $i\geq j$.
\begin{lemma}\label{linear}
Let $m\geq 1$. Let  $C_k$, $k=1,\ldots,m$, be deterministic matrices such that for some absolute constant $\bar c\geq 1$,
$\|C_{i,j}\|\leq \bar c$ for all $1\leq i\leq j\leq m$.
Let $S_j \in \mathcal{M}_{2\times 2}$ be independent random matrices such that
$$ \forall 1\leq j\leq m,\quad \|\EE S_j\| \leq  \frac{1}{L}, \quad \EE \| S_j-\EE S_j\|^2 \leq \frac{1}{L},$$
where $L\geq 1$.
Then,
\begin{equation}
\label{eq-expand}
\prod_{j=1}^{m} \big(C_j + S_j\big)  = C_{0,m} + \sum_{j=1}^{m} C_{j,m}S_j C_{0,j-1}  +Z,
\end{equation}
where, for some absolute constant $C$,
$$ \|\EE Z\| \leq \bar c \Big(e^{\frac{ C\bar c m}{L} } - 1- \frac{C \bar c m}{L}\Big),\quad \EE \| Z - \EE Z\|^2\leq e^{\frac{C{\bar c}^2 m}{L}}\Big( e^{\frac{C {\bar c}^2 m}{L}} - 1- \frac{C {\bar c}^2 m}{L}\Big).$$
\end{lemma}
\begin{proof}
Expanding the product $\prod_{j=1}^{m} \big(C_j + S_j\big)$, we get from \eqref{eq-expand} that
$$Z = \sum_{p=2}^{m}\sum_{1\leq  l_1<\ldots <l_p\leq m}\Big( \prod_{i=1}^p C_{l_i, l_{i+1}-1} S_{l_i} \Big) C_{0,l_1-1},$$
where in the inner sum, $l_{p+1}=m+1$.
Thus,
$$
\| \EE Z\| \leq  \sum_{p=2}^{m}\sum_{1\leq l_1<\ldots <l_p\leq m} {\bar c}^{p+1 }\Big(\frac{1}{L}\Big)^p \leq \bar{c}\big( e^{\frac{\bar c m}{L}} -1 - 
\frac{\bar c m}{L}\Big).$$
To compute the second moment of
$\|Z-\EE Z\|$, we denote for any $j$, $\hat S_j = S_j- \EE S_j$. For any $I,J$ disjoint subsets of $\{1,\ldots, m\}$, we set
$$ M_{I,J} =\Big(\prod_{i=1}^p C_{l_i, l_{i+1}-1} M_{l_i} \Big) C_{l,l_1-1},$$
where $l_{p+1}=m+1$, $I\cup J = \{l_1<\ldots<l_p\}$, and
$$ M_l =\begin{cases}
 \hat S_l &\text{ if } l \in I,\\
\EE S_l & \text{ if } l \in J.
\end{cases}$$
With this notation, we have
$$ Z - \EE Z =  \sum_{p=2}^{m}\sum_{K\subset \{1,\ldots,m\} \atop |K|=p} \sum_{K = I\sqcup J\atop |I|\geq 1 }M_{I,J},$$
where $K=I \sqcup J$ means that $(I,J)$ is a partition of $K$.
Since the $\hat S_{l_i}$ are independent and centered random variables,
$$ \EE \tr \big(M_{I,J}^T M_{I', J'}\big) =0,$$ unless $I=I'$. Thus,
$$ \EE \| Z - \EE Z \|^2 = \sum_{I \subset \{1,\ldots, m\} \atop |I|\geq 1}\sum_{J \subset I^c\atop |I\cup J| \geq 2  }\sum_{J' \subset I^c \atop |I\cup J'|\geq 2}  \EE\tr(M_{I,J} M_{I,J'}^T).$$
 But, since $\EE \|\hat S_l  \|^2 \leq 1/L$ and $\|\EE S_l\|\leq 1/L$,
$$ \EE \tr \big( M_{I,J}^T M_{I,J'}) \leq  4^{2|I| +|J|+|J'|} {\bar c}^{(2|I|+|J|+|J'|+2)} \Big( \frac{1}{L}\Big)^{|I| +|J|+| J'|}.$$
Therefore,
$$\EE \| Z - \EE Z \|^2 \leq \sum_{I \subset \{1,\ldots,m\} \atop |I|\geq 1} \Big(\frac{C{\bar c}^2}{L}\Big)^{|I|}\Big(\sum_{J\subset I^c\atop |I\cup J|\geq 2} \Big(\frac{C \bar c }{L}\Big)^{|J|}\Big)^2,$$
where $C>0$ is an absolute constant. Using that 
$$\sum_{J\subset I^c\atop |I\cup J|\geq 2} \Big(\frac{C\bar c }{L}\Big)^{|J|} 
\leq  \Big( e^{\frac{C \bar c m}{L}}-1\Big) \Car_{|I|=1} + e^{\frac{C \bar c m}{L}} \Car_{|I|\geq 2},$$
we conclude that 
\begin{align*}
\EE \| Z - \EE Z \|^2 & \leq \frac{C {\bar c}^2 m}{L}\Big( e^{\frac{C \bar c m}{L}}-1\Big)^{2} + e^{\frac{2C \bar c m}{L}}\Big(  e^{\frac{C {\bar c}^2 m}{L}}-1 - \frac{C {\bar c}^2 m}{L}\Big)\\
&  \leq 2 e^{\frac{2C\bar c m}{L}}\Big(  e^{\frac{C{\bar c}^2 m}{L}}-1 - 
\frac{C{\bar c}^2 m}{L}\Big),
\end{align*}
where in the last inequality 
we used the fact that for any $x\geq 0$,
$ (e^x -1)^2 \leq 2 e^{x}(e^x-1-x)$. Modifying $C$ if needed, 
we obtain the claim.
\end{proof}

\subsubsection{Initialization} We will set the first block by choosing the first time $l_1$ the rotation $R_{l_0,l_1}$ turns the vector $Y_{l_0}$ as close as possible to the direction $(1,0)$. As the noise accumulated in this block will only have a marginal influence, the angle of $Y_{l_1}$ will be close as well to zero.  

 For any $l<l'$, we denote by $T_{l,l'}$ the matrix
\begin{equation} \label{defT} T_{l,l'} = \prod_{j=l+1}^{l'} \big(R_j+\hat W_j+V_j\big) Q_{j+1}^{-1}Q_j = \prod_{j=l+1}^{l'} \big( R_j +\hat W_j + \Delta_j + B_j\big),\end{equation}
where $\Delta_j$, $B_j$ are defined in Lemma \ref{decompbaseQ}.
We define a second block $(l_0, l_1)$ chosen so that the output vector $T_{l_0,l_1} Y_{l_0}$ is, up to a small error, in the direction $e_1=(1,0)$.
\begin{lemma}\label{init}
There exists $\kappa_0$ so that for all  $\kappa>\kappa_0$ and all $n$ large enough, there exists a  $\mathcal{F}_{k_0+l_0}$-measurable
random integer $l_1$  such that $l_0\leq l_1 \leq l_0 + 8\pi \kappa^{-2/3}l_0$ a.s., and for any $\eta>0$,
\begin{equation}
\label{eq-t1}
 T_{l_0,  l_1}Y_{l_0}/\|Y_{l_0}\| =t_1(1,\veps_1)^T,
 \;\;
 \mbox{\rm where}\;\;
 \EE(t_1-1)^2 \lesssim \kappa^{-\frac{3}{2}}, \ \mbox{\rm and} \ 
 \PP( |\veps_1| \geq \kappa^{-\frac{3}{4}+\eta}) \lesssim \kappa^{-2\eta}.
\end{equation}
\end{lemma}
\begin{proof}We define
$$ l_1 = \inf \big\{ l \geq l_0 : |\langle R_{l_0,l}Y_{l_0},(0,1)\rangle| \leq 2 \sqrt{l_0/k_0} \| Y_{l_0} \| \big\}.$$
Let $\theta$ denote the argument, in $[0,2\pi)$,  of $Y_{l_0}$ viewed as a two-dimensional vector. Note that if $\theta \leq 2\sqrt{l_0/k_0}$, the claim holds with $l_1=l_0$, since in that case $T_{l_0,l_1} = I_2$, $t_1=\cos(\theta) = 1 + O(l_0/k_0)$, and $\veps_1 = \tan(\theta) = O(\sqrt{l_0/k_0})$. In the following we assume that $\theta >2\sqrt{l_0/k_0}$. Let $\theta_l \in [0,2\pi)$ be the angle of the rotation $R_l$, see \eqref{eq-lambdadef}. We set
$$ l' = \inf \big\{ l\geq l_0 : \theta +\sum_{j= l_0+1}^{l} \theta_{j}   \geq 2\pi\big\}.$$
Since $\sin \theta_l = \frac{1}{2} \sqrt{4-w_{l}^2}$, we deduce that for
$\kappa$ and $n$ large enough and any $1 \leq l\leq  k_0/\kappa$,
we have that
\begin{equation} \label{encadrtheta} \sqrt{l/4 k_0} \leq \theta_l \leq \frac{3}{2}\sqrt{l/k_0}.\end{equation}
Therefore,
\begin{equation}\label{bornel'} l' \leq l_0+ \lceil 2(2\pi - \theta)\sqrt{k_0/l_0}  \rceil\leq l_0+8\pi \kappa^{-\frac{3}{2}}l_0,\end{equation}
where we used the fact that $l_0 =\lfloor \kappa k_0^{1/3}\rfloor$.
By the definition of $l'$, we have that $ 2\pi \leq \theta +\sum_{j= l_0+1}^{l'} \theta_{j}  \leq 2\pi+ \theta_{l'}$,
and
using \eqref{bornel'} and \eqref{encadrtheta}, we get that for $\kappa$ large enough,
$0\leq  \sin(\theta_{l'})\leq \theta_{l'} \leq 2\sqrt{{l_0}/{k_0}}.$
 Therefore,
$$0\leq  \langle R_{l_0,l'}Y_{l_0},(0,1)\rangle\leq \sin (\theta_{l'})  \| Y_{l_0} \|\leq 2\|Y_{l_0}\|\sqrt{l_0/k_0}.$$
Thus, $l_1 \leq l'$.
From Lemma \ref{decompbaseQ}, we have that for any $l\geq l_0$,
$T_{l_0,l} = \prod_{j=l_0+1}^l \big( R_j + S_j\big),$
where $S_j$ are independent matrices such that
$$\| \EE S_j \| =O(1/l_0), \ \EE \| S_j - \EE S_j \|^2 =O( 1/l_0).$$
By Lemma \ref{linear} and the bound \eqref{bornel'}, we deduce that
$ T_{l_0,l_1} = R_{l_0,l_1} + S,$
where $\| \EE S\| =O(\kappa^{-3/2})$ and $\EE \| S-\EE S\|^2 = O(\kappa^{-3/2})$. Therefore,
$$ \EE \| S Y_{l_0}/\|Y_{l_0}\|\|^2=O(\kappa^{-3}),$$
and by the definition of $l_1$, we have that
$$ \| R_{l_0,l_1}Y_{l_0}/ \| Y_{l_0}\|  - e_1 \| =O\Big( \sqrt{l_0/k_0} \Big)= O\Big(\sqrt{\kappa }k_0^{-1/3}\Big),$$
which yields that
$ T_{l_0,l_1} Y_{l_0}/\| Y_{l_0}\| = e_1+ z_1,$
where $\EE \| z_1\|^2 =O(\kappa^{-3/2})$. The first estimate on $t_1$  in \eqref{eq-t1}
 is then straightforward. On the other hand, $\veps_1 = z_1(2)/(1+z_1(1))$. We  write:
$$ \PP( |\veps_1| \geq \kappa^{-3/4 +\eta}) \leq \PP\Big( |z_1(2)| \geq \kappa^{-3/4+\eta}/2\Big)+\PP\Big( z_1(1)\leq -1/2\Big).$$
Using that $\EE \|z_1\|^2 =O(\kappa^{-3/2})$, we get the second part of \eqref{eq-t1}.
\end{proof}
\subsubsection{Decomposition along a block} We will later consider the evolution
at properly chosen
stopping times $l_i$, and therefore we 
need
to understand the effect of the transition matrices
along a block between two different times.
Let $l\geq l_0$ be a stopping time w.r.t the filtration $(\mathcal{F}_j)_{j\geq 0}$ and $m\geq 1$, an $\mathcal{F}_l$-measurable random integer. Using the notation from Lemma \ref{decompbaseQ},
by Lemma 
\ref{linear} we have that
\begin{equation} \label{decompositionT} T_{l,l+m} = S_{l,l+m} + V_{l,l+m} +
  D_{l,l+m} + Z_{l,l+m},\end{equation}
where
\begin{equation}
  \label{eq-VR}
  S_{l,l+m} = R_{l,l+m} + \sum_{j=l+1}^{l+m} R_{j,l+m}\Delta_{j}R_{l,j-1}, \quad   V_{l,l+m} = \sum_{j=l+1}^{l+m} R_{j,l+m} \hat W_j R_{l,j-1},
\end{equation}
\begin{equation}
  \label{eq-Z}
  D_{l,l+m} =   \sum_{j=l+1}^{l+m} R_{j,l+m} B_j R_{l,j-1},\end{equation}
and, for an absolute constant $C$,
  \begin{equation} \label{eqZ}  \EE_{\mathcal{F}_l} ||Z_{l,l+m}|| \lesssim (m/l)^2e^{Cm/l}, \quad \EE_{\mathcal{F}_l} ||Z_{l,l+m}||^2 \lesssim (m/l)^2e^{Cm/l}.
\end{equation}
In the following lemma we compute the leading order of the transition matrix 
$T_{l,l+m}$, along a block where the rotation $R_{l,l+m}$ is close to the identity, by identifying the leading term in the drift $S_{l,l+m}$, and in noise term $V_{l,l+m}$. We use the notation $a=b[2\pi]$ for ``$a$ equals $b$ modulo $2\pi$''.
\begin{lemma}\label{decomprandomblock} Let $l$ be a stopping time w.r.t the filtration $(\mathcal{F}_j)_{j\geq 1}$ and let $m$ be
a $\mathcal{F}_l$- measurable random integer such that $ l+m \leq k_0/\kappa $. Let $\delta \in (-\pi,\pi]$ such that $\delta = \sum_{j=l+1}^{l+m} \theta_j [2\pi]$.
Then,
$$ S_{l,l+m} =
\Big(1- \frac{m}{4l}\Big)I_2 + \left(\begin{array}{cc} -\delta^2/2&-\delta\\
\delta& -\delta^2/2
\end{array}
\right)+\tilde{S}_{l,l+m},$$
\begin{equation}
  \label{eq-Gdef}
  V_{l,l+m} =:\frac{1}{2}\sqrt{\frac{v}{l}}\left( \begin{array}{cc}g_{1}
 & -\sqrt{3} g_{2}\\
\sqrt{3} g_{3} &  -g_{1}
\end{array}\right) + \tilde G_{l,l+m}=:G_{l,l+m}+\tilde G_{l,l+m},
\end{equation}
where $g_{1},g_{2},g_{3}$ are such that conditionally on $\mathcal{F}_l$, $g_{1}$ is decorellated from $(g_{2}, g_{3})$ and
$$  \EE_{\mathcal{F}_l}( g_i) = 0\; \mbox{\rm and} \; \ \EE_{\mathcal{F}_l} (g_i^2) = m, i=1,2,3,
 \quad \EE_{\mathcal{F}_l} (g_{2} g_{3}) = m/3,$$
and  $\tilde G_{l,l+m}$ is a centered random matrix,  such that
$$ \EE_{\mathcal{F}_l }\| \tilde G_{l,l+m}\|^2 \lesssim \Big( \delta^2 + \frac{l}{k_0}+\sqrt{\frac{m}{k_0}}+\frac{m}{l}\Big) \frac{m}{l}+ \sqrt{\frac{k_0}{l^3}} |\delta|,$$
and
\begin{equation} \label{claimSD} \| \tilde{S}_{l,l+m} \| \lesssim \Big( |\delta| + \sqrt{\frac{l}{k_0}}+ \sqrt{\frac{m}{k_0}}+\frac{m}{l}\Big) \frac{m}{l}+ \sqrt{\frac{k_0}{l^3}} |\delta| + |\delta|^3,
\quad \EE_{\mathcal{F}_l} \|D_{l,l+m}\|^2  \lesssim \frac{m}{l^2} + \frac{m^2}{k_0l}.\end{equation}
Finally, there exist absolute constants $\lambda_0,C>0$ so that for any $|\lambda|\leq \lambda_0$,
\begin{equation}\label{controlexpoG}  \EE_{\mathcal{F}_l} e^{\lambda \sqrt{l}\|G_{l,l+m}\|_{\infty}}\leq 4e^{C\lambda^2 m }, \ \EE_{\mathcal{F}_l} e^{\lambda \sqrt{l}\|\tilde{G}_{l,l+m}\|_{\infty}}\leq 4e^{C\lambda^2 m }.\end{equation}
\end{lemma}
\begin{proof}
Let $U_\alpha$ denote the rotation matrix with angle $\alpha$,
$$ U_\alpha =\left( \begin{array}{cc}
\cos \alpha & -\sin \alpha\\
\sin \alpha & \cos \alpha
\end{array}\right).$$
Since $R_k = U_{\theta_k}$, see \eqref{eq-lambdadef},
we have that
\begin{equation} \label{devR} R_{l,l+m} = U_\delta =I_2 + \left(\begin{array}{cc} -\delta^2/2&-\delta\\
\delta& -\delta^2/2
\end{array}
\right) + O( |\delta|^3).\end{equation}
 For any $j\geq l$,
\begin{equation} \label{approxR} R_{l,j-1} = U_{\alpha_{l,j}}U_{\theta_j}^{-1},\end{equation}
where $\alpha_{l,j} = \sum_{p=l+1}^{j} \theta_p$, and thus
\begin{equation} \label{approxR2} R_{j,l+m} = U_{\delta}U_{\alpha_{l,j} }^{-1}.\end{equation}
Using the fact that $\| \EE B_j\| =O(1/\sqrt{k_0l})$ and $\EE \| B_j-\EE B_j\|^2 =O(1/l^2)$, we deduce that
$$\| \EE_{\mathcal{F}_l} D_{l,l+m}\| = O\Big( \frac{m}{\sqrt{k_0 l}}\Big), \ \EE_{\mathcal{F}_l} \|D_{l,l+m}- \EE_{\mathcal{F}_l} D_{l,l+m}\|^2 =   
O\Big(\frac{m}{l^2}\Big).$$
Using \eqref{approxR}, \eqref{approxR2} 
and 
the fact that $U_\delta = I_2 +O(\delta)$, $U_{\theta_l} =I_2 +O(\sqrt{l/k_0})$, $\hat{W_j}$ are centered independent random variables with variance $O(1/l)$ for $j\geq m$, we find that
\begin{align}
 V_{l,l+m} &= \sum_{j=l+1}^{l+m} U_\delta U_{\alpha_{l,j}}^{-1} \hat{W_j} U_{\alpha_{l,j}}U_{\theta_j}^{-1} =: \sum_{j=l+1}^{l+m} U_{\alpha_{l,j}}^{-1} \hat{W_j} U_{\alpha_{l,j}} + H_{l,l+m} \nonumber \\
& = \sum_{j=l+1}^{l+m}\frac{\sqrt{2}\zeta_j}{\sqrt{j}} \left( \begin{array}{cc}
 \sin \alpha_{l,j} \cos \alpha_{l,j}   &-\sin^2 \alpha_{l,j}   \\
 \cos^2 \alpha_{l,j}   &  -\sin \alpha_{l,j} \cos \alpha_{l,j}
\end{array}\right) + H_{l,l+m}, \label{Vdecomp}
\end{align}
where 
\begin{equation} \label{poubelle1}\EE_{\mathcal{F}_{l}} H_{l,l+m} =0,  \ \EE_{\mathcal{F}_{l}} \|H_{l,l+m}\|^2 \lesssim (\delta^2 +l/k_0)(m/l).\end{equation}
Similarly, using \eqref{devR} we get 
\begin{equation} \label{decom:S} S_{l,l+m} = I_2 + \left(\begin{array}{cc} -\delta^2/2&-\delta\\
\delta& -\delta^2/2
\end{array}
\right) - \sum_{j=l+1}^{l+m}\frac{1}{2j} \left( \begin{array}{cc}
\sin^2 \alpha_{l,j} & \sin \alpha_{l,j} \cos \alpha_{l,j}\\
\sin \alpha_{l,j}\cos \alpha_{l,j} & \cos^2 \alpha_{l,j}
\end{array}\right) + Y_{l,l+m},\end{equation}
where $Y_{l,l+m}$ is a matrix such that
\begin{equation} \label{poubelle}  \| Y_{l,l+m}\| \lesssim  \Big( |\delta| + \sqrt{\frac{l}{k_0}}\Big) \frac{m}{l} + |\delta|^3.\end{equation}
We compute the leading term in \eqref{Vdecomp} which we denote by $G_{l,l+m}$, while putting into  $\tilde{G}_{l,l+m}$ all the remainder terms, and similarly for $S_{l,l+m}$. We will give a detailed proof for the expansion of the $(1,1)$-coefficient of $S_{l,l+m}$, the computations for the other coefficients and the covariance between the coefficients of $V_{l,l+m}$ being similar and yielding the same error terms. More precisely, we will show that
\begin{equation} \label{eq:claim} \sum_{j=l+1}^{l+m} \frac{\sin^2\alpha_{l,j}}{j} = \frac{m}{2l}+O\Big(\Big( \frac{m}{l}\Big)^2\Big)+O\Big( \sqrt{\frac{m^3}{k_0l^2} }\Big)+O\Big( \sqrt{\frac{k_0}{l^3}} \delta\Big).\end{equation}
Once proven, we obtain using \eqref{Vdecomp} and \eqref{poubelle1}, the claimed estimate on the error $\tilde{G}_{l,l+m}$, similarly, using \eqref{decom:S} and \eqref{poubelle} it yields the estimate on the error $\tilde{S}_{l,l+m}$.

To see \eqref{eq:claim}, note that
on the one hand,
 \begin{equation} \label{eq:1}\sum_{j=l+1}^{l+m} \frac{\sin^2 \alpha_{l,j}}{j} = \sum_{j=l+1}^{l+m} \frac{\sin^2 \alpha_{l,j}}{l} + O\Big(\Big( \frac{m}{l}\Big)^2\Big).\end{equation}
On the other hand,
\begin{equation}\label{eq:2} \sum_{j=l+1}^{l+m} \sin^2\alpha_{l,j} = \frac{m}{2}- \frac{1}{2} \sum_{j=l+1}^{l+m} \cos (2\alpha_{l,j}).\end{equation}
Let $\hat \theta = \frac{1}{m} \sum_{j=l+1}^{l+m}\theta_j$. Let $j \in\{ l+1,\ldots, l+m\}$. From \eqref{eq-lambdadef}
we have that
$$\sin(\theta_j) = \frac{1}{2} \sqrt{4 - w_j^2}=\sqrt{\frac{j}{j+k_0} +O\Big( \frac{1}{k_0}\Big)}.$$
We deduce that for $\kappa$ and $n$ large enough, for any $j \leq k_0/\kappa$,  we have that $\theta_j\leq \pi/3$ .
Thus, for any $l+1 \leq j, j' \leq l+m$,
$$ |\theta_{j'} -\theta_j |\leq 2 |\sin(\theta_j) - \sin(\theta_{j'})| \lesssim \sqrt{\frac{m}{k_0}}.$$
Therefore,
\begin{equation} \label{cont} \sum_{j=l+1}^{l+m}| \cos (2\alpha_{l,j}) -   \cos (2\hat \theta (j-l))|=O\Big( \sqrt{\frac{m^3}{k_0}}\Big).\end{equation}
But,
\begin{equation}\label{sum-cos}\sum_{j=l+1}^{l+m} \cos (2\hat \theta (j-l)) = \cos\big(\hat \theta (m+1)
\big) \frac{\sin( \hat \theta m)}{\sin (\hat \theta)} = O\Big( \sqrt{\frac{k_0}{l}} \delta\Big),\end{equation}
where we used the fact that $(3/2) \sqrt{l/k_0}\geq \hat \theta\geq (1/2) \sqrt{l/k_0}$ by \eqref{encadrtheta}.
Putting together \eqref{eq:claim}, \eqref{eq:1}, \eqref{cont} and \eqref{sum-cos}, we get \eqref{eq:claim}. For the last claim \eqref{controlexpoG}, observe that any coefficient of $V_{l,l+m}$ can be written as $\sum_{j=l+1}^{l+m} a_j \zeta_j/\sqrt{j}$ where $\|a\|_\infty \leq r$ and $r$ is a numerical constant. As a consequence of \eqref{boundlaplacezeta}, the second derivative of the moment generating function of $\zeta_j$ is bounded on $[-\lambda_0/2,\lambda_0/2]$. Since $\EE \zeta_j =0$ and $\EE \zeta_j^2=v$, it follows that $\log \EE e^{\theta \zeta_j} \leq c \theta^2$ for any $|\theta|\leq \lambda_0/2$. Moreover, as the estimate \eqref{boundlaplacezeta} is uniform in $j$, the constant $c$ is uniform in $j$. By the independence of the variables $\zeta_j$, the claim \eqref{controlexpoG} follows.
\end{proof}
\subsubsection{Definition of the blocks}
\label{sec-defblocks}
With $l_1$ as in Lemma \ref{init}, write  $l_1= \nu k_0^{1/3}$, where $\nu$ is an $\mathcal{F}_{l_0}$-measurable random variable such that $\kappa \leq \nu \leq 2\kappa$ almost surely. Let $\tau \in (1/4,2/5)$, set $i_0= \lfloor k_0^{\tau}\rfloor-\lceil\nu^{3/2}\rceil$, $i_1 =\lfloor  k_0^{(1-\tau)/3}\rfloor $,  $h_0 = i_1 -\lceil\sqrt{\kappa}\rceil$, and  define  the deterministic blocks $\hat l_i$ as
$$ \hat l_i =\begin{cases}
 \lfloor k_0^{\frac{1}{3}}(i-1+\nu^{\frac{3}{2}})^{\frac{2}{3}}\rfloor & \text{ if } 1\leq i \leq i_0,\\
 \lfloor k_0/(i_0+i_1-i)^2\rfloor  & \text{ if } i_0<i \leq i_0 + h_0.
\end{cases}$$
The definition of the deterministic blocks $\hat l_i$ -- in particular, the choice of 
 $\tau \in (1/4,2/5)$ -- is motivated by Fact \ref{comparaison}, which collects the important aspects of the sequence $\hat{l}_i$.

 Let $t_0 = i_0+h_0$. A first simple observation is the following.
\begin{fact}
$(\hat l_i)_{1\leq i\leq t_0}$ is an increasing sequence of integers.
\end{fact}
\begin{proof}
We only need to check that $\hat l_{i_0+1} \geq  \hat l_{i_0}$, i.e. to check that
$k_0^{\frac{1}{3}} (i_0-1+\nu^{\frac{3}{2}})^{\frac{2}{3}} \leq k_0/i_1.$
Since $i_0 -1 +\nu^{3/2} \leq k_0^\tau$, it suffices to check that $ k_0^{(1+ 2\tau)/3} \leq k_0/i_1$,
that is $i_1 \leq k_0^{(1-\tau)/3}$, which is implied by the definition.
\end{proof}
To ease the notation of the blocks $\hat{l}_i$, set
\begin{equation}
  \label{eq-ip}
  j_i=\left\{ \begin{array}{ll}
    i-1+\nu^{3/2},&  1\leq i\leq i_0\\
    i_0+i_1-i,&
i_0< i\leq t_0.
\end{array}\right.
\end{equation}
Let $n_{\kappa}= \hat l_{t_0} $ and  $\Delta \hat l_i=\hat l_{i+1}- \hat l_i$.

 \begin{fact}\label{comparaison} For $\kappa$  and $n$ large enough,
     uniformly in $ i\in \{1,\ldots,t_0\}$,
\begin{itemize}
 \item[$(i).$]
$\Delta \hat l_i/\hat l_i\asymp 1/j_i.$
\item[$(ii).$]   $ \hat l_i/k_0 =   O\big(j_i^{-a} \big),$
where $a = \frac{2}{3}\big( \frac{1}{\tau}-1\big)\in (1,2)$.
\item[$(iii).$] $\sqrt{{k_0}/{\hat l_i^3}}=O(1/j_i).$
\end{itemize}
\end{fact}
Note that the bound $(ii)$ is very crude and does not reflect the true behaviour of $\hat{l}_i/k_0$. As announced at the beginning of the section, $1/j_i$ will correspond to the order of the variance accumulated in the $i^{\text{th}}$-block. From \eqref{eq-lambdadef}, one can see that $\theta_l$ is of order $(\hat l_i/k_0)^{1/2}$ in the block $[\hat l_i,\hat l_{i+1}]$. The condition $(ii)$  above ensures that the discretization error coming from the rotation is smaller than the standard deviation of the noise accumulated in a block.
For any $l\leq l'$, we define
\begin{equation}
\label{eq-alphallp}
 \alpha_{l,l'} = \sum_{j = l+1}^{l'} \theta_j.
 \end{equation}
Let $y_1 = T_{l_0, l_1}(Y_{l_0}/\|Y_{l_0}\|)$.
Starting from $l_1$ and $y_1$, we define recursively stopping times $l_i$ and the vectors $y_i = t_i(1,\veps_i)^T$. Assume  that we have
constructed the stopping times $l_1,l_2,\ldots, l_i$ and vectors $v_1,v_2,\ldots, v_i$.
If $l_i<n_{\kappa}$ and $|\veps_i|\leq 1/2$,   we define
\begin{equation} \label{defli} l_{i+1} = \inf \Big\{ l \geq l_i+\Delta \hat l_i :|\veps_i + \delta_{l_i,l}-\frac{\delta_{l_i,l}^2}{2}\veps_i |\leq 6 \sqrt{l_i/k_0}\Big\},\end{equation}
where for any $l\geq 1$,  $\delta_{l_i,l} \in (-\pi,\pi]$ is such that 
$\delta_{l_i,l} = \alpha_{l_i, l} [2\pi]$, and
\begin{equation}\label{defyi} t_{i+1}(1,\epsilon_{i+1})^T=y_{i+1} = T_{l_i,l_{i+1}} y_{i}.\end{equation}
 If $l_i \geq n_{\kappa}$ or $|\veps_i|>1/2$ then we set $l_{i+1} = l_i$ and $y_{i+1} =y_i$. For any $i$, we denote
   $\Delta l_i = l_{i+1} -l_i$. In view of Lemma \ref{decomprandomblock} , one can describe equivalently $l_{i+1}$ as corresponding to the earliest time the rotation $R_{l_i,l_{i+1}}$ brings $y_{i}$ as close as possible to the direction $(1,0)$. The condition $l_{i+1} \geq l_i + \Delta \hat l_i$ is taken so that enough variance is accumulated in the $i^{\text{th}}$-block.
\begin{proposition}\label{estimli}
  For any $i\geq 1$, $l_i$ is measurable on $\mathcal{F}_{l_{i-1}}$. Further, 
  for $\kappa$ large enough,
 if $l_i<n_\kappa$ and $|\veps_i|\leq 1/2$ then
  \begin{equation} \label{upperbound} 0 \leq \Delta l_i-\Delta \hat l_i  \leq \lfloor 10\pi \sqrt{k_0/l_i}\rfloor,\end{equation}
and
\begin{equation} \label{deltaestim}\delta_{l_i, l_{i+1}} = -\veps_i + O\Big(\sqrt{l_i/k_0}\Big) +O\big(\veps_i^3\big).\end{equation}
Further, uniformly in $i$,
  \begin{equation}
    \label{eq-Deltacomp}
    \Delta l_i/ l_i=O\Big(\Delta \hat l_i/\hat l_i\Big).
  \end{equation}
\end{proposition}
\begin{proof}
Assume that $l_i <n_{\kappa}$ and $|\veps_i|\leq 1/2$. From the definition of the stopping times $l_{i'}$, this means that for any $i' \leq i$, we have $l_{i'} <n_\kappa$ and $|\veps_{i'}|<1/2$. Therefore,  $l_{i'}$ is defined by \eqref{defli}  for any $i'\leq i$. As $l_1= \hat{l}_1$, we have $l_i \geq \hat{l}_i$.

We  prove first that $l_{i+1}$ is well-defined by \eqref{defli} and show \eqref{upperbound}; the definition then implies that
$l_{i+1}$ is $\mathcal{F}_{l_i}$-measurable. To this end, let
$$ l = \min \big\{ l' \geq l_i: \sum_{j = l_i+1}^{l'} \theta_j \geq 2 \pi(  k_i + \Car_{ \veps_i <0}) + s_{l_i}\big\},$$
where $k_i = \lceil \frac{1}{2\pi}\sum_{j = l_i+1}^{l_i + \Delta \hat l_i} \theta_j \rceil$, and
\begin{equation}\label{encadalpha} s_{l_i} = \veps_i^{-1} \Big(1- \sqrt{ 1+2\veps_i^2}\Big).\end{equation}
We will show that $l\in\mathcal{A}_i:=
\Big\{ l \geq l_i+\Delta \hat l_i :| \veps_i+\delta_{l_i,l}-\delta_{l_i,l}^2 \veps_i/2 |\leq 6 \sqrt{l_i/k_0}\Big\}$.
From the definition of $l$, we have
\begin{equation} \label{encadrement} 2\pi (k_i+\Car_{\veps_i<0}) +s_{l_i}\leq\sum_{j = l_i+1}^{l} \theta_j\leq 2\pi(k_i+\Car_{\veps_i<0}) + s_{l_i} + \theta_l.\end{equation}
As we observed in \eqref{encadrtheta}, for $\kappa$ and $n$ large enough, we have for any $1 \leq j \leq n_{\kappa}$ that
\begin{equation} \label{encadtheta} \sqrt{ j/4k_0} \leq \theta_j \leq 2 \sqrt{ j/k_0}.\end{equation}
From the estimate above, if we let $l'=l_i + \Delta \hat l_i + 
  \lfloor 10\pi\sqrt{ k_0/l_i}\rfloor$, then
  \[ \sum_{j=l_i +1}^{l'} \theta_j \geq 2\pi (k_i-1) + 5 \pi \geq 2\pi(k_i+ \Car_{\veps_i<0}) +s_{l_i},\]
where we used the fact that $|s_{l_i}|\leq |\veps_i| \leq 1/2$. Therefore, by definition of $l$
\begin{equation}\label{upperboundli} l \leq l_i + \Delta \hat l_i + 
  \lfloor 10\pi\sqrt{ k_0/l_i}\rfloor.\end{equation}
Since $\Delta \hat l_i \leq \hat l_i \leq l_i$ and $\sqrt{k_0/l_i}= o_\kappa (l_i)$, by Fact \ref{comparaison} (iii), we get for $\kappa$ large enough that
$l\leq 3l_i$, and thus
\begin{equation}\label{boundthetal} \theta_l \leq 4 \sqrt{ l_i/k_0}.\end{equation}
Therefore, we deduce from \eqref{encadalpha} that
$$2\pi( k_i+\Car_{\veps_i <0}) +s_{l_i}\leq\sum_{j = l_i+1}^{l} \theta_j\leq 2\pi (k_i+\Car_{\veps_i <0}) + s_{l_i} + 4 \sqrt{ l_i/k_0}.$$
Since $|s_{l_i}| + 4\sqrt{l_i/k_0} <2\pi$  for $\kappa$ large enough, we have that $ s_{l_i} \leq \delta_{l_i,l} \leq s_{l_i} +4 \sqrt{l_i/k_0}$. As $ \veps_i+ s_{l_i}-s_{l_i}^2 \veps_i/2=0,$
we deduce that
$$| \veps_i+\delta_{l_i,l} - \delta_{l_i,l}^2 \veps_i/2 |\leq 6\sqrt{ l_i/k_0},$$
where we used the facts that $|\veps_i| \leq 1/2$ and that
$|s_{l_i} + 4\sqrt{l_i/k_0}|\leq 1$
for $\kappa$ large enough.  This shows that $l\in \mathcal{A}_i$, and \eqref{upperboundli} yields the claimed upper bound \eqref{upperbound}. To prove \eqref{deltaestim}, observe that we already proved that
\begin{equation}
  \label{eq-already}
  \alpha_{l_i,l_{i+1}}\leq 2\pi ( k_i+\Car_{\veps_i<0}) + s_{l_i} + 4\sqrt{l_i/k_0}.
\end{equation}
To prove a lower bound on $\alpha_{l_i, l_{i+1}}$, we will show a lower bound on $l_{i+1}$. We set
$$ l' = \max \big\{ l'' \geq l_i: \sum_{j = l_i +1}^{l''} \theta_j \leq 2 \pi(k_i +\Car_{\veps_i<0})+ s_{l_i}-10\sqrt{l_i/k_0}\big\}.$$
 Note that $l' <l$ so that $\theta_{l'+1} \leq 4\sqrt{l_i/k_0}$ by \eqref{boundthetal}. Thus,
$$ s_{l_i} -14\sqrt{l_i/k_0}\leq \delta_{l_i,l'} \leq  s_{l_i} -10\sqrt{l_i/k_0}.$$
Therefore,
$$ \veps_i +\delta_{l_i,l'} - \delta_{l_i,l'}^2 \veps_i/2   < -6\sqrt{l_i/k_0},$$
for $\kappa$ large enough. Thus $l' <l_{i+1}$. Combined with \eqref{eq-already}, we obtain that
$$ 2\pi (k_i+\Car_{\veps_i<0}) +s_{l_i} -8\sqrt{l_i/k_0}\leq \alpha_{l_i,l_{i+1}} \leq  2\pi ( k_i+\Car_{\veps_i<0}) + s_{l_i} + 4\sqrt{l_i/k_0},$$
which ends the proof of \eqref{deltaestim}.

To see \eqref{eq-Deltacomp}, we note from
\eqref{upperbound} that
\[\Delta l_i/l_i=
  \Delta \hat l_i/l_i+O\Big(\sqrt{k_0/l_i^3}\Big)
  = O(\Delta \hat l_i/\hat l_i)+O\Big(\sqrt{k_0/\hat l_i^3}\Big)
 = O\Big(\Delta \hat l_i/\hat l_i\Big),
\]
where we used that $\hat l_i\leq l_i$ and Fact \ref{comparaison}.
\end{proof}
Combining Proposition \ref{estimli} with Fact \ref{comparaison} yields the following corollary.
\begin{corollary}\label{controlrandom}
For any $1 \leq i\leq t_0$,
\begin{equation} \label{controlrandomvariance} 
  \Delta l_i/l_i =O( 1/j_i), \quad  l_i/k_0 =O(j_i^{-a}),\end{equation}
where $a>1$ is as in Fact \ref{comparaison} (ii).
Moreover, if $l_i<n_\kappa$ and $|\veps_i|\leq 1/2$, then $l_i \geq \hat l_i$, and as a consequence
\begin{equation}  \label{controlrandom2}\sqrt{k_0/l_i^3} =O( 1/j_i).\end{equation}
\end{corollary}
\begin{proof}
Note that \eqref{eq-Deltacomp} is vacuously true  if  $l_i\geq n_\kappa$ or $|\veps_i|>1/2$, since in that case
$\Delta l_i = 0$. Combining \eqref{eq-Deltacomp} with Fact \ref{comparaison} (i), we obtain that for any $i\leq t_0$,
\[  \Delta l_i/l_i \lesssim 1/j_i,\]
which proves the first part of \eqref{controlrandomvariance}.
Let $p\leq i$. If $l_p<n_\kappa$ and $|\veps_p|\leq 1/2$, then combining \eqref{upperbound} and Fact \ref{comparaison} (i), (iii), we get
\begin{equation} \label{bornesupDelta}  \Delta l_p \lesssim \Delta \hat l_p.\end{equation}
If $l_p\geq n_\kappa$ or $|\veps_p|>1/2$, the last inequality is trivially true since $\Delta l_i =0$ in that case. Summing the last display  for $p\leq i-1$,
we get that
\begin{equation} \label{bornesupli} l_i \lesssim \hat l_i,\end{equation}
and using Fact \ref{comparaison} (ii) gives the second part of \eqref{controlrandomvariance}.

We finally prove \eqref{controlrandom2}. It will follow from Fact \ref{comparaison} (iii) and the fact that if $l_i<n_\kappa$ and $|\veps_i|\leq 1/2$ then $l_i \geq \hat l_i$. Indeed, by construction of the stopping times $l_i$, this means that $l_p<n_\kappa$ and $|\veps_p|\leq 1/2$ for any $p\leq i$. Therefore, for any $p\leq i$, $l_p$ is defined by \eqref{defli} and we have $\Delta l_p\geq \Delta \hat l_p$ for any $p\leq i$. Since $l_1 = \hat l_1$,  we have $l_i\geq \hat l_i$.
\end{proof}
In the next lemma, we prove that the definition of the stopping times $l_i$ implies that if $\veps_i$ is small enough, then $\veps_{i+1}$ is small as well, showing that the time change stabilizes the direction $e_1$ in the recursion.
We recall that
for a matrix $A$ we denote  $\|A\|_{\infty} = \max_{i,j} |A_{i,j}|$.
\begin{lemma}\label{recursionborne}There exists a constant $c_0>0$ such that for any  $\veps\leq c_0$ and
  $i\geq 1$,
if ,  $l_i<n_{\kappa}$, $|\veps_i| \leq 2\veps$ and $\| V_{l_i,l_{i+1}}+D_{l_i,l_{i+1}}+ Z_{l_i,l_{i+1}}\|_{\infty} \leq \veps/2$, then
$$|\veps_{i+1}|\leq 3\veps/4 + O(1/\sqrt{j_{i+1}}), \quad |1-(t_{i+1}/t_i)| \leq 3\veps/4 +O(1/j_{i+1}).$$
\end{lemma}
\begin{proof}Let $i\geq 1$ such that $l_i <n_\kappa$, $|\veps_i|\leq 2\veps$ and $\| V_{l_i,l_{i+1}}+D_{l_i,l_{i+1}}+  Z_{l_i,l_{i+1}}\|_{\infty} \leq \veps/2$.
We have $y_{i+1} = T_{l_i,l_{i+1}}y_i$. Therefore, writing $(\tilde{u}_i,\tilde{v}_i)^T =   \tilde{S}_{l_i,l_{i+1}} (1,\veps_i)^T$ and $M_{l_i,l_{i+1}}=V_{l_i,l_{i+1}}+D_{l_i,l_{i+1}} +Z_{l_i,l_{i+1}}$, we have that
\begin{align}
  \label{eq-firstco}
  \frac{t_{i+1}}{t_i} =
  1 - \frac{\Delta l_i}{4l_i} - \frac{\delta_{l_i,l_{i+1}}^2}{2} -\delta_{l_i,l_{i+1}} \veps_i +
  M_{l_i,l_{i+1}}(1,1) +
  M_{l_i,l_{i+1}}(1,2)\veps_i+\tilde{u_i},
\end{align}
where for any matrix $M$, we denote by $M(k,\ell)$ its $(k,\ell)$-entry.
Using the fact that $ |\veps_i -\delta_{l_i,l_{i+1}} -\delta_{l_i,l_{i+1}}^2 \veps_i /2|=O(\sqrt{l_i/k_0})$
by the definition of $l_i$, we obtain
\begin{equation}
  \label{eq-secondco}
  \frac{t_{i+1}}{t_i} \veps_{i+1} =O\Big( \sqrt{\frac{l_i}{k_0}}\Big)- \frac{\Delta l_i}{4l_i}\veps_i+ M_{l_i,l_{i+1}}(2,1) + M_{l_i,l_{i+1}}(2,2) \veps_i+\tilde{v}_i.\end{equation}
By Proposition \ref{estimli} and  Fact \ref{comparaison}, we deduce that 
\begin{equation}\label{controldelta}  |\delta_{l_i,l_{i+1}}| \lesssim \veps +
  1/j_{i+1} .\end{equation}
But, by Lemma  \ref{decomprandomblock} (with $l=l_i$ and  $m=\Delta l_i$),  we have that
$$ \| \tilde{S}_{l_{i},l_{i+1}}\| \lesssim  \Big( |\delta_{l_i,l_{i+1}}| +\sqrt{\frac{l_i}{k_0}}+ \sqrt{\frac{\Delta l_i}{k_0}}\Big) \frac{\Delta l_i}{l_i}  +\sqrt{\frac{k_0}{l_i^3}} |\delta_{l_i,l_{i+1}}| + |\delta_{l_i,l_{i+1}}|^3.$$
Using  \eqref{controldelta} and Corollary \ref{controlrandom},
we get
\begin{equation}
  \label{eq-normS}\| \tilde{S}_{l_{i},l_{i+1}}\| \lesssim |\veps|^3 + |\veps|/j_{i+1}.
\end{equation}
Equation \eqref{eq-firstco}
together with \eqref{eq-normS},
the assumption $\|M_{l_i,l_{i+1}}\|_{\infty}\leq
\veps/2$, and Corollary \ref{controlrandom}, yield
\begin{equation}\label{minoti} |t_{i+1}/t_i -1| \leq \veps/2 +O(\veps^2)  +O(1/j_{i+1}).
\end{equation}
On the other hand, by  Fact \ref{comparaison},
\eqref{eq-secondco} and \eqref{eq-normS}, we get
$$|t_{i+1}\veps_{i+1}/t_i| \leq \veps/2 +O(\veps^2) +
O(1/j_{i+1}^{a/2})\leq\veps/2 +O(\veps^2) + O(1/\sqrt{j_{i+1}}).$$
Using \eqref{minoti} and $\veps^2\leq c_0\veps$,
 the claim follows.
\end{proof}
Recall the notation introduced in Lemma \ref{decomprandomblock}. The next lemma allows for the propagation of errors from block to block.
\begin{lemma}
 \label{cor-1}
For any $i\leq t_0$,
 \begin{equation}
 \label{eq-gammablock}
 \PP_{\mathcal{F}_i}\Big(\|G_{l_i,l_{i+1}}\|_\infty+\|\tilde{G}_{l_i,l_{i+1}}\|_\infty+
 \|D_{l_i,l_{i+1}}\|_\infty+\|Z_{l_i,l_{i+1}}\|_\infty>j_{i+1}^{-5/12}
 \Big)\leq  Cj_{i+1}^{-7/6},
 \end{equation}
where $C$ is an absolute positive constant.
 \end{lemma}
The choice of the exponent $5/12$ 
will be justified in the proof of Lemma \ref{erreur}, 
where it will be crucial to have a control on the noise 
of the order of $j_i^{-c}$ with $c\in (1/3,1/2)$.
 \begin{proof}Observe first that the statement is vacuously true if $|\veps_i|>1/2$ or $l_i =n_\kappa$, since in that case $\Delta l_i =0$. Assume now that $|\veps_i|\leq 1/2$ and $l_i<n_\kappa$. By Lemma \ref{decomprandomblock} and \eqref{eqZ}, we have that
 \[ \EE_{\mathcal{F}_i} \| D_{l_i,l_{i+1} }\|^2_{\infty} + \EE_{\mathcal{F}_i} \|Z_{l_i,l_{i+1}}\|_\infty^2 \lesssim \frac{\Delta l_i}{l_i^2}+\frac{(\Delta l_i)^2}{k_0l_i}+\Big(\frac{\Delta l_i}{l_i}\Big)^2\lesssim \Big(\frac{\Delta l_i}{l_i}\Big)^2,\]
where we used the fact that $l_i\leq k_0$. From Corollary \ref{controlrandom}, we know that $\Delta l_i /l_i \lesssim j_{i+1}^{-1}$ and $l_i\geq \hat l_i$. Thus,
 \begin{equation} \label{moment2DZ}  \EE_{\mathcal{F}_i} \| D_{l_i,l_{i+1} }\|^2_{\infty} +\EE_{\mathcal{F}_i} \|Z_{l_i,l_{i+1}}\|_\infty^2 \lesssim 1/ j_{i+1}^2.\end{equation}
Therefore, by Markov's inequality,
\begin{equation}\label{boundZ} \PP_{\mathcal{F}_i}( \| D_{l_i,l_{i+1} }\|_{\infty}+\|  Z_{l_i,l_{i+1}}\|_\infty  \geq j_{i+1}^{-5/12}/2)
 \lesssim j_{i+1}^{-7/6}.\end{equation}
Let now $X = \| G_{l_i, l_{i+1}}\|_\infty$ or $\| \tilde{G}_{l_i, l_{i+1}}\|_\infty$.  By Lemma \ref{decomprandomblock},  there exist absolute constants $C,\lambda_0>0$ such that for any $|\lambda|\leq \lambda_0$,
\[ \EE_{\mathcal{F}_{l_i}} e^{\lambda \sqrt{l_i}X}\leq 4e^{ C\lambda^2 \Delta l_i}.\]
Therefore, there exist absolute constants $c,r>0$ such that for any $|t| \leq r \Delta l_i/\sqrt{l_i}$ we have
\[ \PP( X>t) \leq4e^{-ct^2 l_i/\Delta l_i}.\]
We claim that for $n$ large enough, $j_{i+1}^{-5/12} \leq 4r \Delta l_i/\sqrt{l_i}$. Provided this is true, we can conclude the proof.
Indeed, Corollary \ref{controlrandom} yields
\[ \Delta l_i/l_i \gtrsim \Delta \hat l_i/\hat l_i\gtrsim 1/j_{i+1}.\]
Therefore, there exists an absolute constant $c'>0$ such that
\[ \PP\big( X>j_{i+1}^{-5/12}/4\big) \lesssim e^{- c' j_{i+1}^{1/6}}\lesssim j_{i+1}^{-7/6}.\]
Using \eqref{boundZ} and a union bound ends the proof.
We are left with proving
 the claim that $j_{i+1}^{-5/12} \leq 4 r \Delta  l_i /\sqrt{l_i}$ a.s for $n$ large enough. Using \eqref{bornesupli} and \eqref{bornesupDelta}, we obtain that
$\Delta l_i/\sqrt{l_i} \gtrsim \Delta \hat l_i/\sqrt{\hat l_i}.$
Since $\Delta \hat l_i /\hat l_i \gtrsim 1/j_{i+1}$ by Fact \ref{comparaison} (i), we have that
$\Delta l_i/\sqrt{l_i} \gtrsim \sqrt{\hat l_i}/j_{i+1}$. It remains to check that $j_{i+1}^{7/6} = o(\hat l_i)$. When $i\leq i_0$, this is equivalent to say that $j_i^{7/6}=o(k_0^{1/3}j_i^{2/3})$, that is $j_i =o(k_0^{2/3})$. This latter condition is true when $i\leq i_0$ as $i_0\leq k_0^\tau$ and $\tau<2/5$. In the regime where $i\geq i_0$, one obtains similarly that $j_{i+1}^{7/6} =o(\hat{l}_i)$, using the fact that $\tau >1/4$.
\end{proof}

We next introduce recursively ``good events'', on which the control on block length propagates.
Define $\mathcal{E}_1 = \{ |\veps_1|\leq j_1^{-5/12} \}$ and,
for any $i\geq 2$,
\begin{equation}
\label{eq-Ei}
\mathcal{E}_i 
= \mathcal{E}_1 
\bigcap \bigcap_{q=1}^{i-1}
\big\{ \big( \| G_{l_{q},l_{q+1}}\|_{\infty}+\| \tilde{G}_{l_{q},l_{q+1}}\|_{\infty}+\|D_{l_{q},l_{q+1}}\|_{\infty}+\|Z_{l_{q},l_{q+1}}\|_{\infty}\big) \leq
j_{q+1}^{-5/12}/2
\big\}.
\end{equation}
Note that with this choice, the sets $\mathcal{E}_i$ are nondecreasing, that is, for any $i\geq 1$, $\mathcal{E}_{i+1}\subset \mathcal{E}_i$.
Applying Lemma \ref{recursionborne} to $\varepsilon=j_l^{-5/12}$, for $\kappa$ large enough it follows by induction that on the event $\mathcal{E}_i$ that for any $1\leq l\leq i$,  $|\veps_l|\leq j_l^{-5/12}$. As a consequence of Lemma \ref{cor-1},
we have the following lemma, that states that with high probability, all events $\mathcal{E}_i$ occur simultaneously.
\begin{lemma}\label{cotnrolbruit}Let $\mathcal{E}_\infty = \bigcap_{i \geq 1} \mathcal{E}_i$. Then,
$\PP( \mathcal{E}_\infty) \underset{\kappa \to \infty}{\longrightarrow} 1.$
\end{lemma}
\begin{proof}
We have that 
\begin{align*}
&\PP\Big( \bigcup_{i\geq 1} \mathcal{E}_i^c\Big)\leq\PP(\mathcal{E}_1^c) +\sum_{q=1}^{i_0+i_1}\PP(\mathcal{E}_{q}\cap \mathcal E_{q+1}^c)\\
&\leq\PP(\mathcal{E}_1^c) +\sum_{q=1}^{i_0+i_1}  \PP\Big( ( \| G_{l_{q},l_{q+1}}\|_{\infty}+\| \tilde{G}_{l_{q},l_{q+1}}\|_{\infty}+\|D_{l_{q},l_{q+1}}\|_{\infty}+\|Z_{l_{q},l_{q+1}}\|_{\infty}>j_{q+1}^{-5/12}/2\mid \mathcal{E}_q\Big).
\end{align*}
Observe that $j_1^{-5/12} \geq (2\kappa)^{-5/8}$. By Lemma \ref{init}, we get $\PP( \mathcal{E}_1^c) =o_\kappa(1)$ when $\kappa \to +\infty$. 
Using  Lemma \ref{cor-1},
the sum on the right hand side of the last display is bounded above by
$O(\sum_{i=1}^{i_0+i_1} j_{i+1}^{-7/6}) = O( \sum_{ j\geq \kappa^{3/2}} j^{-7/6}) \leq C_\kappa,$
where $C_\kappa\to_{\kappa\to\infty} 0.$
\end{proof}

In the sequel, we write $\PP_i$ and $\EE_i$ for probabilities and expectations conditioned on $\mathcal{F}_{l_i}$. The next proposition uses the events $\mathcal{E}_i$ to get a control on the growth of the norms $t_i$.
\begin{proposition}\label{increm}There exists $b>1$ such that for any $i\geq 2$,
 \begin{equation}\label{devepsilon}
 \veps_{i} = G_{l_{i-1},l_{i}}(2,1) + \eta_{i},\end{equation}
where $G_{l_{i-1},l_{i}}$ is as in \eqref{eq-Gdef} and
$\EE_{i-1}( \eta_{i}^2 \Car_{\mathcal{E}_i}) \lesssim j_i^{-b}$. Further,
\begin{equation} \label{devti}t_{i+1}/t_{i} = 1 +(3v/8)\cdot (\Delta l_{i-1}/l_{i-1})- \Delta l_i/4l_i +G_{l_{i},l_{i+1}}(1,1) + \xi_{i+1},
\end{equation}
where  $|\EE_{i-1} (\xi_{i+1} \Car_{\mathcal{E}_{i}})| \lesssim j_i^{-b}$ and $\EE_{i-1}( \xi_{i+1}^2 \Car_{\mathcal{E}_{i}}) \lesssim j_i^{-b}$.
\end{proposition}

\begin{proof}We will first prove \eqref{devepsilon}.
Let $i\geq 1$. On the event $\mathcal{E}_i$, $l_{i+1}$ is well-defined by \eqref{defli} and $y_{i+1} =t_{i+1}(1,\veps_{i+1})^T$ by \eqref{defyi}. Using the decomposition \eqref{decompositionT} of the transition matrix $T_{l_i,l_{i+1}}$ and Lemma \ref{decomprandomblock}, we obtain that on the event $\mathcal{E}_i$,
$$ t_{i+1}/t_i = 1 - \Delta l_i/4l_i - \delta_{l_i,l_{i+1}}^2/2 -\delta_{l_i,l_{i+1}} \veps_i +G_{l_i,l_{i+1}}(1,1) +G_{l_i,l_{i+1}}(1,2)\veps_i+u_i.$$
and 
\[ \veps_{i+1}(t_{i+1}/t_i) =\veps_i+\delta_{l_i,l_{i+1}} -  \veps_i\delta_{l_i,l_{i+1}}^2/2- \veps_i \Delta l_i/4l_i+ G_{l_i,l_{i+1}}(2,1) + G_{l_i,l_{i+1}}(2,2) \veps_i+v_i,\]
where $(u_i,v_i)^T = (\tilde{S}_{l_i,l_{i+1}} + \tilde{G}_{l_i,l_{i+1}} + D_{l_i,l_{i+1}} +Z_{l_i,l_{i+1}})(1,\veps_i)^T$ (see \eqref{eq-VR}, \eqref{eq-Z} and \eqref{eq-Gdef}
for definitions). By the definition \eqref{defli} of the stopping time $l_i$,  we have that
\begin{equation}\label{devepsi} \veps_{i+1}  t_{i+1}/t_i  =O\Big( \sqrt{l_i/k_0}\Big)- \veps_i\Delta l_i/4l_i+ G_{l_i,l_{i+1}}(2,1) + G_{l_i,l_{i+1}}(2,2) \veps_i+v_i.\end{equation}
Using \eqref{deltaestim} of Proposition \ref{estimli}, we conclude that
$$  \delta_{l_i,l_{i+1}} \veps_i + \delta_{l_i,l_{i+1}}^2/2 = {-\veps_i^2/2} +O\Big( \sqrt{l_i/k_0}|\veps_i|\Big) +O(\veps_i^4),$$
which yields that on the event $\mathcal{E}_i$,
\begin{equation}
\label{eq-newstar}
 \frac{t_{i+1}}{t_i} = 1 - \frac{\Delta l_i}{4l_i} +\frac{3\veps_i^2}{2}+G_{l_i,l_{i+1}}(1,1) +G_{l_i,l_{i+1}}(1,2)\veps_i +O\Big( \sqrt{\frac{l_i}{k_0}}|\veps_i|\Big) +O(\veps_i^4)+u_i.
 \end{equation}
Using that on the event $\mathcal{E}_{i+1}$ we have  that $|\veps_i|\lesssim j_i^{-5/12}$, and
\[\| G_{l_i,l_{i+1}}\|_\infty +  \| \tilde{G}_{l_i,l_{i+1}}\|_\infty +\| Z_{l_i,l_{i+1}}\|_\infty  \leq j_{i+1}^{-5/12},\]
we get, as a direct consequence of Lemma \ref{recursionborne}, that  on the event $\mathcal{E}_{i+1}$,
\begin{equation}\label{controlti}
\Big|\frac{t_{i+1}}{t_i} -1\Big|\lesssim  j_{i+1}^{-5/12}.
\end{equation}
Using Corollary \ref{controlrandom}, \eqref{controlrandomvariance}, and the fact that $|\veps_i|\lesssim j_i^{-1/4}$ on $\mathcal{E}_i$, we get that on $\mathcal{E}_i$,
\[ O\Big( \sqrt{l_i/k_0}\Big) - \veps_i \Delta l_i/4l_i  = O\Big( j_i^{-a/2}\Big),\]
where $a>1$ is as in Fact \ref{comparaison} (ii).
We deduce from \eqref{devepsi} that on $\mathcal{E}_i$,
$$ \frac{t_{i+1}}{t_i} \veps_{i+1} =G_{l_i,l_{i+1}}(2,1) + \tilde \veps_{i+1},$$
where
$$\tilde{\veps}_{i+1} =O\big( j_i^{-a/2 }\big) + G_{l_i,l_{i+1}}(2,2) \veps_i+v_{i}.$$
Using \eqref{controlti}, we deduce that on $\mathcal{E}_{i+1}\subset \mathcal{E}_i$,
$$ \veps_{i+1} = G_{l_i,l_{i+1}}(2,1) +\eta_{i+1},$$
where
$$ \eta_{i+1} = \frac{t_i}{t_{i+1}} \tilde{\veps}_{i+1} - O\Big( \frac{t_{i+1}}{t_i} -1\Big) G_{l_i,l_{i+1}}(2,1).$$
In order to control the error term $\tilde{\veps}_{i+1}$, we need the following.
\begin{lemma}\label{erreur}On the event $\mathcal{E}_i$,
$$  | \EE_{i} u_{i+1}|  +| \EE_{i}
 v_{i+1}|  \lesssim j_i^{-5/4}, \quad \EE_{i}(u_{i+1}^2+ v_{i+1}^2)
 \lesssim  j_i^{-5/4}.$$
\end{lemma}
\begin{proof}[Proof of Lemma \ref{erreur}.] Recall that $(u_{i+1},v_{i+1})^T = (\tilde{S}_{l_i,l_{i+1}} + \tilde{G}_{l_i,l_{i+1}} +D_{l_i,l_{i+1}}+ Z_{l_i,l_{i+1}})(1,\veps_i)^T$.
Since $l_{i+1}$ is $\mathcal{F}_{l_i}$-measurable, we have by Lemma  \ref{decomprandomblock} that $\EE_{i} (\tilde{G}_{l_i,l_{i+1}})=0$, and therefore
\[ \EE_{i}\Big(  \tilde{G}_{l_i,l_{i+1}}(1,\veps_i)^T \Big) =0.\]
By Lemma \ref{decomprandomblock} and Proposition \ref{estimli}, we have  that on $\mathcal{E}_i$,
\[ \EE_i \| D_{l_i,l_{i+1}} \| \leq \Big( \EE_i \| D_{l_i,l_{i+1}} \|^2\Big)^{1/2} \lesssim \frac{\sqrt{\Delta \hat l_i}}{\hat l_i} + \frac{\Delta \hat l_i}{\sqrt{k_0 \hat l_i}}.\]
Using Fact \ref{comparaison} and the definition of $\hat l_i$, one gets that 
\[ \EE_i \| D_{l_i,l_{i+1}} \| \lesssim j_i^{-5/4}.\]
Further,  by \eqref{eq-normS}, we have that on $\mathcal{E}_i$, $ \| \tilde{S}_{l_{i},l_{i+1}}\| \lesssim  j_i^{-5/3}$. (The exponent $5/12$ in the definition \eqref{eq-Ei} of the event $\mathcal{E}_i$ was precisely chosen to handle this error term).  Thus, by \eqref{eqZ}, we get that
\[ | \EE_{i}u_{i+1}| + | \EE_{i} v_{i+1}| \lesssim   \EE_{i} \big( \| D_{l_i,l_{i+1}}\|\big)+  \EE_{i} \big( \| Z_{l_i,l_{i+1}}\|\big)+\EE_{i} \big( \| \tilde{S}_{l_i,l_{i+1}}\|\big) \lesssim j_i^{-5/4}.\]
Using Lemma \ref{decomprandomblock}, we obtain that on the event $\mathcal{E}_i$,
$$ \EE_{i} \big( \| \tilde{G}_{l_i,l_{i+1}}\|^2\big)\lesssim \Big( \EE_{i} \big( \delta_{l_i,l_{i+1}}^2\big) +\frac{l_i}{k_0}+ \sqrt{\frac{\Delta l_i}{k_0}} \Big) \frac{\Delta l_i}{l_i} + \sqrt{\frac{k_0}{l_i^3}} \EE_{i}\big(|\delta_{l_i,l_{i+1}}|\big).$$
On the event $\mathcal{E}_i$, we have that $|\veps_i|\leq j_i^{-5/12}$ and therefore, by
\eqref{controldelta},
 $|\delta_{l_i,l_{i+1}}|\lesssim j_i^{-5/12}$. 
By Corollary \ref{controlrandom}, we have that $l_i/k_0 \lesssim j_i^{-a}$ with $a>1$, $\Delta l_i /l_i \lesssim j_i^{-1}$ and on the event $\mathcal{E}_i$, $\sqrt{k_0/l_i^3} \lesssim j_i^{-1}$. Thus, we deduce that on the event $\mathcal{E}_i$,
$ \EE_{i} \big( \| \tilde{G}_{l_i,l_{i+1}}\|^2\big)\lesssim j_i^{-5/4}.$
In addition, we have  by \eqref{moment2DZ} that
$ \EE_{i} \big( \| D_{l_i,l_{i+1}}\|^2+ \| Z_{l_i,l_{i+1}}\|^2\big) \lesssim j_i^{-2}$.
Thus, we conclude that on $\mathcal{E}_i$, as $|\veps_i| \lesssim 1$,
\begin{align*} \EE_{i}  (u_{i+1}^2 +v_{i+1}^2) &\lesssim   \EE_{i} \big( \| Z_{l_i,l_{i+1}}\|^2\big)+\EE_{i} \big( \| \tilde{S}_{l_i,l_{i+1}}\|^2\big)+\EE_{i} \big( \| \tilde{G}_{l_i,l_{i+1}}\|^2\big) \lesssim  j_i^{-5/4},
\end{align*}
which ends the proof.
\end{proof}
We continue with the proof of \eqref{devepsilon}.
Using Lemma \ref{erreur} and the fact that $|\veps_i|\leq 2/j_i^{5/12}$ on $\mathcal{E}_i$, we deduce that on $\mathcal{E}_i$,
$$ \EE_{i} \tilde{\veps}_{i+1}^2\lesssim j_i^{-a} + j_i^{-\frac{5}{6}}\EE_i G_{l_i,l_{i+1}}(2,2)^2 + j_i^{-5/4}.$$
But $\EE_i \|G_{l_i, l_{i+1}}\|^2 \lesssim \Delta l_i/l_i$ by Lemma \ref{decomprandomblock}. Using Corollary \ref{controlrandom}, we get that
\begin{equation} \label{controlG} \EE_i \|G_{l_i, l_{i+1}}\|^2 \lesssim j_i^{-1}.\end{equation}
Thus, we deduce that  on $\mathcal{E}_i$,
$$ \EE_{i} \tilde{\veps}_{i+1}^2\lesssim j_i^{-(5/4) \wedge a}.$$
As  $\mathcal{E}_{i+1}\subset \mathcal{E}_i$ by definition, we obtain that
\begin{align*}
&\EE_i \big[\Car_{\mathcal{E}_{i+1}}\eta_{i+1}^2 \big]\lesssim \EE_i\Big[\Car_{\mathcal{E}_{i+1} }\Big(\frac{t_i}{t_{i+1}} \tilde{\veps}_{i+1}\Big)^2\Big] + \EE_i\Big[ \Car_{\mathcal{E}_{i+1}} \Big(\frac{t_{i+1}}{t_i} -1\Big)^2 G_{l_i,l_{i+1}}(2,1)^2\Big]
 \lesssim j_i^{-(5/4) \wedge a},
\end{align*}
where we used \eqref{controlG}. This ends the proof of the first claim \eqref{devepsilon}.

We turn to the proof of \eqref{devti}. Assume that $i\geq 2$.  Recall that on $\mathcal{E}_{i}$,
\begin{equation}
\label{eq-recall}
 \frac{t_{i+1}}{t_i} = 1 - \frac{\Delta l_i}{4l_i} +\frac{3\veps_i^2}{2}+G_{l_i,l_{i+1}}(1,1) +G_{l_i,l_{i+1}}(1,2)\veps_i +O\Big( \sqrt{\frac{l_i}{k_0}}\veps_i\Big) +O(\veps_i^4)+u_{i+1}.
 \end{equation}
Let
$$ \tilde \xi_{i+1} = G_{l_i,l_{i+1}}(1,2)\veps_i +O\Big( \sqrt{\frac{l_i}{k_0}}\veps_i\Big) +O(\veps_i^4)+u_{i+1}.$$ We will first prove the following lemma.
\begin{lemma} With $a\in (1,2)$ as in Fact \ref{comparaison} (ii),
  \label{erreurbis}
$$ |\EE_{i-1}  \big(\Car_{\mathcal{E}_{i}}\tilde{\xi}_{i+1}\big) |\lesssim j_i^{-(a+1)/2},\quad \EE_{i-1}\big(\Car_{\mathcal{E}_{i}} \tilde{\xi}_{i+1}^2\big) \lesssim j_i^{-\frac{3}{2}}.$$
\end{lemma}
\begin{proof}[Proof of Lemma \ref{erreurbis}.]
Since $l_{i+1}$ is $\mathcal{F}_{l_i}$-meaurable, we know by Lemma \ref{decomprandomblock} that $\EE_{i}( G_{l_i,l_{i+1}} )=0$. As a consequence,
\begin{equation}\label{esperr1} \EE_{i-1}\Big( \Car_{\mathcal{E}_{i}} G_{l_i,l_{i+1}}(1,2) \veps_i \Big) =0.\end{equation}
Using \eqref{controlG} and the fact that $|\veps_i|\lesssim j_i^{-5/12}\lesssim j_i^{-1/4}$ on $\mathcal{E}_i$, we deduce that
\begin{equation}\label{err1} \EE_{i-1}\Big( \Car_{\mathcal{E}_{i}} G_{l_i,l_{i+1}}(1,2)^2 \veps_i^2 \Big) \lesssim j_i^{-\frac{3}{2}}.\end{equation}
Note that \eqref{devepsilon} and \eqref{controlG} yield that
\begin{equation} \label{controlvarepsilon}\EE_{i-1} \big(\Car_{\mathcal{E}_i} \veps_i^2\big) \lesssim j_i^{-1}.\end{equation}
Using the fact that on $\mathcal{E}_i$,  $l_i/k_0 \lesssim j_i^{-a}$ by Corollary \ref{controlrandom} and $|\veps_i|\lesssim j_i^{-5/12}$, we obtain that
\begin{equation} \label{err:2}  \EE_{i-1} \big(\Car_{\mathcal{E}_i} \veps_i^2 l_i/k_0\big) +\EE_{i-1} \big(\Car_{\mathcal{E}_i}\veps_i^8\big)\lesssim j_i^{-(a+1)} + j_i^{-\frac{10}{3} }\lesssim j_i^{-(a+1)},\end{equation}
where we used the fact that $a<2$, see Fact \ref{comparaison}(ii).
Putting together  \eqref{esperr1} and \eqref{err:2}, and using Lemma \ref{erreur}, we deduce that
\[  |\EE_{i-1} \Car_{\mathcal{E}_i} \tilde{\xi}_{i+1} | \lesssim j_i^{-\frac{1}{2}(a+1)}.\]
 On the other hand, combining together \eqref{err1}, \eqref{err:2} and Lemma \ref{erreur}, we have that
\[ \EE_{i-1} \big(\Car_{\mathcal{E}_i} \tilde{\xi}_{i+1}^2 \big)\lesssim j_i^{-\frac{3}{2}}+j_i^{-(a+1)}.\]
Using the fact that $a>1$ ends the proof.
\end{proof}
The last estimate needed is contained in the next lemma.
\begin{lemma}\label{secondmomenteps}
With notation as above,
$$ \veps_i^2 = \frac{3v}{4} \frac{\Delta l_{i-1}}{ l_{i-1}} + r_i,$$
where
$$| \EE_{i-1} ( \Car_{\mathcal{E}_i} r_i) | \lesssim j_i^{-a},\quad \EE_{i-1} (\Car_{\mathcal{E}_i} r_i^2) \lesssim j_i^{-\frac{3}{2}}.$$
\end{lemma}
\begin{proof}
Observe first that
$$  \EE_{i-1} \big( \Car_{\mathcal{E}_i} \big( \veps_i^2 -\EE_{i-1}( \veps_i^2\mid \mathcal{E}_{i}) \big)^2 \big) \leq 
\EE_{i-1} \big(\Car_{\mathcal{E}_{i}} \veps_i^4 \big).$$
Using \eqref{controlvarepsilon} and the fact that $|\veps_i|\lesssim j_i^{-1/4}$ on $\mathcal{E}_i$, we obtain that  $\EE_{i-1} \big( \Car_{\mathcal{E}_i} \veps_i^4\big) \lesssim j_i^{-3/2}$, and therefore
\[  \EE_{i-1} \big( \Car_{\mathcal{E}_i} \big( \veps_i^2 -\EE_{i-1}( \Car_{\mathcal{E}_{i}}\veps_i^2 ) \big)^2 \big) \lesssim j_i^{-\frac{3}{2}}.\]
It now remains to show that 
\begin{equation}\label{claimvar}\EE_{i-1}( \Car_{\mathcal{E}_{i}}\veps_i^2 ) = (3v/4)\cdot  (\Delta l_{i-1}/l_{i-1}) + O( j_i^{-a}).\end{equation}
By \eqref{devepsilon}, we know that $\veps_i = G_ {l_{i-1}, l_i}(2,1) + \eta_i$, where $\EE_{i-1}( \Car_{\mathcal{E}_{i}} \eta_i^2 )\lesssim j_i^{-a}$ with $a>1$. Since $\EE_{i-1} ( \Car_{\mathcal{E}_{i-1}} G_{l_{i-1}, l_i}(2,1)^2) \lesssim j_i^{-1}$ as we proved in \eqref{controlG}, we get
\begin{align} \EE_{i-1} \big( \Car_{\mathcal{E}_i} \veps_i^2\big) & =  \EE_{i-1} \big( \Car_{\mathcal{E}_i} G_{l_{i-1},l_i}(2,1)^2\big) +  2 \EE_{i-1} \big( \Car_{\mathcal{E}_i} G_{l_{i-1},l_i}(2,1)\eta_i\big)  + \EE_{i-1} \big( \Car_{\mathcal{E}_i} \eta_i^2\big) \nonumber \\
& =   \EE_{i-1} \big( \Car_{\mathcal{E}_i} G_{l_{i-1},l_i}(2,1)^2\big)  + O\big( j_i^{-(a+1)/2}\big)+ O\big( j_i^{-a}\big).\label{moment2}
\end{align}
 Since $a>1$, we get
\begin{equation} \label{varepsi} \EE_{i-1} \big( \Car_{\mathcal{E}_i} \veps_i^2\big) =  \EE_{i-1} \big( \Car_{\mathcal{E}_i} G_{l_{i-1},l_i}(2,1)^2\big) + O\big(j_i^{-a}\big).\end{equation}
Further, using  Lemma \ref{decomprandomblock}, we have that
\begin{equation}\label{varcomp}\EE_{i-1}\Car_{\mathcal{E}_{i-1}} G_{l_{i-1}, l_i}(2,1)^2=(3v/4)\cdot (\Delta l_{i-1}/ l_{i-1}).\end{equation}
On the other hand, using the Cauchy-Schwarz inequality, we get that
$$
\EE_{i-1} \big( \Car_{\mathcal{E}_{i-1}\setminus \mathcal{E}_i}G_{l_{i-1},l_i}(2,1)^2 \big)
\leq  \PP_{i-1} \big( \mathcal{E}_{i-1}\setminus \mathcal{E}_i \big)^{1/2} \EE_{i-1}\big( \Car_{\mathcal{E}_{i-1}}G_{l_{i-1},l_i}(2,1)^4\big)^{1/2}.
$$
From \eqref{controlexpoG}, we deduce that in particular
\[ \EE_{i-1} e^{ t \sqrt{l_i/\Delta l_i} |G_{l_{i-1}, l_i}(2,1)|} \leq C,\]
where $t, C$ are absolute constants, so that
\begin{equation} \label{moment4} \EE_{i-1} G_{l_{i-1},l_i}(2,1)^4 \lesssim  (\Delta l_i/l_i)^2\lesssim j_i^{-2},\end{equation}
where we used \eqref{controlrandomvariance} in the second inequality.
By Lemma \ref{cor-1}, we have that $\PP_{i-1} ( \mathcal{E}_{i-1}\setminus \mathcal{E}_i) \lesssim j_i^{-7/6}$. Therefore,
\[  \EE_{i-1} \big( \Car_{\mathcal{E}_i} G_{l_{i-1},l_i}(2,1)^2\big)  =  \EE_{i-1} \big( \Car_{\mathcal{E}_{i-1}} G_{l_{i-1},l_i}(2,1)^2\big)  +O( j_i^{-4/3}).\]
Using \eqref{varepsi} and \eqref{varcomp}, we deduce that
\[  \EE_{i-1} \big( \Car_{\mathcal{E}_i} \veps_i^2\big) = (3v/4)\cdot (\Delta l_{i-1}/ l_{i-1})+O( j_i^{-(4/3)\wedge a}),\]
which ends the proof of  \eqref{claimvar} and therefore of the lemma.
\end{proof}
Combining the last lemma with \eqref{eq-recall} and Lemma \ref{erreurbis} yields \eqref{devti} and completes
the proof of Proposition \ref{increm}.
\end{proof}
\begin{proposition}\label{convlawblock}
Let $i_* = \min \big\{ i \geq 1: l_i =l_{i+1}\big\}$ and denote by $\nu_{n,\kappa}^{(0)}$ the conditional law of
$$ \frac{\log \Big|\frac{t_{i_*}}{t_1}\Big|- \frac{v-1}{6}\log n}{\sqrt{\frac{v}{6}\log n }},$$
given $\mathcal{F}_{k_0-l_0}$.
Then,
$$ \lim_{\kappa \to \infty} \limsup_{n\to \infty} d(\nu_{n,\kappa}^{(0)},\gamma) =0.$$
\end{proposition}
\begin{proof}
We have
$$  \log |t_{i_*}/t_1| = \sum_{i=1}^{i_*-1} \log |t_{i+1}/t_{i} |.$$
By \eqref{controlti}, we deduce in particular that for any $i\geq 1$, on the event $\mathcal{E}_{i+1}$,
\begin{equation} \label{lowerboundtic}\Big| \frac{t_{i+1}}{t_i}- 1\Big|\leq c,\end{equation}
where $c<1$ is some absolute positive constant independent of $i$. Proposition \ref{increm} implies that for $i\geq 2$,
\[ \frac{t_{i+1}}{t_i} =1+ \frac{3v}{8}\frac{\Delta l_{i-1}}{l_{i-1}}-\frac{\Delta l_i}{4l_i}+G_{l_{i},l_{i+1}}(1,1)+ \xi_{i+1},\]
where $|\EE_{i-1}\big( \Car_{\mathcal{E}_{i}}  \xi_{i+1} \big)|\lesssim j_i^{-b}$ and $\EE_{i-1}\big( \Car_{\mathcal{E}_{i}} \xi_{i+1}^2 \big)\lesssim j_i^{-b}$, for some $b>1$.
 By using the Taylor expansion of the $\log$ to the second order, we  find that on $\mathcal{E}_{i+1}$,
$$ \log \frac{t_{i+1}}{t_{i}} = \frac{3v}{8}\frac{\Delta l_{i-1}}{l_{i-1}}-\frac{\Delta l_i}{4l_i}+G_{l_{i},l_{i+1}}(1,1) - \frac{1}{2}G_{l_{i},l_{i+1}}(1,1)^2+ \tilde \xi_i,$$
where $|\EE_{i-1}\big( \Car_{\mathcal{E}_{i}} \tilde \xi_{i+1} \big)|\lesssim j_i^{-b}$ and $\EE_{i-1}\big( \Car_{\mathcal{E}_{i}} \tilde \xi_{i+1}^2 \big)\lesssim j_i^{-b}$. The same argument as in the proof of Lemma \ref{secondmomenteps} shows that
\[ G_{l_{i},l_{i+1}}(1,1)^2=  \frac{v\Delta l_i}{4l_i} + \zeta_{i+1},\]
where $|\EE_{i-1}\big( \Car_{\mathcal{E}_{i}} \zeta_{i+1} \big)|\lesssim j_i^{-b}$ and $\EE_{i-1}\big( \Car_{\mathcal{E}_{i}} \zeta_{i+1}^2 \big)\lesssim j_i^{-b}$. Therefore, for any $i\geq 2$, on the event $\mathcal{E}_i$,
\[ \log \frac{t_{i+1} }{t_i} =  \frac{3v}{8}\frac{\Delta l_{i-1}}{l_{i-1}}-\frac{v+2}{8}\frac{\Delta l_i}{l_i}+G_{l_{i},l_{i+1}}(1,1)+ \tilde{\zeta}_{i+1},\]
where $|\EE_{i-1}\big( \Car_{\mathcal{E}_{i}} \tilde{\zeta}_{i+1} \big)|\lesssim j_i^{-b}$ and $\EE_{i-1}\big( \Car_{\mathcal{E}_{i}} \tilde{\zeta}_{i+1}^2 \big)\lesssim j_i^{-b}$. Thus, on $\mathcal{E}_\infty$,
\begin{align*}
\log \frac{t_{i_*}}{t_1}   &=  \sum_{i=2}^{i_*-1} \Big(\frac{3v}{8}\frac{\Delta l_{i-1}}{l_{i-1}}- \frac{v+2}{8}\frac{\Delta l_i}{l_i}+G_{l_{i},l_{i+1}}(1,1) +  \tilde{\zeta}_{i+1}\Big)+\log \frac{t_2}{t_1}\\
& = \sum_{i=1}^{i_*-1}\frac{v-1}{4}\frac{\Delta l_{i}}{l_{i}}+\sum_{i=2}^{i_*-1}\big(G_{l_{i},l_{i+1}}(1,1) +  \tilde{\zeta}_{i+1}\big)+O(1),
\end{align*}
where the fact that for any $i$, $\Delta l_i/l_i =O(j_i^{-1})$ by Corollary \ref{controlrandom} and \eqref{lowerboundtic}. Note that for any $i$, $\sum_{l=l_i+1}^{l_{i+1}} l^{-1} = \Delta l_i/l_i + O( \Delta l_i/l_i)^2$. Therefore,
\begin{equation} \label{calculsum}  \sum_{i=1}^{i_*-1} \Delta l_i/l_i  =\sum_{l=l_1+1}^{l_{i_*}}l^{-1} + O\Big(\sum_{i=1}^{\infty} \Big(\Delta l_i/l_i\Big)^2\Big)=
 \log (l_{i_*}/l_1) + O(1), \end{equation}
where we used again that  $\Delta l_i/l_i = O(j_i^{-1})$.
From the definition \eqref{defli} of the stopping times $l_i$, we deduce that on $\mathcal{E}_\infty$, we have $i_* = \inf\{ i\geq 2: l_i > n_\kappa\}$ and by Corollary \ref{controlrandomvariance}, $l_i \geq \hat l_i$ for any $i\leq i_*$. Since $\hat l_{t_0} =  n_\kappa$ and $(\hat l_i)_i$ is strictly increasing, we obtain that $i_*\leq t_0$. Besides, we know by \eqref{bornesupli} that on $\mathcal{E}_\infty$, $l_{i_*}\lesssim \hat l_{i_*}$. Therefore,
\begin{equation} \label{encardrelfinal} n_\kappa \leq l_{i*}\lesssim n_\kappa.\end{equation}
As $l_1 = \nu k_0^{1/3}$, with $\kappa \leq \nu \leq 2\kappa$ a.s, and $n_\kappa = \lfloor k_0/\lfloor \sqrt{\kappa}\rfloor^2 \rfloor$, we deduce that on $\mathcal{E}_\infty$,
\[ \log \frac{t_{i_*}}{t_1}  =\frac{v-1}{6} \log n+\sum_{i=2}^{i_*-1}\big(G_{l_{i},l_{i+1}}(1,1) +  \tilde{\zeta}_{i+1}\big)+O_\kappa(1).\]
In the next lemma we control the error terms $\tilde{\zeta}_{i+1}$.
 \begin{lemma}\label{moment2err}
\[ \EE \Big( \sum_{i=2}^{i_*-1} \tilde{\zeta}_{i+1}\Car_{\mathcal{E}_i}\Big)^2 = O(1).\]
\end{lemma}
\begin{proof}
Since $|\EE_{i-1} \tilde{\zeta}_{i+1}\Car_{\mathcal{E}_i} |\lesssim j_i^{-b}$, with $b>1$, we get
\begin{align*}
\EE \Big( \sum_{i=2}^{i_*-1} \tilde{\zeta}_{i+1}\Car_{\mathcal{E}_i}\Big)^2 &\leq 2\EE \Big( \sum_{i=2}^{i_*-1} \big(\tilde{\zeta}_{i+1}\Car_{\mathcal{E}_i}-\EE_{i-1}[\tilde{\zeta}_{i+1}\Car_{\mathcal{E}_i}]\big) \Big)^2+2\EE\Big(\sum_{i=2}^{i_*-1} \EE_{i-1}[\tilde{\zeta}_{i+1}\Car_{\mathcal{E}_i}]\big) \Big)^2\\
& =2\EE \Big( \sum_{i=2}^{i_*-1} \big(\tilde{\zeta}_{i+1}\Car_{\mathcal{E}_i}-\EE_{i-1}[\tilde{\zeta}_{i+1}\Car_{\mathcal{E}_i}]\big) \Big)^2+O(1).
\end{align*}
Note that $\big(\tilde{\zeta}_{i+1}\Car_{\mathcal{E}_i}-\EE_{i-1}[\tilde{\zeta}_{i+1}\Car_{\mathcal{E}_i}]\big)_i$ is not a $\mathcal{F}_{l_i}$-martingale, but 
\[\Big(\sum_{i=1}^{j} \big(\tilde{\zeta}_{2i+1}\Car_{\mathcal{E}_{2i}}-\EE_{2i-1}[\tilde{\zeta}_{2i+1}\Car_{\mathcal{E}_{2i}}]\big)\Big)_{j\leq \lfloor (i_*-1)/2\rfloor}\]
 is a martingale w.r.t the filtration $(\mathcal{F}_{k_0+l_{2j+1}})_{j\leq \lfloor( i_*-1)/2\rfloor}$, and similarly
\[\Big(\sum_{i=2}^{j} \big(\tilde{\zeta}_{2i}\Car_{\mathcal{E}_{2i}}-\EE_{2(i-1)}[\tilde{\zeta}_{2i}\Car_{\mathcal{E}_{2i}}]\big)\Big)_{j\leq \lfloor i_*/2\rfloor}\]
 is a martingale w.r.t the filtration $(\mathcal{F}_{k_0+l_{2j}})_{j\leq \lfloor i_*/2\rfloor}$. Therefore,
\begin{align*}
\EE \Big( \sum_{i=2}^{i_*-1} \big(\tilde{\zeta}_{i+1}\Car_{\mathcal{E}_i}-\EE_{i-1}[\tilde{\zeta}_{i+1}\Car_{\mathcal{E}_i}]\big) \Big)^2& \lesssim \sum_{i=1}^{\lfloor( i_*-1)/2\rfloor} \EE\Big(\tilde{\zeta}_{2i+1}\Car_{\mathcal{E}_{2i}}-\EE_{2i-1}[\tilde{\zeta}_{2i+1}\Car_{\mathcal{E}_{2i}}]\Big)^2\\
&+ \sum_{i=2}^{\lfloor i_*/2\rfloor} \EE\Big(\tilde{\zeta}_{2i}\Car_{\mathcal{E}_{2i-1}}-\EE_{2(i-1)}[\tilde{\zeta}_{2i}\Car_{\mathcal{E}_{2i-1}}]\Big)^2
=O(1),
\end{align*}
where we used the fact that $\EE_{i-1}(\Car_{\mathcal{E}_i} \tilde{\zeta}_{i+1}^2)\lesssim j_i^{-b}$ with $b>1$.
\end{proof}
Observe that $(\sum_{i=2}^{j-1} G_{l_{i},l_{i+1}}(1,1)\Car_{\mathcal{E}_i})_{j\geq 3}$ is a martingale w.r.t the filtration $(\mathcal{F}_{k_0+l_{j}})_{j\geq 3}$. 
One can check that on the event $\mathcal{E}_\infty$, the quadratic variation of this martingale converges indeed to $(v/6)\log n$. In order to be able to apply the martingale central limit theorem, we will introduce a sequence $(W_i)_{i\geq 3}$ of independent random variables, independent from $\mathcal{F}_n$ such that
$$ \EE W_i =0, \ \EE W_i^2 = v \Delta \hat l_{i-1}/4\hat l_{i-1}, \ \EE |W_i|\leq 1,$$
and we define for any $j\geq 3$
$$ M_{j,\kappa} = \sum_{i=2}^{j-1} G_{l_{i},l_{i+1}}(1,1)\Car_{\mathcal{E}_i}+\sum_{i=2}^{j-1} W_{i+1}\Car_{\mathcal{E}_i^c}.$$
As we have shown, on the event $\mathcal{E}_\infty$, $i_*\leq t_0$. Thus, on this event we have  $M_{t_0,\kappa} =\sum_{i=2}^{i_*-1} G_{l_i,l_{i+1}}(1,1)$.
As a consequence of Lemma \ref{moment2err} and the fact that $\PP(\mathcal{E}_\infty) \to 1$ as $\kappa \to \infty$ by Lemma \ref{cotnrolbruit}, we get that for any $\delta >0$,
\[ \lim_{\kappa \to \infty} \limsup_{n\to \infty}\PP\Big( \Big| \log \frac{t_{i_*}}{t_1} - \frac{v-1}{6} \log n - M_{t_0,\kappa}\Big| > \delta \sqrt{\log n} \Big) =0.\]
Therefore, it remains to show that for any $\kappa\geq 1$,
\begin{equation}\label{CLTM} \frac{M_{t_0,\kappa}}{\sqrt{\frac{v}{6}\log n}} \underset{n\to \infty}{\leadsto} \gamma,\end{equation}
where $\leadsto$ denotes  weak convergence.
Let $\tilde{\mathcal{F}}_i$ denote the $\sigma$-algebra generated by $(W_k)_{k\leq i}$. Clearly, $(M_{j,\kappa})_{j\geq 1}$ is a martingale w.r.t the filtration $(\mathcal{G}_j )_{j\geq 1}$ with $\mathcal{G}_j=\tilde{\mathcal{F}}_j \vee \mathcal{F}_{k_0+l_{j}}$.   For any $j\geq 3$, set $\sigma_j^2 = \EE((M_{j,\kappa}-M_{j-1,\kappa})^2|\mathcal{G}_{j-1})$. We have
\begin{align*}
\sum_{i=3}^{t_0} \sigma_i^2&= \sum_{i=2}^{t_0-1} \frac{v\Delta l_i}{4l_i}\Car_{\mathcal{E}_i} + \sum_{i=2}^{t_0-1} \frac{v\Delta \hat l_i}{4\hat l_i}\Car_{\mathcal{E}_i^c}.
\end{align*}
Define
$$ j_ * = \left\{\begin{array}{ll}
1,& \omega\not \in \mathcal{E}_2,\\
\max\big\{  i\leq t_0 : \omega \in \mathcal{E}_i\big\},& \mbox{\rm else}.
\end{array}
\right.$$
Then, we can write
$$
\sum_{i=3}^{t_0} \sigma_i^2= v\big(\sum_{i=2}^{j_*} \Delta l_i/4l_i + \sum_{i=j_*+1}^{t_0-1} \Delta \hat l_i/4\hat l_i\big).$$
Using the same arguments as in the proof of \eqref{calculsum}, we deduce that
\[ \sum_{i=3}^{t_0} \sigma_i^2= v\big(\sum_{l=l_2+1}^{l_{j_*}}l^{-1}/4 +\sum_{l=\hat l_{j_*}+1}^{\hat l_{t_0}}l^{-1}/4\big)+O(1).\]
As $\omega \in \mathcal{E}_{j_*}$, we have $\hat l_{j_*} \leq l_{j_*} \leq C \hat l_{j_*}$ by Corollary \ref{controlrandom} and \eqref{bornesupli}, with $C$ an absolute constant. Therefore,
$ \sum_{l=\hat l_{j_*}+1}^{l_{j_*}} l^{-1} \leq (l_{j_*}-\hat l_{j_*})/\hat l_{j_*} \leq C.$
Thus,
\[\sum_{i=3}^{t_0} \sigma_i^2=v \sum_{l=l_2+1}^{\hat l_{t_0}}l^{-1}/4 +O(1)=\frac{v}{6}\log n + O_\kappa(1),\]
where in the last equality we used that
that $\Delta l_1/l_1 =O(1)$ by \eqref{controlrandomvariance}, that $l_1= \nu k_0^{1/3}$ with $\nu\in [\kappa, 2\kappa]$, and that
$\hat l_{t_0}\sim \kappa^{-1}k_0$.
Besides, we proved in \eqref{moment4} that
$$  \sum_{i=2}^{t_0-1} \EE \big( \Car_{\mathcal{E}_i} G_{l_i,l_{i+1}}^4\big) =O(1).$$
Thus, by the martingale central limit theorem, see e.g  \cite[Theorem 3.3]{Hall-Heyde}, we conclude that   if $\tilde \nu^{(0)}_{\kappa,n}$ is the conditional law of
$ M_{t_0,\kappa}/\sqrt{\frac{v}{6}\log n}$
given $\mathcal{F}_{k_0-l_0}$, then
$\tilde \nu_{\kappa,n}^{(0)} \underset{n \to \infty}{\leadsto}  \gamma.$
\end{proof}
We are now ready to give a proof of Proposition \ref{osci}.
\begin{proof}[Proof of Proposition \ref{osci}]
  We first show that it is sufficient to prove the same statement with $\| X_n\|$ instead of $|\psi_n|$. Since $X_k=(\psi_k,\psi_{k-1})$,
  we have the inequality
$|\psi_n| \leq \| X_n \| \leq |\psi_n| (1+|\psi_{n-1}|/|\psi_n|)$, and therefore
it suffices for this purpose to prove the following.
\begin{lemma}\label{controlratio} For any $\veps>0$, and $n$ large enough,
$\PP\big( | \psi_n| \leq \veps |\psi_{n-1}| \big)\lesssim \veps^{1/6}.$
\end{lemma}
\begin{proof}
Recall that for any $k\geq k_0$, we have that $X_k = (A_k+W_k)X_{k-1}$ with $A_k$ and $W_k$  defined in \eqref{eq-Xk}.
We start by showing the following behaviour of $A_k$  for  $k\geq k_0(1+\kappa^{-1} )$.
\begin{fact}\label{controlA}There exists a constant $C_\kappa$ depending on $\kappa$ only  such that
for any $(1+\kappa^{-1}) k_0 \leq k < l \leq n$,
$ \|  A_{k,l}\| \leq C_\kappa.$
\end{fact}
\begin{proof}
According to Fact \ref{newbasis}, for any $l\geq 1$,
$ A_{k_0+l} =\sqrt{1- c_{k_0+l}} Q_l R_l Q_l^{-1},$
where $R_l$ is a rotation matrix, and $Q_l$ is defined in \eqref{defQ}. Thus, for any $1\leq l<l'\leq n-k_0$,
$$ \|\prod_{j=l+1}^{l'} A_{k_0+j} \| \leq \Big( \prod_{j=l+1}^{l'} (1-c_{k_0+j})\Big)^{\frac{1}{2}}\|Q_{l'}\| \prod_{p=l+1}^{l'-1} \|Q_{p+1}^{-1} Q_p \|\cdot \| Q_{l+1}^{-1}\|.$$
Using the first equality of \eqref{qinvq}, we obtain that for any $p\geq \kappa^{-1} k_0$, 
$ \| Q_{p+1}^{-1} Q_p\| \leq 1+O_{\kappa}(1/n).$
Morevoer, for any $l\geq \kappa^{-1}k_0$, $\| Q_{l+1}^{-1}\| =O_{\kappa}(1)$, as one can check with \eqref{defQ}. Also, for any $l$, we have that $\|Q_l\|\leq C$.
Therefore, using the fact that $c_{k_0+j} \lesssim 1/n$, we get that
$ \|\prod_{j=l+1}^{l'} A_{k_0+j} \| \leq C_{\kappa},$
for some constant $C_\kappa$ depending on $\kappa$.
\end{proof}
Let $\veps>0$ and set 
$ n_{\veps}= n-\lfloor \veps n \rfloor$. Using Lemma \ref{linear} together with the fact that $\EE_{\mathcal{F}_{n_\veps}} \| W_k\|^2 \lesssim 1/n$ for any $k \geq n_\veps$ and
Fact \ref{controlA}, we get that
$$ X_n = A_{n_\veps,n} X_{n_{\veps}} +G_nX_{n_{\veps}} + \mathcal{Y}_n,$$
where $\EE_{\mathcal{F}_{n_\veps}} \| \mathcal{Y}_n\|^2 = O_{\kappa}( \veps^2 \| X_{n_\veps}\|^2)$, and
\begin{equation} \label{defGn}G_n=  \sum_{k=n_{\veps}+1}^n A_{k,n} W_k A_{n_{\veps}, k-1}.\end{equation}
By Fact \ref{controlA} we have that
$\EE_{\mathcal{F}_{n_\veps}} \|G_n\|^2 =O(\veps)$. The proof of Lemma \ref{controlratio} will follow classical 
anti-concentration arguments using the randomness of the $\lfloor \veps n \rfloor$ steps. Let $s>0$ such that $s\sqrt{\veps}\leq 1$ and define the event
$$ \mathcal{E} = \Big\{ |A_{n_\veps,n}X_{n_\veps}(2)| \geq s \sqrt{\veps} \| X_{n_\veps}\|,\ \|G_n\|\leq s\sqrt{\veps}/4, \|\mathcal{Y}_n\| \leq  s \sqrt{\veps} \| X_{n_\veps}\|/4 \Big\}.$$
On the event $\mathcal{E}$, one can linearize $\psi_n/\psi_{n-1}$ around $A_{n_\veps,n} X_{n_\veps}(1)/A_{n_\veps,n} X_{n_\veps}(2)$ into a sum of independent random variables given $\mathcal{F}_{n_\veps}$ and use the anti-concentration estimate of Lemma \ref{anticonc}, together with a second moment estimate
that controls the tail of the second order noise $\mathcal{Y}_n$. On the 
complementary event $\mathcal{E}^c$, either $\psi_{n}$ is small or the noise is large. The later case is again handled by a second moment estimate whereas the former case is handled as before by observing that
$\psi_{n}$ is at first order the sum of independent random variables given $\mathcal{F}_{n_\veps}$ and using again the anti-concentration Lemma \ref{anticonc}.

We now proceed with the detailed proof. 
 On the event $\mathcal{E}$, we have $|G_n X_{n_\veps}(2) +\mathcal{Y}_n(2)| \leq (1/2)|A_{n_\veps,n} X_{n_\veps}(2)|$, so that
$$ \frac{1}{\psi_{n-1}} = \frac{1}{A_{n_\veps,n}X_{n_\veps}(2)} - \frac{G_n X_{n_\veps}(2)+\mathcal{Y}_n(2)}{A_{n_\veps,n}X_{n_\veps}(2)^2} + O\bigg( \frac{| G_n X_{n_\veps}(2)|^2+ |\mathcal{Y}_{n}(2)|^2}{s^3 \veps^{3/2} \| X_{n_\veps}\|^{3}}\bigg).$$
Using again the fact that $|A_{n_\veps,n}X_{n_\veps}(2)| \geq s \sqrt{\veps} \|X_{n_\veps}\|$, we get
\begin{align*}
  \frac{1}{\psi_{n-1}} &= 
 \frac{1}{A_{n_\veps,n}X_{n_\veps}(2)} - \frac{G_n X_{n_\veps}(2)}{A_{n,n_\veps}X_{n_\veps}(2)^2}\\
 & + O\bigg( \frac{\| G_n \|^2}{s^3 \veps^{3/2}\|X_{n_\veps}\|}\bigg)+O\Bigg( \frac{\|\mathcal{Y}_n\|}{s^2\veps \|X_{n_\veps}\|^2}\Bigg) + O\bigg( \frac{\|\mathcal{Y}_{n}\|^2}{s^3 \veps^{3/2} \| X_{n_\veps}\|^3}\bigg).\end{align*}
 Since on $\mathcal{E}$ we have that
 $\|\mathcal{Y}_n\| \leq  s \sqrt{\veps} \| X_{n_\veps} \|/4$, 
 we deduce that
\[  \frac{1}{\psi_{n-1}} = \frac{1}{A_{n_\veps,n}X_{n_\veps}(2)} - \frac{G_n X_{n_\veps}(2)}{A_{n,n_\veps}X_{n_\veps}(2)^2}+O\bigg( \frac{\| G_n \|^2}{s^3 \veps^{3/2}\|X_{n_\veps}\|}\bigg)+O\Bigg( \frac{\|\mathcal{Y}_n\|}{s^2\veps \|X_{n_\veps}\|^2}\Bigg).\]
Therefore,
\begin{equation}\label{psin} \frac{\psi_n}{\psi_{n-1}} = \frac{A_{n_\veps,n}X_{n_\veps}(1)}{A_{n_\veps,n}X_{n_\veps}(2)} + L_n + Z_n,\end{equation}
where, again on $\mathcal{E}$,
\begin{eqnarray} \label{defLn} \qquad L_n &=&- \frac{A_{n_\veps,n}X_{n_\veps}(1)}{A_{n_\veps,n}X_{n_\veps}(2)^2}G_n X_{n_\veps}(2) +\frac{1}{A_{n_\veps,n}X_{n_\veps}(2)}G_n X_{n_\veps}(1),\\
|Z_n| &\lesssim &\frac{|G_n X_{n_\veps}(2)|. |G_n X_{n_\veps}(1)|}{|A_{n_\veps,n} X_{n_\veps}(2)|^2}+ \frac{|\mathcal{Y}_n(1)|}{|\psi_{n-1}|}+|\psi_n| \Big(\frac{\|G_n\|^2}{s^3 \veps^{3/2}\|X_{n_\veps}\|}+\frac{\|\mathcal{Y}_n\|}{s^2 \veps \|X_{n_\veps}\|^2}\Big).\nonumber
\end{eqnarray}
On the event $\mathcal{E}$, we have that
$|A_{n_\veps,n} X_{n_\veps}(2)|\geq s\sqrt{\veps} \|X_{n_\veps}\|$
and $|\psi_{n-1}|\geq (s/2) \sqrt{\veps} \|X_{n_\veps}\|$. Moreover, from  Fact \ref{controlA} we deduce that $|A_{n_\veps,n} X_{n_\veps}(1)|\lesssim \|X_{n_\veps}\|$. Therefore,
$$ |Z_n| \lesssim \frac{\|G_n\|^2 }{s^2 \veps} + \frac{\|\mathcal{Y}_n\|}{s\sqrt{\veps} \|X_{n_\veps}\|}+\frac{\|G_n\|^2}{s^3 \veps^{3/2}}+  \frac{ \|\mathcal{Y}_n\|}{s^2 \veps \|X_{n_\veps}\|}.$$
Since $s\sqrt{\veps} \leq 1$, we get that
$$ |Z_n| \lesssim\frac{\|G_n\|^2}{s^3 \veps^{3/2}}+  \frac{ \|\mathcal{Y}_n\|}{s^2 \veps \|X_{n_\veps}\|}.$$
In view of \eqref{psin}, we obtain, by performing a union bound:
\[  \PP \Big( \Big\{ \Big|\frac{\psi_n}{\psi_{n-1}}\Big| \leq t \Big\} \cap \mathcal{E}\Big) \leq 
 \EE\Big(\Car_{\mathcal{E}'}\PP_{\mathcal{F}_{n_\veps}}\Big( \Big|\frac{A_{n_\veps,n}X_{n_\veps}(1)}{A_{n_\veps,n}X_{n_\veps}(2)} +L_n \Big|\leq 2t \Big) \Big) + \PP\big(\{ |Z_n|\geq t\} \cap\mathcal{E}\big),\]
where $\mathcal{E}' = \{  |A_{n_\veps,n}X_{n_\veps}(2)| \geq s \sqrt{\veps} \| X_{n_\veps}\|\}$.
Again by a union bound and Chebychev's inequality, we get that
\[ \PP(|Z_n| \geq t) \lesssim \frac{\EE \|G_n\|^2 }{s^3 \veps^{3/2} t} + \frac{\EE \|\mathcal{Y}_n\|^2}{s^4\veps^2 \|X_{n_\veps}\|^2 t^2} .\]
But $\EE_{\mathcal{F}_{n_\veps}} \| G_n\|^2 = O(\veps)$
and $\EE_{\mathcal{F}_{n_\veps}} \| \mathcal{Y}_n\|^2 =O(\veps^2 \| X_{n_\veps}\|^2)$. Therefore,
\begin{equation} \label{tailZn} \PP(|Z_n| \geq t) \lesssim t^{-1}s^{-3} \veps^{-1/2} + s^{-4}t^{-2}.\end{equation}
In view of  \eqref{defLn} and \eqref{defGn}, one can check that conditionning on $\mathcal{F}_{n_\veps}$, on the event $\mathcal{E}'$, $L_n$ is the sum of $\lfloor n\veps\rfloor$ independent random variables $V_k$, $n_\veps+1 \leq k\leq n$, such that 
\[ V_k = - \frac{A_{n_\veps,n}X_{n_\veps}(1)}{A_{n_\veps,n}X_{n_\veps}(2)^2}A_{k,n} W_k A_{n_\veps, k-1}  X_{n_\veps}(2) +\frac{1}{A_{n_\veps,n}X_{n_\veps}(2)}A_{k,n} W_k A_{n_\veps, k-1}  X_{n_\veps}(1).\]
Since $W_k$ is centered and $\EE_{\mathcal{F}_{k}} \|  W_k \|^2 \lesssim 1/n$, we deduce that 
$$\EE_{\mathcal{F}_{n_\veps}} V_k = 0, \ \EE_{\mathcal{F}_{n_\veps}} V_k^2\lesssim s^{-4}\veps^{-2} n^{-1},$$
where we used Fact \ref{controlA} and the fact that on $\mathcal{E}'$, $|A_{n_\veps, n} X_{n_\veps}(2) | \geq s\sqrt{\veps} \|X_{n_\veps}\|$. 
Note that $V_k = r_k b_k + s_k g_k$, where $r_k$ and $s_k$ are $\mathcal{F}_{n_\veps}$-measurable, and $b_k$, $g_k$ are independent and have their Laplace transform uniformly bounded as in \eqref{boundlaplace}. We deduce that
$$ \ \EE_{\mathcal{F}_{n_\veps}} |V_k|^3\leq M' \Big( \EE_{\mathcal{F}_{n_\veps}}V_k^2 \Big)^{3/2},$$
where $M'$ depends on $M$ and $\lambda_0$ of \eqref{boundlaplace}. In particular, 
  on $\mathcal{E}'$,  $\EE_{\mathcal{F}_{n_\veps}} L_n^2=O(1/s^4\veps)$. By Lemma \ref{anticonc}, we obtain that on $\mathcal{E}'$, for any $t \gtrsim  1/s^2\veps\sqrt{n}$,
\begin{equation} \label{anticoncLn}  \PP_{\mathcal{F}_{n_\veps}}\Big( \Big|\frac{A_{n_\veps,n}X_{n_\veps}(1)}{A_{n_\veps,n}X_{n_\veps}(2)} +L_n \Big|\leq 2t \Big)\lesssim s^2 \sqrt{\veps} t.\end{equation}
Combining  \eqref{anticoncLn} and \eqref{tailZn}, we get that
\begin{equation}\label{ratioE}\PP \Big( \big\{ |{\psi_n}/{\psi_{n-1}}| \leq t \big\} \cap \mathcal{E}\Big) \lesssim s^2\sqrt{\veps} t +  t^{-1}s^{-3} \veps^{-1/2} + s^{-4}t^{-2}.\end{equation}
On the other hand,
\[\PP\Big( \Big\{ \Big|\frac{\psi_n}{\psi_{n-1}}\Big| \leq t \Big\} \cap \mathcal{E}^c\Big)  \leq \PP\Big( |\psi_{n}| \leq \frac{3ts}{2} \sqrt{\veps} \|X_{n_{\veps} } \| \Big) + \PP\big( \|G_n\| \geq \frac{s}{4} \sqrt{\veps}\big)+\PP\big(\|\mathcal{Y}_n\|\geq \frac{s}{4} \sqrt{\veps} \| X_{n_\veps}\|\big).\]
Using Chebychev's inequality and the fact that $\EE_{\mathcal{F}_{n_\veps}} \|G_n\|^2 = O(\veps)$ and $\EE_{\mathcal{F}_{n_\veps}} \|\mathcal{Y}_n\|^2=O(\veps^2 \|X_{n_\veps}\|^2)$, we find that
\begin{equation}\label{tailYnGn} \PP_{\mathcal{F}_{n_\veps}} \big( \|G_n\| \geq s\sqrt{\veps}/4\big)+\PP_{\mathcal{F}_{n_\veps}} \big(\|\mathcal{Y}_n\|\geq 
  s \sqrt{\veps} \| X_{n_\veps}\|/4\big) \lesssim s^{-2}.\end{equation}
But, using a union bound
\begin{align}\PP\Big( |\psi_{n}| \leq 3ts \sqrt{\veps} \|X_{n_{\veps} } \|/2 \Big) & \leq \PP\Big( | A_{n_\veps,n} X_{n_{\veps}}(1) +G_nX_{n_{\veps}}(1) | \leq 2ts \sqrt{\veps} \|X_{n_{\veps} } \| \Big)\nonumber \\
& +\PP\Big(| \mathcal{Y}_n(1)|\geq ts \sqrt{\veps} \|X_{n_{\veps} } \|/2\Big).\label{unionboudEc}
\end{align}
Using a similar argument as in \eqref{anticoncLn}, we deduce by Lemma \ref{anticonc} that if 
 $ts\sqrt{\veps} \gtrsim 1/n$, then
\begin{equation} \label{anticonclinear} \PP_{\mathcal{F}_{n_\veps}} \Big( | A_{n_\veps,n} X_{n_{\veps}}(1) +G_nX_{n_{\veps}}(1) | \leq 2ts \sqrt{\veps} \|X_{n_{\veps} } \| \Big) \lesssim ts,\end{equation}
where we used that  $\EE_{\mathcal{F}_{n_\veps}}\|G_n\|^2 = O(\veps)$. 
Using again Chebychev's inequality, we obtain together with \eqref{anticonclinear} that if $ts\sqrt{\veps} \gtrsim 1/n$, then
\begin{equation} \label{antinconcpsinEc} \PP_{\mathcal{F}_{n_\veps}} \Big( |\psi_{n}| \leq 3ts \sqrt{\veps} \|X_{n_{\veps} } \|/2 \Big)  \lesssim ts + \veps/ t^2s^2 +1/s^2.\end{equation}
Putting together \eqref{antinconcpsinEc} and \eqref{tailYnGn} and taking expectations, we conclude that for any $t\gtrsim 1/s\sqrt{\veps} n$,
\[\PP\Big( \big\{ |{\psi_n}/{\psi_{n-1}}| \leq t \big\} \cap \mathcal{E}^c\Big)  \leq ts + \veps t^{-2}s^{-2} +s^{-1}.\]
Together with \eqref{ratioE}, we find, using that $s\sqrt{\veps} \leq 1$ that if $t\gtrsim 1/s^2 \veps n$, then
$$\PP\big(  |{\psi_n}/{\psi_{n-1}}| \leq t \big) \lesssim  
s^2\sqrt{\veps} t  +  t^{-1}s^{-3} \veps^{-1/2} + 
s^{-4}t^{-2}+ts + \veps t^{-2}s^{-2} +s^{-1}.$$
Taking $t= \veps$ and $s= \veps^{-2/3}$, we get that
$\PP(  |\psi_n/\psi_{n-1}| \leq \veps) \lesssim \veps^{1/6}.$
\end{proof}
We proceed with the proof of Proposition \ref{osci}. We can write
$$ \frac{\|X_n\|}{|\psi_{k_0- l_0}|} = \frac{\|X_n\|}{\|X_{k_0+n_\kappa}\|}\frac{\|X_{k_0+n_\kappa}\|\cdot |t_1|}{|\psi_{k_0-l_0}t_{i_*}|} \frac{|t_{i_*}|}{|t_1|},$$
where $i_* = \inf \{ i\geq 1 : l_i  =l_{i+1}\}$ and $t_i$ are
as in \eqref{defyi}.
We first show that the noise accumulated between the iterations $k_0+n_\kappa$ and $n$ is negligible, meaning that the random variables $\log (\| X_n\|/\|X_{k_0+n_\kappa}\|)$ are  tight.
\begin{lemma}
\label{lastblock}There exists $c_\kappa>0$ such that
$$ \lim_{\kappa\to \infty} \liminf_{n\to \infty} \PP\Big(c_\kappa^{-1} \leq  
\|X_n\|/\|X_{k_0+n_\kappa}\|\leq c_\kappa\Big)=1.$$
\end{lemma}
\begin{proof}
 We start by proving that there exists $c_\kappa>0$ such that 
\begin{equation}\label{tightnessupperbound} \lim_{\kappa\to \infty} \liminf_{n\to \infty} \PP\Big( \|X_n\|/\|X_{k_0+n_\kappa}\|>c_\kappa\Big)=0.\end{equation}
 Let $n' = k_0+ n_\kappa$ and  $M = \prod_{k=n'+1}^{n} (A_k + W_k)$. Using the fact that $(W_j)_{j\geq 1}$ are centered independent random matrices, we get 
$$ \EE\tr M^T M  = \sum_{(M_i)_{n'+1\leq i\leq n}\atop M_i \in \{ A_k, W_k\}} \EE \tr| M_{n'+1}M_{n'+2}\ldots M_n|^2.$$
Let $\ell_n = n-k_0-n_\kappa$. Using Lemma \ref{controlA}, we obtain that
\begin{equation}\label{controlM} 
 \EE\tr M^T M   \leq \sum_{ p=0}^{\ell_n} \sum_{n'<i_1<\ldots <i_p\leq n} C_{\kappa}^{2(p+1)} 
n^{-p}
\leq C_{\kappa}^2e^{C_\kappa^2\ell_n/n}=O_\kappa(1).
\end{equation}
We get that
$ \EE (   ||X_n||/||X_{n'}||)^2 \lesssim 1,$
and conclude the proof of \eqref{tightnessupperbound}.
To complete the proof of Lemma \ref{lastblock}, it is sufficient to prove that the smallest singular value of $M$, denoted $s_1$,
 is bounded away from $0$ with large probability. More precisely, we show that
there exists $m_\kappa>0$ depending on $\kappa$ only such that
$$ \lim_{\kappa\to \infty} \liminf_{n\to \infty} \PP( s_1 \leq m_\kappa^{-1} ) =0.$$
 Note first that \eqref{controlM} entails that if $s_2$ is the largest singular value of $M$, then
$$ \lim_{\kappa \to \infty} \liminf_{n\to \infty}\PP( s_2 \geq m_\kappa')=0,$$
for some constant $m_\kappa'$ depending on $\kappa$ only.
 On the other hand,
$$ \mathrm{det}(M)= \prod_{k=n'+1}^n\Big( 1 - c_k -\frac{g_k}{\sqrt{k}}\Big).$$
As $s_1= |\det M|/s_2$, Lemma \ref{lastblock} is an immediate consequence of
Lemma \ref{lem-orphan} below.
\end{proof}
\begin{lemma}
  \label{lem-orphan}
\[ \lim_{t \to 0} \limsup_{n\to \infty} \PP\Big( \Big|\prod_{k=n'+1}^n( 1 - c_k -g_k/\sqrt{k})\Big| \leq t \Big)=0.\]
\end{lemma}
\begin{proof}[Proof of Lemma \ref{lem-orphan}]
 Define for any $k \in \{n'+1,\ldots, n\}$, 
$$\tilde g_k = 
\big(g_k\wedge (\log n)^2 \big)\vee (-(\log n)^2),$$
 and set $ \mu_k = \log \Big( 1- c_k - \frac{\tilde g_k}{\sqrt{k}} \Big)$. We will first prove that 
\begin{equation} \label{claimYk} \lim_{s\to \infty }\liminf _{n\to \infty}\PP\Big( \sum_{k=n'+1}^n \mu_k \leq -s\Big) =0.\end{equation}
We have the inequality
\begin{equation} \label{encadrementlog} -c_k- \frac{\tilde g_k}{\sqrt{k}} - \frac{1}{2} \Big( c_k+ \frac{\tilde g_k}{\sqrt{k}} \Big)^2 \leq \mu_k \leq -c_k- \frac{\tilde g_k}{\sqrt{k}}.\end{equation}
Since $c_k=O(1/n)$ for $k\geq n'$, it is sufficient to prove that
\begin{equation} \label{claimgtilde} \lim_{s\to \infty} \liminf_{n\to \infty} \PP\Big( \sum_{k=n'+1}^n \frac{\tilde{g}_k}{\sqrt{k}}\geq s \Big) = 0, \  \lim_{s\to \infty} \liminf_{n\to \infty} \PP\Big( \sum_{k=n'+1}^n \frac{\tilde{g}_k^2}{k}\geq s \Big) = 0.\end{equation}
One can check that the bound \eqref{boundlaplace} on the log-Laplace transform of $g_k$ entails that $ \EE \tilde g_k = o(1/n)$.  Using Chebychev's inequality and the fact that uniformly in $k$, $\EE\tilde{g}_k^2 =O(1)$, we get \eqref{claimgtilde} and therefore \eqref{claimYk}. To complete the proof, it remains to show that 
$$\PP\Big( \sup_{n'+1< k \leq n} |g_k| \geq (\log n)^2\Big) \underset{n\to \infty}{\longrightarrow} 0.$$
But the bound \eqref{boundlaplace} 
on the log-Laplace transform of $g_k$ together with their independence
imply that
$$ \EE  \sup_{n'+1< k \leq n} |g_k|  =O(\log n),$$
which completes the proof of Lemma \ref{lem-orphan}.
\end{proof}

We continue with the proof of Proposition \ref{osci}.
Using the fact that the change of basis described in \eqref{changebase}, only induces a change of the norm $\| X_{k_0+n_\kappa}\|$ by an absolute multiplicative constant, we obtain by Lemma \ref{controlYl0} and Lemma \ref{init} the following.
\begin{lemma} \label{changebasis}There exists a constant $C_\kappa>0$ such that
$$\lim_{\kappa \to \infty} \liminf_{n\to \infty} \PP\Big( C_\kappa^{-1} \leq \frac{\|X_{k_0+n_\kappa}\|\cdot |t_1|}{| \psi_{k_0-l_0} \cdot  t_{i_*}|}\leq C_\kappa\Big)= 1,$$
where $i_* = \inf \{ i\geq 2 :l_i = l_{i+1}\}$.
\end{lemma}
\begin{proof}For $\kappa$ large enough, on the event $\mathcal{E}_\infty$ we have $l_{i_*} =n_\kappa$. According to \eqref{defyi}, we have that 
 $${Y_{n_\kappa}}/{\|Y_{l_0}\|} = t_{i_*}(1,\veps_{i_*})^T .$$
As $|\veps_{i_*}|\leq 1$ on $\mathcal{E}_{\infty}$, we have on this event
\[1 \leq \frac{\|Y_{n_\kappa}\|}{\|Y_{l_0}\|.|t_{i_*}|} \leq \sqrt{2}.\]
Using the fact that  $\PP(\mathcal{E}_\infty) \to 1$ as $\kappa \to \infty$ by Lemma \ref{cotnrolbruit} and Lemma \ref{controlYl0}, we obtain that
\[ \lim_{\kappa \to \infty }\liminf_{n\to \infty}\PP\Big( C_\kappa^{-1} \leq \|Y_{n_\kappa}\|/|t_{i_*}|\leq C_\kappa \Big)=1,\]
where $C_\kappa$ is some positive function increasing with $\kappa$.
But we know from \eqref{changebase} that $\psi_{k_0-l_0}Y_{n_\kappa} = r_{n_\kappa}  Q_{n_\kappa}^{-1} X_{k_0+ n_\kappa}$. One can check that there exists an absolute constant $c>0$ such that 
$$ c^{-1} \leq r_{n_\kappa} \leq c, \  \| Q_{n_\kappa}\| \leq c, \ \| Q_{n_\kappa}^{-1}\| \leq c\kappa^{1/2},$$
which, together with Lemma \ref{init}, concludes the  proof of Lemma \ref{changebasis}.
\end{proof}
From Lemmas \ref{lastblock} and \ref{changebasis}, we deduce that
$$ \log \frac{\| X_n\|}{|\psi_{k_0-l_0}|} = \log \Big|\frac{t_{i^*}}{t_1}\Big| + \eta_n,$$
where
$$\lim_{\kappa \to \infty} \limsup_{n \to \infty} \PP( | \eta_n|\geq c_\kappa) =0,$$
for some function $c_\kappa$ growing to $\infty$ when $\kappa \to \infty$. The convergence in law of Proposition \ref{convlawblock} ends the proof of Proposition \ref{osci}.
\end{proof}

\section{The scalar regime}
  \label{sec-scalar}
  We provide in this section the proof of Proposition \ref{prop-scalar},
  which as discussed in Section \ref{sec-structure}, is based on an effective
  linearization of the recursion. From \eqref{eq-Xdelta} we obtain 
  \eqref{eq-deltak} where 
  $1+\delta_1=(z\sqrt{n}-b_1)/\alpha_1$, and 
  \begin{equation}
    \label{eq-ukvk}
 u_k =   \frac{z_k}{\alpha_k} - 1 - \frac{1-c_{k} + g_k/\sqrt{k}}{\alpha_k \alpha_{k-1}} - \frac{b_k}{\alpha_{k}\sqrt{k}} , \quad   v_k= \frac{1-c_{k} + g_k/\sqrt{k}}{\alpha_{k} \alpha_{k-1}}. 
 \end{equation}
 This leads to the linearized equation 
 \eqref{eq-bardeltak0}, whose solution is
 \begin{equation}
   \label{eq-bardeltak}\bdelta_k = 
 \sum_{j=2}^k u_j \prod_{l=j+1}^{k}v_l + \delta_1 \prod_{l=2}^k v_l, \quad
 k\geq 2.  \end{equation}
We denote $\Delta_k=\delta_k-\bdelta_k$
and introduce the auxilliary sequence
\begin{equation}
  \label{Deltabar}
  \begin{cases}
	\overline{\Delta_k} = -v_k {\bar\delta_{k-1}}^2 + v_k \overline{\Delta_{k-1}},  \ \forall k\geq 2, \\
	\overline{\Delta_1} = \Delta_1=0.
	\end{cases}
      \end{equation}	
      (The sequence $\overline{\Delta_k}$ is obtained by linearizing the 
      recursion for $\Delta_k$.)
       
       In studying the sequences $(\delta_k)_{k \leq k_0-l_0}$ and
       $(\bdelta_k)_{k\leq k_0-l_0}$,  it will be convenient to
       consider separately the negligible  regime 
       $k \leq (1-\varepsilon)k_0$ from the harder
       $ (1-\varepsilon)k_0 \leq k \leq k_0-l_0$, where $\varepsilon>0$ is an arbitrary (small enough) constant.
       In each of the regions, we derive bounds on 
       the first two moments of $\sum  \bdelta_k$ and 
       $\sum \overline{ \Delta_k}$, a-priori estimates on the individual terms,
       and use those to complete the proof of  Proposition \ref{prop-scalar}.

       Throughout the section, we omit integer parts (thus, writing
       $x$ for $\lfloor x\rfloor$) where no confusion or significant error
       in the estimates occur.
       \subsection{Bounds on the coefficients}
       We collect in this short section some useful bounds on the various 
       coefficients entering our recursions.
       \begin{lemma}\label{event}
  	For any $k \leq k_0$, 
	\begin{equation}
	  \label{eq-Euk}
	  \EE(u_k) = \frac{1}{\alpha_{k}^2} (1-c_k)\big( 1-\frac{\alpha_{k}}{\alpha_{k-1}}    \big) 
	  \quad  \text{and}\quad
	  \mathrm{Var}(u_k)\Big(1+O\Big(\frac{1}{k}\Big)\Big)= \frac{v}{\alpha_{k}^2 k} + \frac{v}{\alpha_{k}^2  \alpha_{k-1}^2 k }  .
	\end{equation}
	Consequently, 
	there exists a universal constant $C_{u}=C_{u,\varepsilon}$ such that
		\begin{equation}
		\label{eq-Euk1}
	\forall k \leq (1-\varepsilon)k_0, \quad	\EE(u_k) \leq \frac{C_u}{n}, \quad \text{ and } \quad \forall k \leq k_0 ,  
	\mathrm{Var}(u_k) \leq \frac{C_u}{n}. \end{equation}
Further, for  $\varepsilon$ small enough, 
for any $k \in \{(1-\varepsilon)k_0, \dots, k_0-l_0 \}$, 
\begin{equation}\label{eq:meanu}
  \EE(u_k) = \frac{1}{2\sqrt{k_0 (k_0-k)}}+ O\Big(\frac{1}{k_0}\Big)  .
\end{equation}
Finally, there exists $N \in \NN$ such that for all $n \geq N$ 
	and for all $l \leq (1-\varepsilon)k_0$, $ \EE(v_l^2) \leq \alpha_{l}^{-2}$.
\end{lemma}
The proofs of this lemma is a
consequence of Assumption \ref{ass-ab}, the relation $\alpha_k^2 -z_k \alpha_k +(1-c_k) = 0$ and the development of $\alpha_k$ near $k=k_0$.

The next lemma is again an elementary consequence of the uniform bound \eqref{boundlaplace} on the Laplace transform of the variables $b_k$ and $g_k$, using a union bound argument.
\begin{lemma}[Maximum of $b_k$'s and $g_k$'s]\label{subexpolemma}
There exists a constant $C>0$ such that, with a probability going to $1$ as $n\to\infty$,
\begin{equation}\label{eq:subexpolemma}
\forall k \leq k_0, \ |b_k|\leq C \log n \text{ and } |g_k|\leq C \log n.
\end{equation}
Similarly,
there 
exists a constant $C>0$ such that, 
with a probability going to $1$ as $n\to\infty$.
\begin{equation}
  \label{eq-subbound1}
  \forall k \leq k_0,  \
  |u_k| \leq \frac{C \log n}{\sqrt{n}}  \quad \text{ and } \quad |1-c_k+g_k/\sqrt{k}| \leq 1 + \frac{C \log n}{\sqrt{k}}.  
\end{equation}
\end{lemma}
We denote the event in Lemma \ref{subexpolemma} by $\Omega_n$.
\subsection{The negligible regime: $k \leq (1-\varepsilon)k_0$.}
\label{subsec-easy}
The goal of this section is to show that the negligible 
regime does not contribute
to the fluctuations in the linearized equations. More precisely, we prove the following.
\begin{proposition} \label{easypart}
The sequences of random variables 
\[ \Big(\frac{1}{\sqrt{\log n}}\sum_{k=1}^{(1-\varepsilon)k_0} \bdelta_k\Big)_{n \in \NN} \quad \text{ and } \quad \Big(\frac{1}{\sqrt{\log n}}\sum_{k=1}^{(1-\varepsilon)k_0} \overline{ \Delta_k}\Big)_{n \in \NN}\] 
converge in probability to $0$.
\end{proposition}

A preliminary a-priori bound on $\bdelta_k$ is the following.
\begin{lemma}\label{part1second}
 There exists a universal constant $c$ so that
 for $k \leq (1-\varepsilon)k_0 $,
 \[ \EE(\bdelta_k^2) \leq \frac{c}{n}.\]
\end{lemma}
\begin{proof}
  We use the decomposition
\[ \bdelta_k = \sum_{j=1}^k (u_j-\EE(u_j)) \prod_{l={j+1}}^k v_l + \sum_{j=1}^k \EE(u_j) \prod_{l={j+1}}^k v_l +\delta_1 \prod_{l=2}^k v_l\]
and we bound the second moment of each term. Using Lemma \ref{event}, we obtain
\[ \EE\Big( \Big( \sum_{j=1}^k (u_j-\EE(u_j)) \prod_{l={j+1}}^k v_l\Big)^2 \Big) \leq \frac{C_u}{n}\sum_{j=1}^k \prod_{l={j+1}}^k \EE(v_l^2) \leq \frac{C}{n}  \]
and, as there exists $C_1$ such that for any $l,l' \leq (1-\varepsilon)k_0$ , $\mathbb{E}(v_lv_{l'}) \leq C_1<1 $, developping the square gives
\[ \EE\Big( \Big(\sum_{j=1}^k \EE(u_j) \prod_{l={j+1}}^k v_l\Big)^2 \Big) \leq \frac{C}{n^2} .  \]
Finally, the second moment of term involving $\delta_1$ is exponentially small as, for $l \leq (1-\varepsilon)k_0$, $\EE(v_l^2) \leq C_1 <1$ .
This yields the bound on $\EE \bdelta_k^2$. 
\end{proof}

Before presenting the proof of Proposition \ref{easypart}, we introduce some notation.
	For any $j \in [1,(1-\varepsilon)k_0]$, define
	$ W_j = \sum_{k=j}^{(1-\varepsilon)k_0} \prod_{l=j+1}^k v_l .$ Then,
\begin{equation} \label{sum delta}
\sum_{k=1}^{(1-\varepsilon)k_0} { \bdelta_k} = \sum_{j=2}^{(1-\varepsilon)k_0} u_j W_j + {\delta_1}  W_1
\end{equation}
and
\begin{equation} \label{sum Delta}
\sum_{k=2}^{(1-\varepsilon)k_0} \overline{ \Delta_k} = \sum_{j=2}^{(1-\varepsilon)k_0} -v_j  \bdelta_{j-1}^2 W_j.
\end{equation}

\begin{proof}[Proof of Proposition \ref{easypart}]
	The proof proceeds by 
	showing 
	that the second moment of the left side of
	\eqref{sum delta} 
	 and the (absolute) first moment of the left side
	 of \eqref{sum Delta}
	are uniformly bounded in $n$.
	
	Since for any $l \leq (1-\varepsilon)k_0$,  $\EE(|v_l|) \leq 1/\alpha_{(1-\varepsilon)k_0}$ and $\EE(v_l^2) \leq 1/\alpha_{(1-\varepsilon)k_0}^2$, see
	Lemma \ref{event},
	the 
	second moment of the
	product in the definition of $W_j$ decays exponentially
with its length and therefore
there exists a constant $C>0$ independent of $j$ and 
$n$ such that $\EE(W_j^2) \leq C$. 
To control the second moment of \eqref{sum delta}
	we decompose
	\[ \sum_{k=1}^{(1-\varepsilon)k_0} {\bdelta_k} - \delta_1 W_1 = \sum_{j=2}^{(1-\varepsilon)k_0} u_j W_j = \sum_{j=2}^{(1-\varepsilon)k_0} (u_j-\EE(u_j)) W_j + \sum_{j=2}^{(1-\varepsilon)k_0} \EE(u_j) W_j =: W_1^u+W_2^u, \]
	and we bound the second moment of each term.
	Using the fact that the variables $u_jW_j$  are decorrelated and $u_j$ is independent of $W_j$, we obtain that
	\begin{align*}
	\EE\big[ (W_1^u)^2\big]&=\sum_{j=1}^{(1-\varepsilon)k_0} \mathrm{Var}(u_j) \EE(W_j^2) \leq C_u C,\\
\EE \big[ (W_2^u)^2\big]&\leq \Big(  \sum_{j=1}^{(1-\varepsilon)k_0}\frac{C_u^2}{n^2}\Big)
\Big(\sum_{j=1}^{(1-\varepsilon)k_0}
\EE(W_{j}^2)\Big) \leq C_u^2 C,
	\end{align*}
where we used the
Cauchy-Schwarz's 
inequality. Since $\EE(\delta_1 W_1)^2 = O(1)$, this shows the first claim.
	
	To bound the first moment of \eqref{sum Delta}, we note that  $\EE\bdelta_{j-1}^2 =O( 1/n)$ for any $j\geq 2$. Since $\bdelta_{j-1}, v_j, W_j$ are independent, we can write
	\[ \sum_{k=2}^{(1-\varepsilon)k_0} \EE(|\overline{\Delta_k}|) = \sum_{j=2}^{(1-\varepsilon)k_0}\EE (|v_j|) \EE (|W_j|) \EE( \bdelta_{j-1}^2)   \lesssim1, \]
where we used the fact that $\EE |W_j| =O(1)$ and $\EE |v_j| =O(1)$.
\end{proof}

Recall the event $\Omega_n$ (see Lemma \ref{subexpolemma}). We proceed to derive
a uniform bound on the variables $\delta_k$ and $\bdelta_k$, in the negligible regime.
\begin{lemma} \label{apriori}
There exists a constant $C$ 
so that, on the event $\Omega_n$,
for all $k \leq (1-\varepsilon) k_0$,
\begin{align*}
|\delta_k|  \leq \frac{C{\log n}}{\sqrt{n}} ,\quad 
|{\bdelta_k}|  \leq \frac{C{\log n}}{\sqrt{n}},
\quad |\Delta_k| \leq \frac{C (\log n)^2}{n},\quad
|\overline{\Delta_k}| \leq \frac{C (\log n)^2}{n},
\end{align*}
and
\begin{equation}\label{RecursionDelta}
\Delta_k =-v_k \bdelta_{k-1}^2 + v_k(1-2\bdelta_{k-1}) \Delta_{k-1} - 
v_k \Delta_{k-1}^2 + v_k \frac{\delta_{k-1}^3}{1+\delta_{k-1}}.
\end{equation}
\end{lemma}

\begin{proof}
	We start by showing that, on $\Omega_n$, there exists a constant 
	$W=W_{\varepsilon}$ such that, 
	\[ \forall k \leq (1-\varepsilon) k_0, \quad \sum_{j=1}^k \prod_{l=j+1}^k |v_l| \leq W. \] Indeed, on  $\Omega_n$ we have that
	\eqref{eq-subbound1} holds, and therefore, recalling that
	$\alpha_k \geq z\sqrt{n/k}/2$, we have that for all $n$ large, 
	\begin{equation}\label{bounvlOmega}   |v_k| \leq \frac{1}{\alpha_{(1-\varepsilon)k_0}}, \quad
	k\leq (1-\veps)k_0,\end{equation}
which leads to, as $\alpha_{(1-\veps)k_0}<1$,
	\[  \sum_{j=1}^k \prod_{l=j+1}^k |v_l| \leq \sum_{j=1}^k \frac{1}{\alpha_{(1-\varepsilon)k_0}^{k-j-1}} \leq \frac{1}{1-1/\alpha_{(1-\varepsilon)k_0}}:=W.\]
	\textbf{Bound on $\bdelta_k$:} 
	On $\Omega_n$ we have that
	\[ {\bdelta_k } = \sum_{j=1}^k u_j \prod_{l=j+1}^k v_l + 
	\delta_1 \prod_{l=2}^k v_l =O\Big(\frac{\log n}{\sqrt{n}}\Big),  \] 
where we used the fact that $\delta_1 = b_1/\alpha_1 +O(1/n)$.

\noindent
\textbf{Bound on $\Delta_k$ and $\delta_k$:}	
If $n$ is large enough, 
then
$\Delta_k$ satisfies
\begin{align}
\label{eq-aprioriDk}
\Delta_k & = v_k \Big( \frac{\delta_{k-1}}{1+\delta_{k-1}} - {\bdelta_{k-1}} \Big) 
= v_k \Delta_{k-1} - v_k\delta_{k-1}^2 + v_k \frac{\delta_{k-1}^3}{1+\delta_{k-1}} \nonumber \\
& = -v_k \bdelta_{k-1}^2 + v_k(1-2{\bdelta_{k-1}}) \Delta_{k-1} - v_k \Delta_{k-1}^2 + v_k \frac{\delta_{k-1}^3}{1+\delta_{k-1}},
\end{align}
which proves \eqref{RecursionDelta}.

On the event $\Omega_n$ where $|{\bdelta_k}| \leq C_{\bdelta}
\log n/\sqrt{n}$ for any $k\leq (1-\veps)k_0$, we prove by induction on $k$ that,  
with $C'=4W(C_{\bdelta})^2$, for $n$ large enough, for any $k \leq (1-\varepsilon)k_0$,
$ |\Delta_k| \leq C' (\log n)^2/n .$

Solving the recursion 
\eqref{RecursionDelta} and substituting the last inequality,
which implies in particular 
together with the bounds on $\bdelta_k$ that $|\delta_k|<1/2$,
we get for all $n$ large that
\begin{align*}
  |\Delta_k| & \leq 2\sum_{j=1}^k (|\bdelta_{j-1}|+|\Delta_{j-1}|)^2 \prod_{l=j}^k |v_l|
  \leq \frac{2(\log n)^2 W}{n}\Big({C'}^2\frac{(\log n)^2}{n}+C_{\bdelta}^2\Big)
  \leq C'\frac{(\log n)^2}{n}.
 \end{align*}
 This completes the induction. The same argument applies to $\overline{\Delta_k}$
 (only simpler, in view of the linearity of the recursion for the latter)
 and  completes the proof of the lemma.
\end{proof}

\subsection{The contributing regime: $(1-\varepsilon)k_0 \leq k \leq k_0-l_0$.}
\label{subsec-contrib}
Recall that $l_0=\kappa k_0^{1/3}$.
We study the 
recursion for the $\delta_k$ in the region where $\alpha_k$ 
is getting close to $1$. This corresponds to $|z_k|$ approaching $2$, 
the edge of the support of the semi-circle law.
The strategy remains the same: in this section,
we study the auxiliary 
sequences $\bdelta_k$ and $\overline{\Delta_k}$. 
In the following Section \ref{sec-det},
we will prove that, with a probability going to $1$ as $\kappa\to\infty$,
those auxiliary sequences remain sufficiently close to
the sequences $\delta_k$ and $\Delta_k$ to obtain the result.

In this part, it is convenient to reindex the sequences with respect to the distance to $k_0$, through the following operation.
\begin{definition}[Star indices]
  For $l\in \{0,\ldots,k_0\}$, we
  define the operator $l\mapsto l^*\in \{0,\ldots,k_0\}$ by 
$l^* =k_0-l.$
\end{definition}
\noindent
We begin by proving a limit law for 
$\sum_{k=(1-\varepsilon)k_0}^{k_0-l_0} \bdelta_k$.
\begin{proposition} \label{TCL delta bar}
For any fixed $\kappa>0$, 
\[ 
\frac{1}{\sqrt{\frac{v}{3} \log n}} \Big(\sum_{k=(1-\varepsilon)k_0}^{k_0-l_0} { \bdelta_k}-\frac{\log n}{6}\Big) \underset{n\to +\infty}{\leadsto}\gamma, 
\]
where $\gamma$ is the standard Gaussian law.
\end{proposition}

The strategy of the proof of Proposition \ref{TCL delta bar} is the following: first we decompose $ \bdelta_k$ as a weighted sum of independent
random variables 
and a negligible rest $R_k+S_k$. We will prove that $\sum (R_k+S_k)/\sqrt{\log n}$ 
converges to zero in probability by some moment estimates, 
and that $\sum G_k/\sqrt{\log n}$ converges in distribution towards a 
certain Gaussian random variable, by using a CLT for sums of independent variables.

The mean in the limiting CLT (for $\sum \delta_k$) will come from
$\sum \Delta_k$, or, as we will show in the next section, from
$\sum \overline {\Delta_k}$, which we control in the second main result of
this section.
 \begin{proposition} \label{Limit Delta bar}
For any fixed $\kappa$, 
\[ \frac{1}{\sqrt{\log n}} \Big( \sum_{k=(1-\varepsilon)k_0}^{k_0-l_0} \overline{ \Delta_k} + \frac{v}{6} \log n \Big) \underset{n\to+\infty}{\longrightarrow}  0, \]
in probability.
\end{proposition}
Before proving Proposition \ref{TCL delta bar} and \ref{Limit Delta bar}, we present several estimates which will be essential in the proofs.  Recall the definition \eqref{eq-ukvk} of $v_k$. 
The following lemma uses a  straight-forward expansion of the coefficients
$\alpha_k$ around the critical point $k=k_0$. The elementary proof, 
which uses the fact that the sequence $\alpha_l$ is decreasing, is
omitted.

\begin{lemma}\label{alpha}
	For any $l  \in [(1-\varepsilon)k_0, k_0-l_0]$, 
\begin{equation}
  \label{eq-alphab}
\alpha_{l} = 1 + \sqrt{\frac{k_0-l}{k_0}} + O\Big(\frac{k_0-l}{k_0}\Big) = 1 + \sqrt{\frac{l^*}{k_0}} + O\Big(\frac{l^*}{k_0}\Big) 
\end{equation}
and for any $(1-\veps)k_0 \leq j<k \leq k_0-l_0$,
\[ e^{-2(k-j)(\sqrt{j^*/k_0} +O(j^*/k_0) )}  \leq   \prod_{l=j+1}^k  \EE(v_l) \leq e^{-2(k-j)(\sqrt{k^*/k_0} - O(k^*/k_0) )}   \]
and 
\[e^{-4(k-j)(\sqrt{j^*/k_0} +O(j^*/k_0) )}    \leq   \prod_{l=j+1}^k  \EE(v_l^2) \leq e^{-4(k-j)(\sqrt{k^*/k_0} - O(k^*/k_0) ) }.   \]
\end{lemma} 
For any $j \in [(1-\varepsilon)k_0, k_0-l_0]$, we define
\begin{equation}
  \label{eq-Wj}
  W_j = \sum_{k=j}^{k_0-l_0} \prod_{l=j+1}^k v_l.
\end{equation}
In order to control $W_j$, we will also  need  a finer estimate than that provided by  Lemma \ref{alpha}, in the form of the following 
important auxiliary computation. 
\begin{lemma}
  \label{lem-aux}
Fix $\alpha>0$ and assume that the sequence $\gamma_l$ satisfies
\begin{equation}
  \label{eq-gammal}
  \gamma_l=e^{-\alpha \sqrt{l^*/k_0}+O(l^*/k_0)}.
\end{equation}
Then,  for $j^*\geq 2l_0$,
\begin{equation}
  \label{eq-sumgammal}
  A_{j}:= \sum_{k=j}^{k_0-l_0} \prod_{l=j+1}^k\gamma_l=\frac{1}{\alpha} 
  \sqrt{\frac{k_0}{j^*}}+O\Big(\sqrt{\frac{k_0}{j^*}}e^{- \bar \mu/2}\Big)+O\Big(\frac{k_0}{(j^*)^2}\Big)+O(1),
\end{equation}
where $\mu=\mu_j =2\alpha (j^*)^{3/2}/ (3\sqrt{k_0})$ and $\bar \mu=\mu(1-(l_0/j^*)^{3/2})$. 

Similarly, for $k\in [(1-\varepsilon)k_0, k_0-l_0]$,
\begin{equation}
  \label{eq-sumgammal1}
  B_{k}:= \sum_{j=(1-\varepsilon)k_0}^{k} \frac{1}{j}
  \prod_{l=j+1}^k\gamma_l=
  \frac{1}{\alpha \sqrt{k_0 k^*}}
  +O\Big(\frac{1}{k_0}\Big).
\end{equation}
\end{lemma}
\begin{proof}
  We have
  \begin{equation}
    \label{eq-prodgamma}
    \prod_{l=j+1}^k\gamma_l=e^{-\mu+\mu (k^*/j^*)^{3/2}+O( (j^*-k^*)j^*/k_0)+O(\sqrt{k^*/k_0})}.\end{equation}
  Therefore,
  \[ A_j=\sum_{m=0}^{j^*-l_0} e^{-\mu(1-(1-m/j^*)^{3/2})+O\big( m j^*/k_0\big)+
  O\big(\sqrt{\frac{j^*-m}{k_0}}\big)}.
\]
Set now $m/j^*=x$. Using Riemann integration, we obtain that
\[ \bar A_j:=A_j+O(1)=
j^* \int_0^{1-l_0/j^*} e^{-\mu(1-(1-x)^{3/2})+O(x (j^*)^2/k_0)
+O(\sqrt{(j^*/k_0) (1-x)})} dx.\]
Making the change of variables $z=\mu(1-(1-x)^{3/2})$ so that $x=1-(1-z/\mu)^{2/3}$, we obtain that
\begin{equation}
  \label{eq-Aj}
  \bar A_j=\frac1\alpha \sqrt{\frac{k_0}{j^*}}
\int_0^{\bar \mu} \frac{e^{-z+O\big( (z/\mu) (j^*)^2/k_0\big)+
O\big( \sqrt{j^*/k_0}\big)}}{\Big(1-\frac{z}{\mu}\Big)^{1/3}} dz.\end{equation}
Now, since $\bar \mu/\mu\in (1/2,1)$, we have that $1/(1-z/\mu)^{1/3}=1+ O\big(\sum_{i=1}^\infty (z/\mu)^i\big)$
and 
\[e^{-z+O\big( (z/\mu) (j^*)^2/k_0\big)+
O\big( \sqrt{j^*/k_0}\big)}=e^{-z(1+O(\sqrt{j^*/k_0}))}
+O\big(\sqrt{j^*/k_0}\big)\]
uniformly in $z\in [0,\bar \mu]$.
Therefore,
\begin{eqnarray}
  \label{eq-mainerror}
  \bar A_j&=&\frac1\alpha \sqrt{\frac{k_0}{j^*}}
\int_0^{\bar \mu} e^{-z(1-O(\sqrt{j^*/k_0}))}\Big(1+O\big( \sum_{i=1}^\infty (z/\mu)^i \big)
\Big)dz+O(1)\\
&=&\frac1\alpha \sqrt{\frac{k_0}{j^*}}+O\big(\sqrt{\frac{k_0}{j^*}}e^{-\bar \mu/2}\big)+O\Big(\frac{k_0}{(j^*)^2}\Big)+O(1).
\nonumber
\end{eqnarray}    
In the last equality, we used that 
\[\int_0^{\bar \mu} e^{-z} z^i dz= 
  \int_0^{\bar \mu} e^{-z+i\log z} dz =ie^{i\log i}\int_0^{\bar \mu/i}e^{-i(x-\log x)}dx= \left\{
\begin{array}{ll}
  O(e^{i\log (i/e)+o(i)}),& i\leq \bar \mu,\\
  O(\bar \mu^i e^{-\bar \mu}),& i>\bar\mu.
\end{array}\right.
\]
We then use that, with $f(x)=x\log(x/ e)$, we have that
$f(x)\leq f(1/\mu)+f'(1/\mu)(x-1/\mu)/2$ for $x-1/\mu\leq \log \mu/\mu$ and therefore (using $x=i/\mu$),
\[ \sum_{i=1}^{ \bar \mu} e^{i\log (i/e)+o(i)-i\log \mu}
  \leq e^{\mu f(1/\mu)}\sum_{i=1}^{ \bar \mu} e^{-\frac 12 (i-1) \log (\mu)+o(i)}
  +\sum_{i=\log \mu}^{\bar \mu} e^{-i} =O(1/\mu)\] 
  while, recalling that $\bar \mu/\mu=1-\delta$ with $\delta=(l_0/j^*)^{3/2}$
  and that $\bar \mu\geq c/\delta$,
  \[ e^{-\bar \mu}\sum_{ i=\hat \mu}^{\infty} (\frac{\bar \mu}{\mu}\big)^i
    =O\Big( e^{-\bar \mu} \frac{1}{\delta} (1-\delta)^{\bar \mu}\Big)=
    O(e^{-\bar\mu/2}).\]
Altogether, the term involving the infinite sum in the right hand side of
\eqref{eq-mainerror} is of order $\sqrt{k_0/j^*}(1 /\mu+e^{-\bar \mu/2})=
O(k_0/(j^*)^2)+O(\sqrt{k_0/j^*}e^{-\bar \mu/2})$.

The argument for $B_k$ is similar. Indeed, let $\nu=2\alpha (k^*)^{3/2} /3\sqrt{k_0}$. We can rewrite \eqref{eq-prodgamma} as 
  \begin{equation}
    \label{eq-prodgamma1}
    \prod_{l=j+1}^k\gamma_l=e^{-\nu [(j^*/k^*)^{3/2}-1]+O( (j^*-k^*)j^*/k_0)+O(\sqrt{k^*/k_0})}.\end{equation}
  Therefore,
  \[ B_k=\frac{1}{k^*}\sum_{l=k^*}^{\varepsilon k_0} \frac{1}{(k_0/k^*)-(l/k^*)}
  e^{-\nu[(l/k^*)^{3/2}-1]+O\Big( l (l-k^*)/k_0\Big)+
  O\Big(\sqrt{\frac{k^*}{k_0}}\Big)}.
\]
Setting $x=l/k^*$, one then obtains that
\begin{eqnarray*}
&&  B_k+O(1/k_0)=\int_1^{\varepsilon k_0/k^*}
\frac{1}{k_0/k^*-x} e^{-\nu(x^{3/2}-1)+O\Big(x(x-1)(k^*)^2/k_0\Big)+O(\sqrt{k^*/k_0})} dx\\
&&= \frac{2}{3}
(1+O(\sqrt{k^*/k_0}))\int_0^{(\varepsilon k_0/k^*)^{3/2}-1}
\frac{1}{(k_0/k^*-(y+1)^{2/3})}\frac{1}{(y+1)^{1/3}}e^{-\nu y+E_y} dy,
\end{eqnarray*}
where $E_y=O( (y+1)^{2/3}( (y+1)^{2/3}-1)/k_0))$.
The dominant contribution occurs near $y=0$ and gives the result.
\end{proof}

As in Section \ref{subsec-easy}, the coefficients $W_j$ appear when 
solving linear recursions. Building on Lemma \ref{lem-aux},
we provide in the next lemma some estimates on these coefficients.
\begin{lemma}\label{Bounds W}
There exists a universal constant $C$ such that 
	\[  \EE(W_{j}) = \begin{cases} 
	  \frac{1}{2}\sqrt{\frac{k_0}{j^*}}+ \varepsilon_{j} & \text{ if } (1-\varepsilon)k_0 \leq j \leq k_0- 2l_0\\
	  \varepsilon_j &  \text{ if } k_0- 2l_0 
	  \leq j \leq k_0-l_0
	\end{cases}  \]
with
\[  |\varepsilon_j| \leq \begin{cases}
  C \frac{k_0}{{j^*}^2} +
C+\frac{C}{2} \sqrt{\frac{k_0}{j^*}}
 e^{-\frac{2 (j^*)^{3/2} (1-(l_0/j^*)^{3/2})}{3\sqrt{k_0}}} & \text{ if } (1-\varepsilon)k_0 \leq j \leq  k_0- 2l_0, 
\\
C(j^*-l_0) &  \text{ if } k_0- 2 l_0 
\leq j \leq k_0-l_0,
\end{cases}  \]
	and 
	\[ \EE(W_j^2)- (\EE W_j)^2\leq \varepsilon'_j,\]
	with
\[  |\varepsilon'_j| \leq \begin{cases}
  C\frac{k_0^{1/3}}{j^*}
 & \text{ if } (1-\varepsilon)k_0 \leq j \leq k_0- 2l_0,\\
C \frac{(j^*-l_0)^3}{n}=o(j^*-l_0)
&  \text{ if } k_0- 2l_0
\leq j \leq k_0-l_0.
\end{cases}  \]	
\end{lemma}
\begin{proof}[Proof of Lemma \ref{Bounds W}] We split the proof into
  estimates of first and second moments.\\
	\textbf{Bound on $\EE(W_j)$:}
Let $j \in [(1-\varepsilon)k_0, k_0-l_0]$. Since the variables $v_l$ are independent, we have that
\[ \EE(W_j) = \sum_{k=j}^{k_0-l_0} \prod_{l=j+1}^k \EE(v_l) \leq \sum_{k=j}^{k_0-l_0} \prod_{l=j+1}^k (1-c_l)   . \]
Since $c_l =O(1/k_0)$ for any $l\in[(1-\veps)k_0, k_0-l_0]$, 
there exists a constant $M$ such that  $\EE (W_j) \leq M (j^*-l_0)$.

For 
$(1-\varepsilon) k_0\leq j\leq k_0-2l_0$ we apply  
Lemma \ref{lem-aux} (with $\alpha=2$) to obtain that
  \begin{equation}
    \label{eq-W1top2} \Big| \EE(W_j)-\frac12 \sqrt{\frac{k_0}{j^*}}\big(1-
    O\big(e^{-\frac{2 (j^*)^{3/2} (1-(l_0/j^*)^{3/2})}{3\sqrt{k_0}}}\big)\big) \Big|\leq 
    O\Big(\frac{k_0}{  {j^*}^2}\Big) + O(1).\end{equation}
\noindent
\textbf{Bound on $\EE(W_j^2)$:}
To obtain this bound, we expand the square and
group the terms: 
\begin{align*}
\EE(W_j^2) = \sum_{k_1,k_2 = j+1}^{k_0-l_0} \EE\Big( \prod_{l_1=j+1}^{k_1} v_{l_1}\prod_{l_2=j+1}^{k_2} v_{l_2} \Big ) = \sum_{k_1,k_2 = j}^{k_0-l_0} \prod_{l_1=j+1}^{k_1 \wedge k_2} \EE(v_{l_1}^2) \prod_{l_2=k_1 \wedge k_2+1}^{k_1 \vee k_2} \EE(v_{l_2}) .
\end{align*}
Using the fact that for any $l \in [(1-\varepsilon)k_0, k_0-l_0] $, $\EE(v_l^2) = \EE(v_l)^2(1+O(\frac{1}{n}))$, we get
\begin{align*}
		\EE(W_j^2)-\EE(W_j)^2 
		&= \sum_{k_1,k_2 = j}^{k_0-l_0} (\prod_{l_1=j+1}^{k_1 \wedge k_2} \EE(v_{l_1}^2)-\prod_{l_1=j+1}^{k_1 \wedge k_2}\EE(v_{l_1})^2) \prod_{l_2=k_1 \wedge k_2+1}^{k_1 \vee k_2} \EE(v_{l_2}) \\
		&\leq \sum_{k_1,k_2 = j+1}^{k_0-l_0} \prod_{l_1=j+1}^{k_1\wedge k_2} \EE(v_{l_1})^2  \prod_{l_2=k_1\wedge k_2+1}^{k_1\vee k_2} \EE(v_{l_2}) (\big( 1 + \frac{C}{n}  \big)^{{k_1+k_2-2j}}-1).
\end{align*}
Using that $ (( 1 + \frac{C}{n}  )^{{k_1+k_2-2j}}-1) \leq \frac{C'(k_1+k_2-2j)}{n}$
and bounding the products by constants, we get
\begin{align*}
\EE(W_j^2)-\EE(W_j)^2 \leq C''  \sum_{k_1,k_2 = j+1}^{k_0-l_0} \frac{k_1+k_2-2j}{n} = 2C'' (j^*-l_0) \sum_{k_1= j+1}^{k_0-l_0} \frac{k_1-j}{n}.
\end{align*}
This leads to 
\[\EE(W_j^2)-\EE(W_j)^2  \leq  O\Big(\frac{(j^*-l_0)^3}{n}\Big),\]
which is the claimed bound 
in the regime  $j \in [k_0-2l_0, k_0-l_0] $.

 For $j \in [(1-\veps)k_0,k_0-2l_0]$,
 by Lemma \ref{alpha}, we have
 \begin{eqnarray*}
  \EE(W_j^2)-( \EE W_j)^2 & = & 
  \sum_{k_1,k_2 = j+1}^{k_0-l_0} 
 \Big( \prod_{l_1=j+1}^{k_1 \wedge k_2} \EE(v_{l_1}^2) 
 -  \prod_{l_1=j+1}^{k_1 \wedge k_2} (\EE v_{l_1})^2\Big)
 \prod_{l_2=k_1 \wedge k_2+1}^{k_1 \vee k_2} \EE v_{l_2}\\
 & \leq&  
\sum_{k_1,k_2 = j+1}^{k_0-l_0} 
 \prod_{l_1=j+1}^{k_1 \wedge k_2} (\EE v_{l_1})^2 \Big( \big(1+O\big(\frac1n\big)\big)^{k_1\wedge k_2-j}-1\Big) 
 \prod_{l_2=k_1 \wedge k_2+1}^{k_1 \vee k_2} \EE v_{l_2}  \\&\leq &
 C \sum_{k_1,k_2 = j+1}^{k_0-l_0} 
\frac{k_1\wedge k_2-j}{n} 
\prod_{l_1=j+1}^{k_1 \wedge k_2} (\EE v_{l_1})^2 
 \prod_{l_2=k_1 \wedge k_2+1}^{k_1 \vee k_2} \EE v_{l_2}.
  \end{eqnarray*}
  Using Lemma \ref{alpha} and the substitutions 
  $r=k_1\wedge k_2, l=k_1\vee k_2$, we obtain therefore that
 \begin{eqnarray*}
  \EE(W_j^2)-( \EE W_j)^2 & \leq  & C
  \sum_{r=j+1}^{k_0-l_0} \sum_{l=r}^{k_0-l_0}
  \frac{r-j}{n} e^{-2(r-j)\sqrt{\frac{r^*}{k_0}}} \cdot
  e^{-(l-r)\sqrt{\frac{l^*}{k_0}}}\\
  &\leq & C' \frac1n \sqrt{\frac{k_0}{l_0^*}}
  \sum_{r=j+1}^{k_0-l_0} (r-j) e^{-2(r-j)\sqrt{\frac{r^*}{k_0}}}
  \leq C'' \frac{k_0^{1/3}}{j^*}.
\end{eqnarray*}
\end{proof}

\begin{proof}[Proof of Proposition \ref{TCL delta bar}]
  Recall from \eqref{eq-bardeltak} that
  for any $k \geq (1-\varepsilon)k_0$,
  \begin{equation}
    \label{eq-Sk}\bdelta_k = \bdelta_{(1-\varepsilon)k_0-1} \prod_{l=(1-\varepsilon)k_0}^k v_l + \sum_{j=(1-\varepsilon)k_0}^{k} u_j \prod_{l=j+1}^k v_l=:
    S_k+\sum_{j=(1-\varepsilon)k_0}^{k} u_j \prod_{l=j+1}^k v_l. \end{equation}
We define for any $k \geq (1-\varepsilon)k_0$
\begin{equation} \label{Gk}
G_k = \sum_{j=(1-\varepsilon)k_0}^k u_j \prod_{l=j+1}^k \EE v_l,
\quad R_k =  \sum_{j=(1-\varepsilon)k_0}^k u_j \Big( \prod_{l=j+1}^k v_l - \prod_{l=j+1}^k \EE v_l \Big),
\end{equation}
which allows us to write
\begin{equation}
\label{eq-sumbardelta}
 \bdelta_k = G_k + R_k +S_k.
\end{equation}
In the next two lemmas, we prove a central limit theorem for $\sum_{k=(1-\varepsilon)k_0}^{k_0-l_0} G_k$ and we prove that $\frac{1}{\sqrt{\log n}}
\sum_{k=(1-\varepsilon)k_0}^{k_0-l_0} (R_k+S_k)$ converges to $0$ in probability.
\begin{lemma} \label{Lemma1}
	We have 
	\[ \EE\Big(\sum_{k=(1-\varepsilon)k_0}^{k_0-l_0} G_k\Big) = \frac{1}{6} \log n+ O(1),
	 \quad
	 \mathrm{Var}\Big(\sum_{k=(1-\varepsilon)k_0}^{k_0-l_0} G_k\Big) =
	 \frac{v}{3} \log n +O(1).  \]
	Furthermore, 
	\[ \frac{1}{ \sqrt{(v/3)\log n}} \Big(  \sum_{k=(1-\varepsilon)k_0}^{k_0-l_0} G_k - \frac{1}{6} \log n \Big) \underset{n\to+\infty}{\leadsto} \gamma.\]
\end{lemma} 
\noindent
\begin{lemma}\label{Lemma2}
	We have
	$$ 
	\sum_{k=(1-\varepsilon)k_0}^{k_0-l_0}( R_k+S_k)
	\underset{n\to +\infty}{\longrightarrow} 0,
	$$
in probability. 
\end{lemma}
These two lemmas combined gives the result of Proposition \ref{TCL delta bar}.
\end{proof}
\begin{proof}[Proof of Lemma \ref{Lemma1}]
  Recall \eqref{eq-Wj} Lemma \ref{Bounds W}. We start by writing
\[\sum_{k=(1-\varepsilon)k_0}^{k_0-l_0} G_k = \sum_{j=(1-\varepsilon)k_0}^{k_0-l_0} u_j \EE(W_j). \]
Hence, using Lemma \ref{Bounds W}  and \eqref{eq:meanu}, we get
\begin{align*}
		\EE\Big(\sum_{k=(1-\varepsilon)k_0}^{k_0-l_0}  G_k\Big)  = & \sum_{j=(1-\varepsilon)k_0}^{k_0-l_0} (\EE u_j) \EE W_j
		=  \sum_{j=(1-\varepsilon)k_0}^{k_0-2l_0}
		\frac{1}{4\sqrt{k_0 j^*}} \sqrt{\frac{k_0}{j^*}} 
	+  O(1)\\
	= &  \frac{1}{4}\log{\varepsilon k_0} - \frac{1}{4} \log(2l_0)
	+O(1)= \frac{1}{6} \log n  + O(1).
\end{align*}

\textbf{Variance computation: } 
By the independence of the $u_j$'s from $W_l$ with $l\geq j$, we have that
\begin{equation*}
\mathrm{Var}\Big(\sum_{j=(1-\varepsilon)k_0}^{k_0-l_0}u_j \EE W_j  \Big) = \sum_{j=(1-\varepsilon)k_0}^{k_0-l_0} \mathrm{Var}(u_j) (\EE W_j)^2.
\end{equation*}
We recall that from Lemma \ref{event} and Lemma \ref{alpha}, we have
\[ \Big( 1+ O\Big( \frac{1}{j}\Big)\Big)\mathrm{Var}(u_j)=\frac{v}{\alpha_{j}^2 j} + \frac{v}{\alpha_{j}^2  \alpha_{j-1}^2 j } = \frac{2v}{j}\Big(1+O\Big(\sqrt{\frac{j^*}{k_0}}\Big)\Big).   \]
We develop the products and we compute each sum. The only term
which eventually contributes is
\begin{align*}
  &\sum_{j=(1-\varepsilon)k_0}^{k_0-l_0}\frac{2v}{j} (\EE W_j)^2  =   \sum_{j=(1-\varepsilon)k_0}^{k_0-2l_0}
  \frac{2v}{j} \frac{1}{4}\Big(\frac{k_0}{j^*} + C\sqrt{\frac{k_0}{j^*}}\varepsilon_j\Big)+
  \sum_{k_0-2l_0}^{k_0-l_0} O\Big(\frac{1}{n}\Big) \varepsilon_j^2\\
& = 
\sum_{j=(1-\varepsilon)k_0}^{k_0-2l_0}
\frac{vk_0}{2jj^*} + O(1)=\frac{v}{2} 
\sum_{j=(1-\varepsilon)k_0}^{k_0-2l_0}
\Big(\frac{1}{j^*}+\frac{1}{j}\Big)+O(1)=\frac{v}{3}\log n+O(1).
\end{align*}
The only term left to bound is
\[ \sum_{j=(1-\varepsilon)k_0}^{k_0-l_0} \frac{2v}{j} O\Big(\sqrt{\frac{j^*}{k_0}}\Big) (\EE W_j)^2. \]
We use the different estimates on $\EE(W_j)$ and obtain for this last sum the
bounds 
\[\sum_{j=(1-\varepsilon)k_0}^{k_0-2l_0}
\frac{1}{j}\sqrt{\frac{j^*}{k_0}} .\frac{k_0}{j^*}
=O(1)\]
and
\[ \sum_{k_0-2l_0}^{k_0-l_0} \frac{1}{j}\sqrt{\frac{j^*}{k_0}} (j^*-l_0)^2 \leq C \frac{l_0^{7/2}}{k_0^{3/2}}
= o(1)
  .\]
This allows us to conclude that
\[ \mathrm{Var}\Big(\sum_{j=(1-\varepsilon)k_0}^{k_0-l_0} G_k\Big) = 
\frac{v}{3} \log n + O(1).\]
\textbf{Central Limit Theorem}
We show the convergence of
$ \mathcal{W}_n=
\sum_{j=(1-\varepsilon)k_0}^{k_0-l_0} X_j/\sqrt{v \log n/3}$
with $X_j= (u_j-\EE(u_j)) \EE(W_j)$
towards a centered Gaussian random variable. This is a direct application of Lindeberg's Central Limit Theorem as we deal with sums of independent random variables. We already know that $\mathrm{Var}(\mathcal W_n)=v\log n/3(1+o(1))$, 
so all we have to check is Lindeberg's condition, i.e. that
\[ \frac{1}{\log n} \sum_{j=(1-\varepsilon)k_0}^{k_0-l_0} \EE\big[ (X_j-\EE(X_j))^2   \Car_{|X_j-\EE(X_j)|\ > \delta \log n} \big] \underset{n\to +\infty}{\longrightarrow} 0.\]
To check this condition, 
we use Markov's inequality
 $$\EE\big[ (X_j-\EE(X_j))^2   \Car_{|X_j-\EE(X_j)|\ > \delta \log n}\big] \leq  \EE(|X_j-\EE(X_j)|^3) / (\delta \log n)$$
along with $\EE(|X_j-\EE(X_j)|^3) \leq C /{j^*}^{3/2}$, which comes from the fact that the (absolute) third moments of the random variables $b_k$'s and $g_k$'s are uniformly bounded in $k$. Thus, Lindeberg's CLT completes the proof.
\end{proof}
\begin{proof}[Proof of Lemma \ref{Lemma2}]
	The proof is a second moment computation. Recall
	the definition of $R_k$ and $S_k$, see \eqref{eq-Sk} and
	\eqref{Gk}. We begin with $S_k$, and develop the product in its definition as
\[ \EE( (\sum_{k=(1-\varepsilon)k_0}^{k_0-l_0} \prod_{l=(1-\varepsilon)k_0}^k v_l)^2 ) =\sum_{k_1, k_2=(1-\varepsilon)k_0}^{k_0-l_0} \prod_{l=(1-\varepsilon)k_0}^{k_1 \wedge k_2} \EE(v_l^2) \prod_{l'=k_1 \wedge k_2+1}^{k_1 \vee k_2} \EE v_{l'}.   \]
We recall, see \eqref{eq-ukvk} and Lemma \ref{event}, 
that there exists a universal constant $C$ such that for any $l \in [(1-\varepsilon)k_0,k_0-l_0]$,
\[ \max\{ \log \EE(v_l), \log \EE(v_l^2) \} \leq -\sqrt{\frac{l^*}{k_0}} + C\frac{l^*}{k_0} . \]
This implies that 
\[  \prod_{l=(1-\varepsilon)k_0}^{k_1 \wedge k_2} \EE(v_l^2) \prod_{l'=k_1 \wedge k_2}^{k_1 \vee k_2} \EE(v_{l'}) \leq e^{-\frac{k_1-(1-\varepsilon)k_0}{2}(\sqrt{\frac{k^*_1}{k_0}} + C\frac{k^*_1}{k_0})}  
e^{-\frac{k_2-(1-\varepsilon)k_0}{2}
(\sqrt{\frac{k^*_2}{k_0}} + C\frac{k^*_2}{k_0})} \]
and finally
\begin{align*}
  &\sum_{k_1, k_2=(1-\varepsilon)k_0}^{k_0-l_0} \prod_{l=(1-\varepsilon)k_0}^{k_1 \wedge k_2} \EE(v_l^2) \prod_{l'=k_1 \wedge k_2}^{k_1 \vee k_2} \EE(v_{l'})  \leq \Big(\sum_{k=(1-\varepsilon)k_0}^{k_0-l_0} e^{-\frac{k-(1-\varepsilon)k_0}{2}(\sqrt{\frac{k^*}{k_0}}+ C\frac{k^*}{k_0})
 }\Big)^2 \\
 & \leq \Big(\sum_{k=(1-\varepsilon)k_0}^{k_0-l_0} e^{-\frac14(k-(1-\varepsilon)k_0)(\sqrt{k^*/k_0})} \Big)^2=:\Big(\sum_{k=(1-\varepsilon)k_0}^{k_0-l_0}
e^{-H(k)}\Big)^2,
\end{align*}
where 
$H(x) =\frac14 (x-(1-\varepsilon)k_0) \sqrt{(k_0-x)/k_0}$.
The function $H(x)\geq 0$ on $x\in [(1-\varepsilon)k_0, k_0-l_0]$,
increases on $[(1-\varepsilon)k_0, (1-\varepsilon/3)k_0)$
and decreases  on $[(1-\varepsilon/3)k_0,k_0]$. 
For $k\in [ (1-\varepsilon)k_0 + 8 \log n / \sqrt{\varepsilon},k_0-l_0]$, 
we thus
have that 
\[H(k) \geq \min \{H((1-\varepsilon)k_0 +
8 \log n / \sqrt{\varepsilon}), H(k_0-l_0)\}\geq \log n,\]
and therefore,
\[ \sum_{k=(1-\varepsilon)k_0}^{k_0-l_0} e^{-H(k)} \leq  \frac{8 \log n}{\sqrt{\varepsilon}}  + \frac{\varepsilon k_0}{n}. \]
We square this bound and we recall that $\EE(\bdelta_{(1-\varepsilon)k_0-1}^2) \leq \frac{C}{n}$, see Lemma \ref{part1second},
to conclude that $S_k$ converges to $0$ in $L^2$. 

We next turn to $R_k$, see \eqref{Gk}, and write
\[ \sum_{k=(1-\varepsilon)k_0}^{k_0-l_0} R_k = 
  \sum_{j=(1-\varepsilon)k_0}^{k_0-l_0} (u_j-\EE(u_j)) (W_j-\EE(W_j) ) + \sum_{j=(1-\varepsilon)k_0}^{k_0-l_0} \EE(u_j) (W_j - \EE(W_j)) = A + B.\]
We bound the second moment of each term.
We use the independence of the $u_j$'s and $W_j$'s to get
\[ \EE(A^2) =  \sum_{j=(1-\varepsilon)k_0}^{k_0-l_0} \mathrm{Var}(u_j) \mathrm{Var}(W_j) \leq \sum_{j=(1-\varepsilon)k_0}^{k_0-l_0} \frac{C_u}{n}\varepsilon'_j
\to_{n\to\infty} 0.   \]

On the other hand, we expand the square to get
\begin{align*}
\EE(B^2) & 
 \leq  \Big(\sum_{j=(1-\varepsilon)k_0}^{k_0-l_0} \EE u_{j} \mathrm{Var}(W_{j}) \Big)^2.
\end{align*}
We use again the fact that $\EE u_j \leq C_u/n$ and the bound on $\mathrm{Var}(W_j)$ from Lemma \ref{Bounds W} to obtain
\[ \sum_{j=(1-\varepsilon)k_0}^{k_0-l_0} \EE u_{j} \mathrm{Var}(W_{j})
\leq C\sum_{j=(1-\varepsilon)k_0}^{k_0-2l_0} \frac{1}{k_0^{2/3} j^*}+
\sum_{j=k_0-2l_0}^{k_0-l_0} \frac{(j^*-l_0)^3}{n^2} =o(1).
\]
Thus,
$ \EE(A^2+B^2)\to_{n\to\infty} 0$,
 which completes the proof of the lemma.
\end{proof}

\begin{proof}[Proof of Proposition \ref{Limit Delta bar}]
Recall \eqref{Deltabar}, which gives that
\begin{equation}
  \label{eq-contbarD}
  \sum_{k=(1-\varepsilon)k_0}^{k_0-l_0} \overline{ \Delta_k} = \sum_{j=(1-\varepsilon)k_0}^{k_0-l_0} -v_j \overline{ \delta_{j-1}}^2 W_j + \overline{\Delta}_{(1-\varepsilon)k_0} \sum_{k=(1-\varepsilon)k_0}^{k_0-l_0} \prod_{l=(1-\varepsilon)k_0}^k v_l .  
\end{equation}
Because 
$\EE(\overline{\Delta}_{(1-\varepsilon)k_0} ^2) \leq C/n$, see Lemma 
\ref{part1second}, the same argument that gave the control on $S_k$ in
the proof of Lemma \ref{Lemma2} gives that
\[ \overline{\Delta}_{(1-\varepsilon)k_0} \sum_{k=(1-\varepsilon)k_0}^{k_0-l_0} \prod_{l=(1-\varepsilon)k_0}^k v_l  \xrightarrow[n \to \infty]{L^2} 0. \]
Hence we only have to estimate the first term in the right
hand side of \eqref{eq-contbarD}. The next lemma, whose proof is deferred,
will be essential in the computation of the mean and variance of that term.
Recall \eqref{eq-bardeltak}.
\begin{lemma}\label{lemma3} For $ k\in [(1-\varepsilon)k_0,k_0-l_0]$ we have that
\begin{equation}
  \EE( \bdelta_{k-1}^2) =\frac{v}{2\sqrt{k_0 k^*}}+ O\big(\frac{1}{k_0}\big),
\end{equation}
and 
\[ 
\EE( \bdelta_k^4  ) = O\Big(\frac{1}{k_0k^*}\Big).\]
\end{lemma}
\noindent
\textbf{The first moment}
We show that 
\begin{equation}
  \label{eq-barDkfirst}
  \sum_{j=(1-\varepsilon)k_0}^{k_0-l_0} \EE(\overline{ \Delta_j}) = -\frac{v}{6} \log n + o(\sqrt{\log n}). 
\end{equation}
Using the independence between the $u_j$'s and $v_j$'s for different indices, we have to compute
\begin{align*}
\EE\Big[ \sum_{j=(1-\varepsilon)k_0}^{k_0-l_0} &  -\EE(v_j)\EE(\bdelta_{j-1}^2) \EE(W_j)\Big] \\
& =  -\sum_{j=(1-\varepsilon)k_0}^{k_0-2l_0} 
\Big(1+O\Big(\sqrt{\frac{j^*}{k_0}}\Big)\Big)\Big(\frac{v \sqrt{k_0}}{2j \sqrt{j^*}}+ O\Big(\frac{1}{k_0}\Big)\Big)\Big(\frac{1}{2} \sqrt{\frac{k_0}{j^*}}+ \varepsilon_j)  \\
& \quad \quad \quad  + \sum_{j=k_0-2l_0}^{k_0-l_0} \Big(1+O\Big(\sqrt{\frac{j^*}{k_0}}\Big)\Big)O(\frac{1}{\sqrt{j^*k_0}}) \varepsilon_j
=- \frac{v}{6}\log n+O(1),
\end{align*}
where the first equality used \eqref{eq-ukvk},
Lemma \ref{Bounds W} and
Lemma \ref{lemma3}, while
the last equality follows from Lemma \ref{Bounds W}.

\noindent
\textbf{Variance:}
We show that 
\[ \mathrm{Var}\Big(\sum_{k=(1-\varepsilon)k_0}^{k_0-l_0} \overline{ \Delta_k}\Big) = o(\log n).  \]
Using that $W_j$ is independent of $\{v_i \bar \delta_{i-1}^2, i\leq j\}$,
we write
\begin{equation*}
\mathrm{Var}\Big( \sum_{j=(1-\varepsilon)k_0}^{k_0-l_0} v_j  \bdelta_{j-1}^2 W_j \Big) =  \sum_{j=(1-\varepsilon)k_0}^{k_0-l_0} \mathrm{Var}(v_j  \bdelta_{j-1}^2) \EE(W_j^2) 
+ \EE\Big(\sum_{j=(1-\varepsilon)k_0}^{k_0-l_0} \EE(v_j \bdelta_{j-1}^2) (W_j-\EE(W_j))  \Big)^2
\end{equation*}
We use the bounds from Lemma \ref{Bounds W} and \ref{lemma3} and we have
\begin{align*}
&\sum_{j=(1-\varepsilon)k_0}^{k_0-l_0} \mathrm{Var}(v_j \bdelta_{j-1}^2) \EE(W_j^2)   \lesssim \sum_{j=(1-\varepsilon)k_0}^{k_0-l_0}  \EE( \bdelta_{j-1}^4) \EE(W_j^2)  \\
& \lesssim  \sum_{j=(1-\varepsilon)k_0}^{k_0-2l_0} \frac{1}{(k_0j^*)} \frac{k_0}{j^*} 
+ \sum_{k_0-2l_0}^{k_0-l_0}  \frac{1}{(k_0j^*)} (j^*-l_0)^2
\end{align*}
which goes to zero as $n\to\infty$.
For the last term, we expand the square and 
the Cauchy-Schwarz inequality to obtain
\begin{align*}
&\EE\Big(\sum_{j=(1-\varepsilon)k_0}^{k_0-l_0} (\EE v_j  \bdelta_{j-1}^2) (W_j-\EE W_j)  \Big)^2 
 \leq \Big(\sum_{j = (1-\varepsilon)k_0}^{k_0-l_0} \frac{C}{\sqrt{k_0 j^*}} (\mathrm{Var} \, W_{j})^{1/2} \Big)^2.
\end{align*}
Using the estimates from Lemma \ref{Bounds W}, we get
 \[ \sum_{j = (1-\varepsilon)k_0}^{k_0-l_0} \frac{C}{\sqrt{k_0 j^*}} 
  (\mathrm{Var}\, W_{j})^{1/2}
  \leq \sum_{j = (1-\varepsilon)k_0}^{k_0-2l_0}
  \frac{Ck_0^{1/6} }{\sqrt{k_0 (j^*)^2}} 
  + \sum_{j =k_0-2l_0}^{k_0-l_0} 
  \frac{C}{\sqrt{n k_0 l_0}} (j^*-l_0)^{3/2}=o(1).
\]
Combining the different bounds presented here, we deduce that 
\[ \frac{1}{\log n} \mathrm{Var}\Big(\sum_{k=(1-\varepsilon)k_0}^{k_0-l_0} 
\overline{\Delta_k}\Big) \xrightarrow[n \to \infty ]{} 0 \]
which implies, together with \eqref{eq-barDkfirst}
and Markov's
inequality, that
\[ \frac{1}{\sqrt{\log n}}\Big( \sum_{k=(1-\varepsilon)k_0}^{k_0-l_0} 
\overline{\Delta_k} + \frac{v}{6} \log n \Big) \underset{n\to +\infty}{\longrightarrow} 0, \]
in probability.
\end{proof}
\begin{proof}[Proof of Lemma \ref{lemma3}]
  Recall from \eqref{eq-bardeltak} that for all $k \leq k_0$
	\[  \bdelta_{k} = \sum_{j=(1-\varepsilon)k_0}^k u_j \prod_{l=j+1}^k v_l = \sum_{j=(1-\varepsilon)k_0}^k (u_j-\EE u_j) \prod_{l=j+1}^k v_l +\sum_{j=(1-\varepsilon)k_0}^k \EE u_j \prod_{l=j+1}^k v_l . \]
	To obtain a good bound on $\EE( \bdelta_k^2)$, we expand the square. Only the first term will contribute, the other two being negligible. For the first term, we use the independence between $u_j$ and $\prod_{l>j} v_l$ to get
	\begin{align*}
	\EE((\sum_{j=(1-\varepsilon)k_0}^k (u_j-\EE u_j) \prod_{l=j+1}^k v_l  )^2) = \sum_{j=(1-\varepsilon)k_0}^k \mathrm{Var}(u_j) \prod_{l=j+1}^k \EE(v_l^2).
	\end{align*}
	From Lemma \ref{event}, we have
	\[\mathrm{Var}(u_j)=
	\frac{2v}{j}(1+O(\sqrt{j^*/k_0})).   \]
	Combining this estimate with the second part of 
	Lemma \ref{lem-aux} (with $\alpha=4$) and Lemma \ref{alpha}
	leads to
	\begin{eqnarray*}
	  \sum_{j=(1-\varepsilon)k_0}^k  \mathrm{Var}(u_j) \prod_{l=j+1}^k 
	\EE(v_l^2) &=& 2v  \sum_{j=(1-\varepsilon)k_0}^k 
	\Big(\frac{1}{j}+O\big(\sqrt{\frac{j^*}{k_0}}\big)\Big)
	\prod_{l=j+1}^k \Big(1-4\sqrt{\frac{l^*}{k_0}}+O\big(\frac{l^*}{k_0}\big)\Big)\\
	&=&
	\frac{1}{2\sqrt{k_0 k^*}}+O\big(\frac{1}{k_0}\big).\end{eqnarray*}
	
	We prove now that the other terms in the development of the square of 
	$\EE \bdelta_{k}^2$ are of lower order.
	We start with
	\[ \EE((\sum_{j=(1-\varepsilon)k_0}^k \EE( u_j) \prod_{l=j+1}^k v_l)^2 ) \leq \frac{C_u^2}{n^2} \EE(W_k^2)= O(1/k_0), \]
	thanks to Lemma \ref{Bounds W}.
	To end the proof of the bound on $\EE ( \bdelta_{k}^2)$, we have to bound
	\begin{align*}
	\EE \Big( \sum_{j_1,j_2=(1-\varepsilon)k_0}^k & (u_j-\EE(u_j)) \EE (u_{j_2})\prod_{l_1=j_1+1}^k v_{l_1}  \prod_{l_2=j_2+1}^k v_{l_2} \Big)\\
	& = \sum_{j_1,j_2=(1-\varepsilon)k_0, j_1>j_2}^k \EE(u_{j_2}) \prod_{l=j_2+1}^{j_1-1} \EE(v_l) \EE(u_{j_1} v_{j_1}) \prod_{l'=j_1+1}^{k} \EE(v_{l'}^2 )
	\end{align*}
	where we used the independence between the different terms. An elementary computation shows that $\EE(u_jv_j) = \EE(u_j)\EE(v_j) + \mathrm{Var}(v_j) \leq C/ (k_0\alpha_j^2)$. Combining this observation with the previously established bounds leads to
	\[ \EE \Big( \sum_{j_1,j_2=(1-\varepsilon)k_0}^k  (u_j-\EE(u_j)) \EE (u_{j_2})\prod_{l_1=j_1+1}^k v_{l_1}  \prod_{l_2=j_2+1}^k v_{l_2} \Big) \leq \frac{C}{k_0^2}  \EE(W_k^2) \leq O\Big(\frac{1}{k_0}\Big) .  \]
	
\noindent \textbf{Bound on $\EE( \bdelta_{k}^4)$:}
We start from the decomposition 
\[ \overline{ \delta_{k}} = \sum_{j=(1-\varepsilon)k_0}^k u_j \prod_{l=j+1}^k v_l = \sum_{j=(1-\varepsilon)k_0}^k (u_j-\EE( u_j)) \prod_{l=j+1}^k v_l +\sum_{j=(1-\varepsilon)k_0}^k \EE (u_j) \prod_{l=j+1}^k v_l . \]
and we bound the fourth moment of each term.
For the first term, we expand the product and group the terms with same index, using the notation $\tilde{u}_j=u_j-\EE (u_j)$. This quadruple sum over the indices $j_1 \leq j_2 \leq j_3 \leq j_4$ can be decomposed as the sum of the following:
\begin{align*}
&\sum_{j=(1-\varepsilon)k_0}^k \EE(\tilde{u}_j^4) \prod_{l=j_1+1}^k \EE(v_l^4),\qquad
j_1=j_2=j_3=j_4, 
\\
&\sum_{j<j_4=(1-\varepsilon)k_0}^k \EE(\tilde{u}_j^3) \prod_{l_1=j+1}^{j_4-1} \EE(v_{l_1}^3) \EE(\tilde{u}_{j_4} v_{j_4}^3) \prod_{l_2=j_4+1}^k\EE(v_{l_2}^4), \qquad j_1=j_2= j_3 <j_4, \\
&\sum_{j_1<j_3 =(1-\varepsilon)k_0}^k \EE(\tilde{u}_{j_1}^2) \prod_{l_1=j_1+1}^{j_3-1} \EE(v_{l_1}^2) \EE(\tilde{u}_{j_3}^2v_{j_3}^2) \prod_{l_3=j_3+1}^{k} \EE(v_{l_3}^4),\qquad j_1=j_2 < j_3=j_4 ,  \\
& \kern-4mm\sum_{j_1<j_3<j_4=(1-\varepsilon)k_0} \kern-7mm \EE(\tilde{u}_{j_1}^2) \prod_{l_1=j_1+1}^{j_3-1}\kern-2mm \EE(v_{l_1}^2) \EE(\tilde{u}_{j_3}^2v_{j_3}^2) \kern-2mm \prod_{l_3=j_3+1}^{j_4-1}\kern-2mm \EE(v_{l_3}^3)\EE(\tilde{u}_{j_4}^2v_{j_4}^3)\prod_{l_4=j_4+1}^k\EE(v_{l_4}^4), 
j_1=j_2 < j_3 < j_4.
\end{align*}
(The terms corresponding to $j_1<j_2$ do not appear because their value is $0$ as $\EE(\tilde{u}_{j_1})=0$.) We note
the following
elementary bounds, which follow
from the definitions of $u_j$ and $v_j$:
\begin{align*}
\EE(v_l)  \leq \frac{1+O(1/n)}{\alpha_{l}^2} , \;
\EE(v_l^2)  \leq \frac{1+O(1/n)}{\alpha_l^4} , \;
\EE(v_l^3)   \leq \frac{1+O(1/n)}{\alpha_l^6} ,\;
\EE(v_l^4)  \leq \frac{1+O(1/n)}{\alpha_l^8} .	
\end{align*}
We deduce from Lemma \ref{alpha} that we can write for any $c>0$, with $n$ large enough
\[ \prod_{l={j+1}}^k \frac{1 + C/l}{\alpha_l^c} \leq e^{-c(k-j)(\sqrt{k^*/k_0} + C/k_0)} \leq  e^{-\frac{c}{2}(k-j)\sqrt{k^*/k_0}}. \]
From this bound, we obtain the estimate which will be used intensively in this lemma
$\sum_{j=(1-\varepsilon)k_0}^k e^{-\frac{c}{2}(k-j)\sqrt{k^*/k_0}} \leq C \sqrt{k_0/k^*}$.
The following bounds are also elementary
\begin{align*}
& \exists C >0, \forall (1-\varepsilon)k_0 \geq j \leq k_0, \quad  \EE(\tilde{u}_j^2) \leq \frac{C}{n} , \quad  \EE(\tilde{u}_j^3) \leq \frac{C}{n^{3/2}}  ,\quad \EE(\tilde{u}_j^4) \leq \frac{C}{n^2}  \\
& \EE(\tilde{u}_j v_j^3) \leq \frac{C}{\alpha_j^6 n} ,\quad 
\EE(\tilde{u}_j^2 v_j^2) \leq \frac{C}{\alpha_j^4 n} ,\quad 
\EE( \tilde{u}_j^3 v_j ) \leq \frac{C}{\alpha_j^4 n^{3/2}}.
\end{align*}
We use these inequalities in each of the four sums and we get, up to multiplicative constants
\begin{align*}
 &\sum_{j=(1-\varepsilon)k_0}^k \frac{1}{n^2} e^{-4(k-j)\sqrt{k^*/k_0}} \lesssim \frac{ \sqrt{k_0}}{n^2 k^*} = O\Big(\frac{1}{k_0 k^*}\Big), 
 \quad j_1=j_2=j_3=j_4, \\
&      \sum_{j<j_4=(1-\varepsilon)k_0}^k \frac{1}{n^{3/2}} \prod_{l_1=j+1}^{j_4-1} \frac{1}{\alpha_{l_1}^6} \frac{1}{\alpha_{j_4}^6n} \prod_{l_2=j_4+1}^k\frac{1}{\alpha_{l_2}^8} \lesssim \frac{1}{n^{5/2}} \frac{k_0}{k^*} = O\Big(\frac{1}{k_0k^*}\Big) ,\quad j_1=j_2= j_3 <j_4,  \\
 &    \sum_{j_1<j_3 =(1-\varepsilon)k_0}^k \frac{1}{n} \prod_{l_1=j_1+1}^{j_3-1} \frac{1}{v_{l_1}^4} \frac{1}{\alpha_{j_3}^4 n} \prod_{l_3=j_3+1}^{k} \frac{1}{\alpha_{l_3}^8} \lesssim \frac{1}{n^2} \frac{k_0}{k^*} = O\Big(\frac{1}{k_0 k^*}\Big),\quad j_1=j_2 < j_3=j_4 ,\\
&  \sum_{j_1<j_3<j_4=(1-\varepsilon)k_0} \kern-3mm \frac{1}{n} \prod_{l_1=j_1+1}^{j_3-1} \frac{1}{\alpha_{l_1}^4} \frac{1}{\alpha_{j_3}^4 n} \prod_{l_3=j_3+1}^{j_4-1} \frac{1}{\alpha_{l_3}^6}\frac{1}{n\alpha_{j_3}^6}\prod_{l_4=j_4+1}^k\frac{1}{\alpha_{l_4}^8} \lesssim \frac{1}{n^3} \frac{k_0^{3/2}}{{k^*}^{3/2}} = O\Big(\frac{1}{k_0k^*}\Big),\\
&\kern60mm \qquad\qquad \qquad \qquad \qquad j_1=j_2 < j_3 < j_4.
\end{align*}
Each of these terms is at most of order $1/(k_0k^*)$, which is sufficient for this bound.

For the second term, we expand the product, and bound each term using the previous bounds
	which, along with $\EE(u_j) \leq C_u/n$, gives
\begin{align*}
\EE\Big( \Big(\kern-2mm\sum_{j=(1-\varepsilon)k_0}^k \kern-2mm \EE (u_j) \prod_{l=j+1}^k v_l\Big)^4 \Big) &= 4!\kern-6mm\sum_{j_1\leq j_2 \leq j_3 \leq j_4  = (1-\varepsilon)k_0}^k \kern-2mm \frac{C_u^4}{n^4} \prod_{l=j_1+1}^{j_2-1}\kern-1mm\EE(v_{l_1})\prod_{l=j_2}^{j_3-1}\EE(v_{l_2}^2)\prod_{l=j_3}^{j_4-1}\EE(v_{l_3}^3)\prod_{l=j_4}^{k}\EE(v_{l_4}^4)\\
& \leq \Big(\sum_{j=(1-\varepsilon)k_0}^k \frac{C_u}{n} e^{-(k-j)(\sqrt{k^*/k_0} + O(k^*/k_0))}\Big)^4 \leq \frac{C}{k_0^2 {k^*}^2}.
\end{align*}
\end{proof}

\subsection{Deterministic bounds on good events}
\label{sec-det}
The purpose of this section is to show that, with a probability going to $1$ as $n\to\infty$, the approximation of $\delta_k$ by $\overline{ \delta_k}$ is good enough. 
This lemma will also be a starting point for the next section, where we need a precise control over the amplitude of $\delta_{k_0-l_0}$.

For any $i \in [1,\varepsilon^{3/2} k_0]$, we define $s_i=k_0- i^{2/3}k_0^{1/3}$ and its inverse function $i(k)$, defined as $\max\{i , s_i \geq k\}$.
Introduce the event
\begin{equation}
\label{eq-OCG}\Omega_{C}= \bigcap_{i=1}^{\varepsilon^{3/2} k_0} 
\bigcap_{k=s_{i+1}}^{s_i} \Big\{  |\bar \delta_k|
\leq C\frac{\log i+10}{(k_0s_{i}^*)^{1/4}}\Big\}.
\end{equation}

\begin{lemma}[Approximation lemma] \label{Deterministic delta}
	For any $\gamma$, there exists a constant $C$ so that 
	\[\PP(\Omega_{C})\geq 1-\gamma.\]
\end{lemma}
\begin{proof}
	This lemma is just the combination of Lemmas
	\ref{DeterministicLemma} and \ref{DeterministicboundR} below.
\end{proof}

For $C>0$, we control the magnitude of all $|G_k|$ from \eqref{Gk}
through the event
\begin{equation}
\label{eq-OCG1}
\widetilde \Omega_{C}= \bigcap_{i=1}^{\varepsilon^{3/2} k_0} 
\bigcap_{k=s_{i+1}}^{s_i} \Big\{|G_{k}| 
\leq C\frac{\log i+10}{(k_0s_{i}^*)^{1/4}}\Big\}.
\end{equation}
\begin{lemma}[Control on $\widetilde 
	\Omega_{C}$] \label{DeterministicLemma}
	For any $\gamma>0$, there exists a constant $C$ such that 
	$\mathbb{P}(\widetilde
	\Omega_{C})\geq 1-\gamma$.
\end{lemma}
Before proving Lemma \ref{DeterministicLemma}, we recall the following fact.
\begin{lemma}
	\label{lem-subgaussian}
	Assume that  $x_j$ are independent, zero mean random
	variables of variance $1$,
	satisfying $\EE(e^{\eta |x_j|})< \bar c$ for some $\bar c$, $\eta$ independent  of $j$.
	Let $\beta_j\geq 0$ be deterministic constants bounded above 
	by $\bar \beta\leq 1$,
	and set $W=\sum_{i=1}^n \beta_j x_j$ and $\sigma^2=\sum_{i=1}^n \beta_j^2$. 
	Then, there exists a constant $c=c(\eta,\bar c)>0$ such that  
	\begin{equation}
	\label{mod-dev}
	\PP(W>t)\leq  
	e^{-c t^2/\sigma^2}+
	e^{-c t/\bar \beta}.
	\end{equation}
\end{lemma}
\begin{proof}
	We note first that there exists a constant $c=c(\eta)$ so that 
	for $z\in \mathbb{R}_+$ we have that $z^3 \leq  c e^{\eta z/2}$. Therefore,
	for all $0<\lambda< \eta/(2 \bar \beta)$,
	\begin{align*}
	\EE(e^{\lambda \beta_j x_j})
	&\leq 1+ \frac12 \lambda^2 \beta_j^2+\frac{\lambda^3 \beta_j^3 }{6}\EE \big( |x_j|^3
	e^{\lambda \beta_j |x_j|}\big)\\
	& \leq 1+\frac12 \lambda^2 \beta_j^2+  \frac{c\eta
		\lambda^2\beta_j^2}{12}\EE e^{(\eta/2+\lambda \beta_j )|x_j|}\leq e^{C(\eta,\bar c)
		\lambda^2  \beta_j^2}. 
	\end{align*}
	The conclusion now follows from \cite[Theorem 3.4.15]{Petrov}.
\end{proof}
\begin{proof}[Proof of Lemma \ref{DeterministicLemma}] 
	We consider the blocks $[s_{i+1},s_i]$, 
	control first the values of $G_{s_i}$ and then use chaining
	within the block
	to control  the maximum 
	of the increments $ |G_{s_i}-G_k|$ for $k$ in the block.
	Intuitively, the proof is sharp
	because the variables $G_k$ are strongly correlated with a
	block $k\in[s_{i+1},s_i]$ and weakly correlated across non
	adjacent blocks.
	Throughout the
	proof, we assume that $C>1$ so that $C^2>C$. \\
	\textbf{Step 1: Control at blocks' endpoints.}
	We begin by noting that 
	$ \mathrm{Var}(G_k) \leq C'/\sqrt{k_0k^*}. $
	Indeed, see \eqref{eq-Euk} and Lemma \ref{alpha},
	\[\mathrm{Var}(G_k) = \sum_{j=(1-\varepsilon)k_0}^k \mathrm{Var}(u_j) \prod_{l=j+1}^k \EE(v_l)^2 \leq \sum_{j=(1-\varepsilon)k_0}^k \frac{C_u}{n} e^{-c(k-j) \sqrt{k^*/k_0}} \leq \frac{C'}{\sqrt{k_0k^*}}. \]
	We now apply Lemma 
	\ref{lem-subgaussian}, 
	noting from \eqref{Gk} that we can write
	$G_k = \sum_{j=(1-\veps) k_0}^k x_j \beta_j$ with $\beta_j\leq c/\sqrt{k_0}:=\bar \beta$ 
	and $x_j$ satisfying the conditions of the lemma with 
	$\eta=O(1)$, namely
	$x_j=(u_j-\EE u_j)/\mathrm{Var}(u_j)$. Writing 
	$t=C(\log i+10)/(k_0k^*)^{1/4}$ and $G=C'/\sqrt{k_0k^*}$, 
	we obtain that there exists a constant $C_2$ (independent of $i,k$)
	such that
	for $n$ large enough and $k \in [s_{i+1},s_i]$,
	\[\mathbb{P}\Big(\frac{G_{k}(k_0k^*)^{1/4}}{\log i+10} > C \Big) \leq  
	e^{-c C^2(\log i+10)^2/C'}+e^{-c C\sqrt{k_0}(\log i+10)/(k_0 k^*)^{1/4}}
	\leq 
	e^{-C_2 C  (\log i+10)}.\]
	In particular, we conclude from the union bound
	that for any
	$\delta >0$, there exists a constant $C=C(\delta)>0$ with 
	$CC_2>1$,
	such that with probability at least $1-\delta$,
	\[ \sup_{i \in [1,\veps^{3/2} k_0]} \frac{G_{s_i}(k_0s_i^*)^{1/4}}{\log i+10} \leq C .\]
	\noindent
	\textbf{Step 2: Variance of the increments within blocks.}
	We prove that there exists a universal constant 
	$\bar C$ such that, within a block $s_{i+1} \leq a< b \leq s_{i}$,
	\begin{equation} \label{vraincre}\Var((k_0 s_i^*)^{1/4}(G_{a} - G_{b}))
	  \leq \bar C(b-a) \sqrt{{s_i^*}/{k_0}}.\end{equation}
	Indeed, 
	we decompose the increments as
	\begin{align*}
	G_a-G_b & = \sum_{j=(1-\varepsilon)k_0}^a u_j \prod_{l=j+1}^a \EE(v_l) (1-\prod_{t=a+1}^b \EE(v_t)) + \sum_{j=a+1}^b u_j \prod_{l=j+1}^b \EE(v_l)
	\end{align*}
	and bound now the variance of each term. For the second term, 
	we just bound the product term by $1$ 
	to obtain
	\[ \Var\Big(\sum_{j=a+1}^b u_j \prod_{l=j+1}^b \EE(v_l) \Big) \leq \frac{{C_u}(b-a)}{k_0}.  \]
	For the first term, we use that
	\[1-\prod_{t=a+1}^b \EE(v_t)  \leq  c (b-a) \sqrt{\frac{a^*}{k_0}}  \]
	which leads, after some computation 
	using Lemma \ref{alpha} and the fact that for $s_{i+1}\leq 
	a<b\leq s_{i}$, we have that $a^*\leq k_0^{1/3} (i+1)^{2/3}$ while
	$b-a\leq s_i-s_{i+1}\leq k_0^{1/3}i^{2/3} ( (1+1/i)^{2/3}-1)$
	and therefore
	$\sqrt{a^*/k_0} \leq 1/(b-a)$, to
	\begin{align*}
	  \text{Var}\Big( \sum_{j=(1-\varepsilon)k_0}^a u_j \prod_{l=j+1}^a \EE(v_l) (1-\prod_{t=a+1}^b \EE(v_t)) \Big) \leq \bar C \frac{(b-a)}{k_0}.
	\end{align*}
	
	\noindent\textbf{Step 3: End of the proof.}
	We show that within a block $[s_{i+1},s_i]$ the 
	increments of the process cannot be too big. 
	More precisely, we show that 
	\begin{equation}
	\label{eq-sumAi}
	\sum_{i=1}^{\veps^{3/2} k_0} A_i:=\sum_i \mathbb{P}\Big(\max_{ k\in[s_{i+1}, s_i]}
	\frac{|G_{k}-G_{s_{i}}|}{\log i+10} > \frac{C}{(k_0s_i^*)^{1/4}} \Big) \underset{C \to \infty}{\longrightarrow}  0.
	\end{equation}
	
	Fix $i$. 
	For any integer $0<l \leq  \log_2(s_{i}-s_{i+1})$, 
	we divide the interval $[s_{i+1},s_i]$ into $2^l$ 
	dyadic intervals of length $2^{-l} (s_i-s_{i+1})$ with endpoints
	$m_p^l$. As $G_k -G_{s_i}$ is the sum of at most $\log_2(s_i -s_{i+1})$ terms of the form $G_{m_{p+1}^l}- G_{m_{p}^l}$, we deduce employing
	a union bound
	and  $\sum_{l\geq 1} l^{-2} = \pi^2/6$,  that
\begin{align*}
	  A_i  &\leq \mathbb{P}\Big( \exists l \leq \log_2(s_i-s_{i+1}) , \exists p \leq 2^l-1 \text{ such that } \frac{|G_{m_p^l} - G_{m_{p+1}^{l}}|}{ \log i + 10} > \frac{6C}{\pi^2l^2(k_0s_i^*)^{1/4}}\Big) \\
	  & \leq \sum_{l=1}^{\log_2(s_i-s_{i+1})} \sum_{p=1}^{2^l-1}
	  \mathbb{P}\Big( \frac{|G_{m_p^l} - G_{m_{p+1}^{l}}|}{ \log i + 10} > \frac{6C}{\pi^2l^2(k_0s_i^*)^{1/4}} \Big).
	  \end{align*}
Using Lemma	\ref{lem-subgaussian} and \eqref{vraincre}, we deduce for some
universal constant $c>0$,
\begin{align*}
	  A_i 
	& \leq \sum_{l=1}^{\log_2(s_i-s_{i+1})} 2^l \exp\Big( -\frac{cC^2}{l^4} (\log i +10)^2 \frac{2^l}{(s_i-s_{i+1})}\sqrt{\frac{k_0}{s_i^*}}\Big)
	\\ &\quad
	+\sum_{l=1}^{\log_2(s_i-s_{i+1})} 2^l 
	\exp\Big(-\frac{cC}{l^2}(\log i+10) \Big(\frac{k_0}{s_i^*}\Big)^{1/4}\Big).
\end{align*}
Using that $k_0/(s_i -s_{i+1})\asymp \sqrt{k_0 s_i^*}$ and 
$s_i^* = k_0^{1/3} i^{2/3}$, we find that
	\[A_i  \leq \sum_{l=1}^{+ \infty} 2^l \exp\{-\frac{c'C^2}{l^4} (\log i +10) 2^l  \}+  M(i,k_0/i).\]
	where, for $x\geq 2 $, 
	$M(i,x)=c'x^{1/3}(\log x+1)e^{-Cx^{1/6} (\log i+10)/c'(\log x)^2}$, $c'$ is a universal constant, and we used that
	$k_0/s_i^*\asymp (k_0/i)^{2/3}$.
	It
	is now straight forward to check, using dominated convergence, that
	for a universal $c''$,
	$\sum_i A_i\leq c''e^{- C/{c''}}\to_{C\to\infty} 0$, showing
	\eqref{eq-sumAi} and completing the proof of the lemma.
\end{proof}
\noindent
The next lemma complements 
Lemma \ref{DeterministicLemma}
by controlling  $S_k$ and $R_k$, recall \eqref{eq-Sk}, \eqref{Gk}.
\begin{lemma}[Bounds on $R_k$ and $S_k$] \label{DeterministicboundR}
	With notation as above, we have that
	\begin{equation}
	\label{eq-Rkbound}
	\PP\Big( \exists k\in [(1-\varepsilon)k_0, k_0-l_0  ],  |R_k| > \frac{\log(i(k))+10}{(k_0k^*)^{1/4}}\Big) \underset{n\to +\infty}{\longrightarrow}  0 
	\end{equation}
	and 
	\begin{equation}
	\label{eq-Skbound}
	\PP\Big( \exists k\in [(1-\varepsilon)k_0, k_0-l_0  ], |S_k| > \frac{\log(i(k))+10}{(k_0k^*)^{1/4}}\Big) \underset{n\to +\infty}{\longrightarrow}  0 .
	\end{equation}
\end{lemma}
\begin{proof}
	We begin with the control of 
	\[S_k=\bdelta_{(1-\varepsilon)k_0-1}
	\prod_{l=(1-\varepsilon)k_0}^k v_l=:
	\bdelta_{(1-\varepsilon)k_0-1}  M_k. \]
	For an appropriate $c(\varepsilon)$, let
	\begin{equation} 
	\label{eq-M_k}
	\mathcal{M}:=\Big\{ \forall k \in [(1-\veps)k_0, k_0], \ 
	M_k \leq 
 e^{-c(\varepsilon) (k-(1-\varepsilon) k_0)}
	\Big\}
	\end{equation}
	Using \eqref{bounvlOmega}, we deduce that on the event $\Omega_n$, $\mathcal{M}$ occurs and therefore $\PP(\mathcal M^\complement)\to_{n\to\infty} 0$. 
	Also,
	since
	$\EE(\bdelta_{(1-\varepsilon)k_0-1}^2) \leq C/n$,
	see Lemma \ref{part1second}, we have that $\PP(|\bdelta_{(1-\varepsilon)k_0-1}|>
	{\log n}/{\sqrt{n}})\to_{n\to\infty} 0,$
	and therefore
	\[ \PP\Big(\! \exists k\in [(1-\varepsilon)k_0, k_0-l_0  ], 
	|S_k| > \!\frac{\log (i(k))+10}{(k_0k^*)^{1/4}}\Big) 
	\!\leq \!\PP\Big(|\bdelta_{(1-\varepsilon)k_0-1}|>\!
	\frac{\log n}{\sqrt{n}}\Big) + \PP(\mathcal{M}^\complement)\to_{n\to\infty} 0,
	\]
	which completes the proof of  \eqref{eq-Skbound}. Turning to the proof of
	\eqref{eq-Rkbound}, we write
	\[ \PP\Big( \exists k \in [(1-\varepsilon)k_0, k_0-l_0  ], 
	|R_k| > \frac{\log (i(k))+10}{(k_0k^*)^{1/4}}\Big) 
	\leq \sum_{k=(1-\varepsilon)k_0}^{k_0-l_0} 
	\frac{(k_0k^*)^{1/2}}{(\log (i(k))+10)^2} \EE(R_k^2). \]
	We thus need to bound 
	$\EE (R_k^2)$. 
	We write $u_j = (u_j-\EE(u_j)) + \EE(u_j)$  and
	$w_{j,k} = \prod_{l=j+1}^k v_l$, 
	and recall that
	\begin{equation}
	\label{eq-Rkdec}
	R_k =  \sum_{j=(1-\varepsilon)k_0}^k (u_j-\EE (u_j)) (w_{j,k}-\EE (w_{j,k}))
	+ \sum_{j=(1-\varepsilon)k_0}^k \EE (u_j)(w_{j,k}-\EE (w_{j,k})).
	\end{equation}
	We compute the second moment of both terms in \eqref{eq-Rkdec} separately.
	The first term gives
	\begin{align*}
	\EE\Big(  \sum_{j=(1-\varepsilon)k_0}^k (u_j-\EE(u_j)) 
	(w_{j,k}-\EE (w_{j,k}))\Big)^2
	=\sum_{j=(1-\varepsilon)k_0}^k \mathrm{Var}(u_j) \mathrm{Var}(w_{j,k}).
	\end{align*}
	First, we note that by \eqref{eq-ukvk} we have that
	\[ \mathrm{Var}(w_{j,k})
	\leq  
	\Big(\prod_{l=j+1}^k\frac{1}{\alpha_l^4}\Big) 
	\EE\Big( \prod_{l=j+1}^k  (1-c_l+g_l/\sqrt{l}) - 
	\prod_{l=j+1}^k  (1-c_l)\Big)^2.\]
	Now, for some universal constants $C,c>0$,
	\begin{align*}
	&
	\EE\Big( \prod_{l=j+1}^k  (1-c_l+g_l/\sqrt{l}) - 
	\prod_{l=j+1}^k  (1-c_l)\Big)^2
	= \prod_{l=j+1}^k \EE(1-c_l+g_l/\sqrt{l})^2- \prod_{l=j+1}^k (1-c_l)^2\\
	&=  \prod_{l=j+1}^k ((1-c_l)^2+c/l) - \prod_{l=j+1}^k (1-c_l)^2
	\leq C\frac{k-j}{k},
	\end{align*}
where we used the fact that $c_l =O(1/k_0)$ for $l\in [(1-\varepsilon) k_0, k_0]$ and therefore $\prod_{l=j+1}^k ((1-c_l)^2+c/l) =O(1)$.
	Hence we obtain by \eqref{eq-Euk1} and Lemma \ref{alpha} that 
	\[\sum_{j=(1-\varepsilon)k_0}^k \mathrm{Var}(u_j) \mathrm{Var}(w_{j,k})
	\lesssim \sum_{j=(1-\varepsilon)k_0}^k  e^{-c(k-j)\sqrt{k^*/k_0}} \frac{(k-j)}{n^2} \lesssim  \frac{1}{k_0k^*}. \]
	For the last term in \eqref{eq-Rkdec}, 
	we bound similarly
	\[ \sum_{j_1,j_2=(1-\varepsilon)k_0}^k \EE(u_{j_1}) \EE(u_{j_2}) \mathrm{Cov}(w_{j_1,k},w_{j_2,k})   \leq \Big(\sum_{j=(1-\varepsilon)k_0}^k \EE(u_{j}) \sqrt{\Var(w_{j,k})} \Big)^2  \lesssim \frac{1}{k_0{k^*}} . \]
	These bounds combined together give, after substitution in \eqref{eq-Rkdec},
	that for any $k \in [(1-\varepsilon)k_0, k_0-l_0]$, $\EE(R_k^2) \leq C/k_0k^*$.
	This gives
	\[ \sum_{k=(1-\varepsilon)k_0}^{k_0-l_0} \frac{(k_0k^*)^{1/2}}{
		(\log (i(k))+10)^2}\EE(R_k^2) \leq  \sum_{k=(1-\varepsilon)k_0}^{k_0-l_0} \frac{C}{(\log (i(k))+10)^2(k_0k^*)^{1/2}} 
	\underset{n\to +\infty}{\longrightarrow}  0. \]
\end{proof}
\begin{lemma}[Bound on $\Delta_k$]\label{Deterministic Delta}
	For any $C\geq 1$, and $\kappa$  large enough (depending on $C$), there exists a constant $\bar C$ 
	such that, on the event $\Omega_{C}\cap \Omega_n$,
	for all $k \in [(1-\varepsilon)k_0, k_0-l_0]$, 
	\begin{equation}\label{Bound}
	|\Delta_k| \leq \frac{\bar C (\log i(k)+10)^2}{k^*} 
	\end{equation}
\end{lemma}
\begin{proof}[Proof of Lemma \ref{Deterministic Delta}]
	We prove this result by a deterministic induction, working on the 
	event $\Omega_{C}\cap \mathcal{M}$, see 
	\eqref{eq-M_k} and \eqref{eq-OCG}, which implies in particular that
	\begin{equation}
	\label{eq-hk}
	|\bar \delta_k|\leq  2C \frac{\log i(k)+10}{(k_0 k^*)^{1/4}}=:2C h_k,\;
	\mbox{\rm for all $k\in [(1-\varepsilon) k_0,k_0-l_0]$}.
	\end{equation}
	The induction hypothesis on $\Omega_{C}$
	is the following: we fix a constant $\bar C$ (which will depend on $C$)
	so that, 
	for any $j \leq k-1$,
	\begin{equation}
	\label{eq-inductDk}
	|\Delta_j| \leq \frac{\bar C (\log i(j)+10)^2}{j^*} .
	\end{equation}
	Note that, with this choice, we have that
	\begin{equation}
	\label{eq-contDbyd}
	|\Delta_j|\leq  \frac{\bar C (\log i(j)+10)}{\kappa^{3/4}
		(k_0 j^*)^{1/4} }=\frac{\bar C}{\kappa^{3/4}} h_j.
	\end{equation}
	
	For $k=(1-\varepsilon)k_0$, the result is exactly the one of 
	Lemma \ref{apriori}, if we choose $\bar C>C_\Delta$.
	Fix $k \in[(1-\varepsilon)k_0,k_0-l_0]$ and assume that the bound \eqref{Bound} is valid up to $k-1$.
	On the event $\Omega_{C}$ where \eqref{eq-hk} holds,
	and together with the induction hypothesis,
	we have in particular that $|\delta_j|<1/2$ for any $j \leq k-1$, and therefore, as in 
	\eqref{eq-aprioriDk} in the proof of Lemma \ref{apriori}, we can write
	\begin{equation}
	\label{eq-DDk}
	\Delta_j = r_j +v_j \Delta_{j-1},\quad (1-\veps)k_0 < j\leq  k, 
      \end{equation}
where
\[ r_j=
-v_j  \bdelta_{j-1}^2 -	2 v_j  \bdelta_{j-1} \Delta_{j-1} - v_j \Delta_{j-1}^2 + v_j \frac{\delta_{j-1}^3}{1+\delta_{j-1}}. \]
On $\Omega_n$, we have by \eqref{bounvlOmega}  that $|v_j|\leq 1$. Thus, using \eqref{eq-contDbyd}, \eqref{eq-hk} and the fact that $|\delta_{j}|<1/2$ and
therefore $|\delta_j^3|/|1+\delta_j|\leq \delta_j^2$, 
we have on $\Omega_C\cap \Omega_n$ that $|r_j| \leq 2 (2C + \bar C/\kappa^{3/4})^2 h_{j-1}$ for any $j\leq k$.
Recall that $M_k = \prod_{j=(1-\veps)k_0}^k v_j$. We deduce from \eqref{eq-DDk} that, on $\Omega_C\cap \Omega_n$,
		\begin{equation}
		\label{eq-DD}
		|\Delta_k| \leq \sum_{j=(1-\varepsilon)k_0+1}^k |r_j \prod_{l=j+1}^k v_l|  + |\Delta_{(1-\varepsilon)k_0} M_{k}|  \leq 2\Big(2C+\frac{\bar C}{\kappa^{3/4}}\Big)^2 A_k + |\Delta_{(1-\varepsilon)k_0}M_k|,
\end{equation}
where $A_k = \sum_{j=(1-\varepsilon)k_0}^k h_{j-1}^2\prod_{l=j+1}^k |v_l|$.
	
	We show below 
	that there exists a universal constant $C_1>0$ (independent of $C$)
	such that
	\begin{equation}
	\label{eq-Ak2}
	A_k = \sum_{j = (1-\varepsilon)k_0}^k  h_{j-1}^2
	\prod_{l=j+1}^k |v_l| \leq \frac{C_1 (\log i(k)+10)^2}{k^*}.  
	\end{equation}
	On $\mathcal{M}$ (see \eqref{eq-M_k}) 
	and using Lemma \ref{apriori},
	 we have that
	\[|\Delta_{(1-\varepsilon)k_0}M_k| \leq
	 C_\Delta \frac{(\log n)^2}
	{n}e^{-c(\epsilon)
	(k-(1-\varepsilon)k_0)}\leq 
	  C_\Delta \frac{(\log i(k)+10)^2}{k^*}
	. \]
	Choosing now 
\begin{equation}
\label{eq-Cchoice}
\bar C>8 C_1\Big(C+\frac{\bar C}{2\kappa^{3/4}}\Big)^2 + C_\Delta
,
\end{equation}
	which is possible if $\kappa>\kappa_0(C)$, we conclude from
	\eqref{eq-DD} and \eqref{eq-Ak2} that \eqref{eq-inductDk}
	holds for $j=k$, thus completing the proof of the
	induction.
	
	It remains to prove \eqref{eq-Ak2}. By Lemma \ref{alpha},  on
	$\Omega_n$,
\[ A_k \lesssim \sum_{j=(1-\veps) k_0}^k \frac{(\log i(j) +10)^2}{(k_0 j^*)^{1/2}} e^{-(k-j) \sqrt{k^*/k_0}}.\]
We claim that for any $j\in [(1-\veps)k_0, k]$
\begin{equation}\label{claimdecre} \frac{\log i(j) +10}{{j^*}^{1/4}}\lesssim \frac{\log i(k)+10}{{k^*}^{1/4}}.\end{equation}
Indeed, for any $j \in[(1-\veps)k_0, k]$ we have $j \in[s_{i(j)+1}, s_{i(j)}]$ with $i(j) \geq i(k)$. Thus, 
\[  \frac{\log i(j) +10}{{j^*}^{1/4}}\leq  \frac{\log i(j) +10}{{s(j)^*}^{1/4}} \leq  \frac{\log i(k)+10}{{s_{i(k)}^*}^{1/4}},\]
where we used the fact that $s_i^* = i^{2/3} k_0^{1/3}$ and that $x \mapsto (\log x +10)/x^{1/6}$ is nonincreasing on $[1,+\infty)$. This ends the proof of the claim. Using \eqref{claimdecre}, we get that
\[ A_k \lesssim \frac{(\log i(k)+10)^2}{(k_0 k^*)^{1/2} } \sum_{j=(1-\veps) k_0}^k  e^{-(k-j) \sqrt{k^*/k_0} }\lesssim \frac{(\log i(k)+10)^2}{k^* }.\]
\end{proof}

\subsection{Proof of Proposition \ref{prop-scalar}}
We begin with a CLT for the sums of the $\delta_k$'s in the scalar regime.
Recall that $\gamma$ denotes the standard Gaussian law.
\begin{proposition} \label{mainhyperbolic}
The following convergence holds: for all $t\in \RR$,
\begin{align}
  \label{eq-finalprop1}
  &\lim_{\kappa \to \infty}\limsup_{n\to\infty}
  \PP\Big( \frac{\sum_{k=1}^{k_0-l_0}\delta_k   + \frac{v-1}{6}\log n}{\sqrt{\frac{v}{3} \log n } }  >  t  \Big) = 
  \lim_{\kappa \to \infty}\liminf_{n\to\infty}
  \PP\Big( \frac{\sum_{k=1}^{k_0-l_0}\delta_k   + \frac{v-1}{6}\log n}{\sqrt{\frac{v}{3} \log n } }  >  t  \Big)\nonumber\\
  &\qquad \qquad = 
  \gamma( (t,\infty)).
\end{align}
Furthermore, there exist constants $C_{\kappa}>0$ so that
\begin{equation}
  \label{eq-finalprop2}
  \lim_{\kappa \to \infty} \liminf_{n\to \infty} \PP\Big( |\delta_{k_0-l_0}|< \frac{C_{\kappa}}{k_0^{1/3}} \Big)=1 . 
\end{equation}
\end{proposition}

Before giving the proof of Proposition \ref{mainhyperbolic}, we 
derive an a-priori estimate on the difference $\Delta_k-\overline \Delta_k$.
\begin{lemma}
  \label{lem-DDk}Let $C>0$.
Choose $\bar C$ so that the conclusions of Lemma
  \ref{Deterministic Delta} hold.
   Then, there exists a constant $C_2=C_2(C,\varepsilon)$ so that,
   on the event $\Omega_n\cap \Omega_{C}$,
  \begin{equation}
    \label{eq-recursion3}
    \sum_{k=1}^{k_0-l_0} |\Delta_k -\overline \Delta_k|\leq C_2.
  \end{equation}
\end{lemma}
\begin{proof}[Proof of Lemma  \ref{lem-DDk}]
  Set $x_k=\Delta_k-\overline \Delta_k$ and note that, by 
  \eqref{Deltabar} and \eqref{eq-aprioriDk}, we have that
  $x_k$ satisfies the recursion
  \begin{equation}
    \label{eq-recursion1}
    x_{k+1}=v_{k+1} x_k+r_k, \quad x_0=0,
  \end{equation}
where $ r_k=-v_{k+1} \Delta_{k}^2-2v_{k+1} \bdelta_{k} \Delta_{k}
    +v_{k+1} \delta_{k}^3/(1+\delta_k)$. For any $k$, we have
\[ |x_k| \leq \sum_{j=0}^{k-1} |r_j|\prod_{l=j+1}^k |v_l|.\]
  On the event $\Omega_n\cap \Omega_{C}$ we have by combining
  Lemmas \ref{apriori},  \ref{Deterministic delta} and \ref{Deterministic Delta}, that
  for some $C_1=C_1(C,\varepsilon)$,
 \begin{equation}
    \label{eq-recursion4}
    |r_k|\leq C_1 \frac{(\log i(k)+10)^3}{ k_0^{1/4} {k^*}^{5/4}},
  \end{equation} where we extended the definition of $i(k)$ so that $i(k)=n$
  for $k<(1-\varepsilon)k_0$. (The term contributing the most is
  $\bar \delta_k \Delta_k$.) Recalling Lemma \ref{alpha}, 
  we obtain that on $\Omega_n\cap \Omega_C$,
  \begin{align*}
    \label{eq-recursion5}
    \sum_{k=1}^{k_0-l_0} \sum_{j=0}^{k-1} |r_j| \prod_{l=j+1}^{k} |v_l|
   &\lesssim \frac{1}{k_0^{1/4}} \sum_{k=1}^{k_0-l_0}\sum_{j=0}^{k-1}\frac{(\log i(j)+10)^3}{{j^*}^{5/4}} e^{-\frac{1}{2}(k-j)\sqrt{k^*/k_0}/2}\\
   &\lesssim \frac{1}{k_0^{1/4}}  \sum_{k=1}^{k_0-l_0}\frac{(\log i(k)+10)^3}{{k^*}^{5/4}} \sqrt{\frac{k_0}{k^*}},
\end{align*}
where we used the fact that, similarly to
\eqref{claimdecre}, $(\log i(j)+10){j^*}^{-5/4} \lesssim (\log i(k)+10){k^*}^{-5/4}$ for any $j\leq k$. As $i(k)^{2/3} \leq k^*/k_0^{1/3}$ and $l_0 = \kappa k_0^{1/3}$, we conclude by Riemann approximation that
\[    \sum_{k=1}^{k_0-l_0} \sum_{j=0}^{k-1} |r_j| \prod_{l=j+1}^{k} |v_l| \lesssim   k_0^{\frac{1}{4}} \sum_{l= l_0}^{+\infty} \frac{(\log(l/k_0^{1/3})+10)^3}{l^{7/4}}\lesssim 1.\]
\end{proof}

\begin{proof}[Proof of Proposition \ref{mainhyperbolic}]
  We first note that by combining Lemma \ref{easypart} with the fact that $\delta_k=\Delta_k+\bar\delta_k$ and Lemma \ref{lem-DDk}, we have
  \[ \frac{\sum_{k=1}^{(1-\varepsilon)k_0} \delta_k}{\sqrt{ \log n}} 
\underset{n\to+\infty}{\longrightarrow} 0, \]
in probability.
It thus suffices to show \eqref{eq-finalprop1} with the sum starting at
$k=(1-\varepsilon) k_0$.
By Propositions \ref{TCL delta bar} and \ref{Limit Delta bar}, we have that for any $\kappa >1$, 
\[ 	\forall t \in \RR,  \lim_{n \to +\infty} \PP\Big( \frac{1}{\sqrt{\frac{v}{3} \log n}} \big(\sum_{k=(1-\varepsilon)k_0}^{k_0-l_0}  \bdelta_k - \frac{1}{6} \log n \big)  >t\Big) = \PP( \mathcal{G} >t) \]
where $\mathcal{G}$ is a standard Gaussian random variable,  
and
	\[ \frac{1}{\sqrt{\log n}} \Big( \sum_{k=(1-\varepsilon)k_0}^{k_0-l_0} \overline{ \Delta_k} + \frac{v}{6} \log n \Big) \underset{n\to +\infty}{\longrightarrow} 0, \]
in probability.
	Combined with Lemma \ref{lem-DDk}, we conclude that for any $\delta>0$,
	\begin{equation}
	  \label{eq-endgame1}
 \limsup_{n\to+\infty}
	 \PP\Big(\Big| \frac{1}{\sqrt{\log n}} \Big( \sum_{k=(1-\varepsilon)k_0}^{k_0-l_0}  \Delta_k + \frac{v}{6} \log n\Big)\Big|>\delta \Big)=0.
       \end{equation}
       Since $\delta_k=\bdelta_k+\Delta_k$,  the claimed
       \eqref{eq-finalprop1} follows.

       Combining now Lemmas \ref{Deterministic delta} and \ref{Deterministic Delta} (with $k=k_0-l_0$) yields \eqref{eq-finalprop2}.
     \end{proof}

Before completing the proof of Proposition \ref{prop-scalar}, we need one more
a-priori estimate on the sums of squares of $\overline{\Delta_k}$.
Recall that we write $\EE_n(\cdot)=\EE(\cdot {\bf 1}_{\Omega_n})$.
\begin{lemma}\label{secondmomentD}
	There exists a constant $c>0$ such that 
	for $(1-\varepsilon)k_0 \leq k \leq k_0-l_0$,
	\[ \EE(\overline{\Delta}_k^2) \leq 
	 \frac{c}{\sqrt{k_0} {k^*}^{3/2}}.\]
\end{lemma}
\begin{proof}[Proof of Lemma \ref{secondmomentD}]
We start from
\[ \overline{\Delta_k} = \sum_{j=(1-\varepsilon)k_0}^k \big(-v_j \bdelta_{j-1}^2 +\EE(v_j) \EE(\bdelta_{j-1}^2)\big) \prod_{l={j+1}}^k v_l - \sum_{j=(1-\varepsilon)k_0}^k \EE(v_j) \EE(\bdelta_{j-1}^2) \prod_{l={j+1}}^k v_l .  \]
From Lemma \ref{lemma3}, we have 
that $\EE(\bdelta_j^4) \leq C/(k_0j^*)$ 
and $\EE(\bdelta_j^2) \leq C/\sqrt{k_0j^*} $. Using Lemma  \eqref{alpha}, 
we get that
\begin{align*}
  \sum_{j=(1-\varepsilon)k_0}^k \EE(\bdelta_{j-1}^4) \prod_{l=j}^k \EE(v_l^2) \lesssim \sum_{j=(1-\varepsilon)k_0}^k \frac{1}{k_0 j^*} e^{-(k-j)\sqrt{k^*/k_0}} \lesssim \frac{1}{\sqrt{k_0} {k^*}^{3/2}}
\end{align*}
and
\begin{align*}
&\EE \Big( \Big( \sum_{j=(1-\varepsilon)k_0}^k \EE(v_j) \EE(\bdelta_{j-1}^2) \prod_{l={j+1}}^k v_l \Big)^2 \Big)
\lesssim \!\!
\sum_{(1-\veps)k_0 \leq j_1< j_2 \leq k} \frac{1}{k_0\sqrt{j_1^*j_2^*}} \prod_{l_1=j_1+1}^{j_2} \EE(v_{l_1}) \prod_{l_1=j_2+1}^{k} \EE(v_{l_2}^2)  \\
& 
\qquad\qquad \lesssim \sum_{(1-\veps)k_0 \leq j_1< j_2 \leq k} 
\frac{1}{k_0\sqrt{j_1^*j_2^*}} 
e^{-(j_2-j_1)\sqrt{j_2^*/k_0}} e^{-(k-j_2)\sqrt{k^*/k_0}}\\
&
\qquad\qquad
\leq \sum_{(1-\varepsilon)k_0<j_2\leq k} \frac{C}{\sqrt{k_0j_2^3}} e^{-(k-j_2)\sqrt{k^*/k_0}}
 \leq \frac{C}{{k^*}^2}.
\end{align*}
\end{proof}

\begin{proof}[Proof of Proposition \ref{prop-scalar}] The claim 
  \eqref{eq-extra} is precisely \eqref{eq-finalprop2}. In view of 
  \eqref{eq-finalprop1}, to see
  \eqref{eq-scalarrec} it is enough to prove that
  the sequence
  $ (\log |\psi_{k_0-l_0}(z)| - \sum_{k=1}^{k_0-l_0} \delta_k)/\sqrt{\log n}$ converges in probability towards zero as $n$ and then
$\kappa$ go to infinity. 

By the definition of $\delta_k$, we have
\[ \log |\psi_{k_0-l_0}(z)| = \sum_{k=1}^{k_0-l_0} \log |\frac{\psi_k(z)}{\psi_{k-1}(z)}| = \sum_{k=1}^{k_0-l_0} \log |1+\delta_k|.   \]
Lemmas \ref{apriori}, \ref{Deterministic delta} and \ref{Deterministic Delta}
imply that with probability approaching $1$ as first $n\to\infty$ and then $\kappa\to\infty$, we have that $|\delta_k|<1/2$ for all $k\leq k_0-l_0$.
Hence we can write, on this event,
\[ \sum_{k=1}^{k_0-l_0} |\log|1+\delta_k| - \delta_k| \leq \sum_{k=1}^{k_0-l_0} \delta_k^2 \leq 2\sum_{k=1}^{k_0-l_0}( { \bdelta_{k}}^2 +  \Delta_k^2). \]
Using Lemma \ref{lem-DDk}, we can replace,
on the event $\Omega_n\cap \Omega_C$, the term
$\Delta_k$ in the last expression
by $\overline{\Delta_k}$, with bounded error. 
It now follows from Lemmas \ref{apriori}, \ref{part1second}, \ref{lemma3} and 
\ref{secondmomentD} that
$\sum_{k=1}^{k_0-l_0} \overline{ \delta_{k}}^2$ and $\sum_{k=1}^{k_0-l_0} \overline{ \Delta_k}^2$ are bounded in $L^1$ on $\Omega_n$, and therefore
 we get that 
\[ \frac{1}{\sqrt{\log n}} \sum_{k=1}^{k_0-l_0} (\bdelta_k^2 + \overline{ \Delta_k}^2) \xrightarrow[n \to \infty]{\PP} 0.\]
It follows that the CLT for $\sum_{k=1}^{k_0-l_0} \delta_k$ is inherited by 
$\log |\psi_{k_0-l_0}(z)| = 
 \sum_{k=1}^{k_0-l_0} \log |1+\delta_k|,$ which completes the proof of the proposition.   
\end{proof}

\appendix
\section{An anti-concentration inequality}
The L\'{e}vy anti-concentration function $Q_\mu(\cdot)$ of a probability measure $\mu\in\mathcal{P}(\RR)$ is defined as
\begin{equation}
\label{eq-anticonc}
Q_\mu(\veps) = \sup_{ x \in  \RR} \mu\big([x-\veps, x+\veps]\big), \qquad \veps>0.
\end{equation}
For a random variable $X\sim \mu$, we often write $Q_X$ for $Q_\mu$.
	\begin{lemma}\label{anticoncunif}
	 For $M\geq 1$, set
	  \begin{equation}
	    \label{eq-1}
\mathcal{C}_M=\left\{ \mu \in \mathcal{P}(\RR): 	  \int x d\mu=0, \quad  \int x^2 d\mu=1, \quad  \int |x|^3 d\mu \leq M\right\}.
	  \end{equation}
Then for  any $\veps\in(0,1/4)$ there exists $\delta=\delta(M,\veps)>0$ such that for all $\mu\in \mathcal{C}_M$,
\begin{equation}
  \label{eq-2}
  Q_\mu(\veps) \leq 1-\delta.
\end{equation}
\end{lemma}
\begin{proof}
Let $\veps>0$, and assume by contradiction that
there exists a sequence $\mu_n \in \mathcal{C}_M$ such that
$ Q_{\mu_n}(\veps) \underset{n\to \infty}{\longrightarrow} 1.$
As $\mathcal{C}_M$ is a compact set for the weak convergence, there exist a sequence
 $n_k$ and $\mu \in \mathcal{P}(\RR)$ such that
$ \mu_{n_k} \underset{k\to \infty} {\leadsto}\mu.$
Since $\int |x|^3 d\mu_{n_k} \leq M$, we deduce by Fatou's Lemma 
that $\mu\in \mathcal{C}_M$.
Observe that as $Q_{\mu_n}(\veps) \to 1$ when $n\to \infty$, we obtain using Markov's inequality that  for $n$ large enough,
$$ Q_{\mu_n}(\veps)  = \sup_{|x| \leq 2M+\veps}\mu_n\big( [x-\veps,x+\veps]\big).$$
Now, let $\phi_\veps$ be a non-negative Lipschitz function such that $\phi_{\veps} =1$ on $[-\veps,\veps]$ and $\phi_\veps =0$ on $[-2\veps,2\veps]^c$, and define for any probability measure $\mu$,
$$\tilde Q_{\mu}(\veps) = \sup_{|x| \leq 2M+\veps} \int \phi_\veps(x-y) d\mu.$$
By continuity, we deduce that for any $k$ there exists $|x_k|\leq 2M +\veps$ such that
$$ \tilde Q_{\mu_{n_k}}(\veps) = \int \phi_\veps(x_k-y) d\mu_{n_k}.$$
Since $ Q_{\mu_{n_k}}(\veps)\leq \tilde Q_{\mu_{n_k}}(\veps)$, we get that
$ \int \phi_\veps(x_k-y) d\mu_{n_k} \underset{k\to \infty}{\longrightarrow} 1.$
At the price of taking another sub-sequence we can assume that $x_k$ converges to some $x \in \RR$. We then obtain that
$ \int \phi_\veps(x-y) d\mu = 1,$
which yields that $\mu\big([x-2\veps,x+2\veps]\big) =1$. Since $\Var(\mu) =1$ and $(4\veps)^2 <1$, we get a contradiction.
\end{proof}

Combining Lemma \ref{anticoncunif} with the  Kolmogorov-Rogozin inequality 
\cite{Rogozin}, we get the following anti-concentration inequality for sums of independent random variables.
\begin{lemma}\label{anticonc}
Assume that $X_1,X_2,\ldots,X_m$ are independent centered random variables such that for some  $M>0$,  $\EE |X_i|^3 \leq M (\EE X_i^2)^{3/2}$ for $i=1,\ldots,m$. Let $S = X_1+X_2+\ldots+X_m$. Then there is a constant $C=C(M)$ so that
for any $t \geq \max_i \sqrt{\EE(X_i^2)}/8$,
$$ Q_S(t) \leq \frac{Ct}{\sqrt{ \sum_{i=1}^m \EE X_i^2}}.$$
\end{lemma}
\begin{proof}
Note that the statement and its assumptions are  scale-invariant, so we may and will assume that $\max_{i=1}^m \EE X_i^2=1$. 
By the Kolmogorov-Rogozin inequality \cite{Rogozin}, we have that for any constants $s_i\leq t$,
$$ Q_S(t) \leq \frac{Ct}{\sqrt{\sum_{i=1}^m s_i^2 (1-Q_{X_i}(s_i))}}.$$
Let $s_i =\sqrt{\EE(X_i^2)}/8$, $i=1,\ldots,m$.
By Lemma \ref{anticoncunif}, we have that $Q_{X_i}(s_i)\leq 1-\delta$,
where $\delta >0$ only depends on $M$. This completes the proof.
\end{proof}

\end{document}